\documentclass[11pt,reqno]{amsart}  

\pdfoutput=1

\usepackage{amscd,amsmath,amssymb,mathtools}
\usepackage{verbatim}
\usepackage{cite}
\usepackage{color}
\usepackage{hyperref}
\usepackage{url}
\usepackage{enumerate,enumitem}
\usepackage{graphicx}
\usepackage{epstopdf,epsfig,subfigure}
\usepackage{curve2e}
\usepackage[inner=1.1in,outer=1in,top=1.2in,bottom=1.05in]{geometry}

\theoremstyle{plain}
\newtheorem{theorem}{Theorem}[section]
\newtheorem{proposition}[theorem]{Proposition}
\newtheorem{lemma}[theorem]{Lemma}
\newtheorem{corollary}[theorem]{Corollary}

\theoremstyle{definition}
\newtheorem{remark}[theorem]{Remark}

\theoremstyle{definition}
\newtheorem{definition}[theorem]{Definition}
\newtheorem{assumption}[theorem]{Assumption}

\numberwithin{equation}{section}

\newcommand{\linspan}{\mathop{\rm span}\nolimits}

\newcommand{\rest}{\left.\kern-2\nulldelimiterspace\right|_}
\newcommand{\norm}[2]{\left|#1\right|_{#2}}
\newcommand{\dnorm}[2]{\left\|#1\right\|_{#2}}
\newcommand{\fractx}[2]{{\textstyle\frac{#1}{#2}}}

\newcommand{\Id}{{\mathbf1}}

\newcommand{\ex}{\mathrm{e}}
\newcommand{\p}{\partial}

\newcommand{\e}{\varepsilon}
\newcommand{\ed}{\mathrm d}

\newcommand*{\Bigcdot}{\raisebox{-.25ex}{\scalebox{1.25}{$\cdot$}}}

\newcommand{\N}{{\mathbb N}}
\newcommand{\R}{{\mathbb R}}

\newcommand{\M}{{\mathbb M}}

\newcommand{\D}{{\mathrm D}}

\newcommand{\KKK}{{\mathbf K}}

\newcommand{\nnn}{\mathbf n}

\newcommand{\CC}{{\mathcal C}}
\newcommand{\DD}{{\mathcal D}}

\newcommand{\FF}{{\mathcal F}}
\newcommand{\GG}{{\mathcal G}}

\newcommand{\KK}{{\mathcal K}}
\newcommand{\LL}{{\mathcal L}}
\newcommand{\MM}{{\mathcal M}}
\newcommand{\NN}{{\mathcal N}}

\newcommand{\ZZ}{{\mathcal Z}}

\newcommand{\Ma}{{\mathbf M}}
\newcommand{\St}{{\mathbf S}}

\newcommand{\ovlineC}[1]{\overline C_{\left[#1\right]}}

\begin{document}
\title{Semiglobal exponential stabilization of nonautonomous semilinear parabolic-like systems}

\author{S\'ergio S.~Rodrigues}

\address{Johann Radon Institute for Computational and Applied Mathematics,
\"OAW, 
Altenbergerstra{\normalfont\ss}e 69, 4040 Linz, Austria.}
%

\email{sergio.rodrigues@ricam.oeaw.ac.at}




\date{}


\begin{abstract}
It is shown that an explicit oblique projection nonlinear feedback controller is able
to stabilize semilinear parabolic equations, with
time-dependent dynamics and with a polynomial nonlinearity. 
The actuators are typically modeled by a finite number of indicator functions of small subdomains.
No constraint is imposed on the sign of the polynomial nonlinearity. The norm
of the initial condition can be arbitrarily large, and
the total volume covered
by the actuators can be arbitrarily small. The number of actuators depend
on the operator norm of the oblique projection, on the polynomial degree of the
nonlinearity, on the norm of the initial condition, and on the total volume
covered by the actuators.
The range of the feedback controller coincides with the range of the oblique projection,
which is the linear span of the actuators.
The oblique projection is performed along the orthogonal complement of a subspace spanned by a suitable finite
number of eigenfunctions of the diffusion operator.
For rectangular domains, it is possible to explicitly construct/place the actuators so that the
stability of the closed-loop system is guaranteed.
Simulations are presented, which
show the semiglobal stabilizing performance of the nonlinear feedback.
\end{abstract}

\maketitle

\pagestyle{myheadings}
\thispagestyle{plain}
\markboth{\sc S.~S.~Rodrigues}
{\sc Semiglobal stabilization of parabolic systems}

\section{Introduction}
Nonlinear parabolic equations appear in many models of real world evolution processes. Therefore, the study of
such equations is important
for real world applications. In particular, it is of interest to know whether it is possible to drive the evolution to
a given desired behavior or whether it is possible to stabilize such evolution process, by means of suitable controls.
The simplest model involving parabolic equations is the heat equation, modeling the evolution of the temperature in a
room~\cite[Chapitre~II]{Fourier1822}. 
Parabolic equations also appear in models for population dynamics~\cite{ChenJungel06,AnitaLanglais09}, traffic dynamics~\cite{NagataniEmmNak98},
and electrophysiology~\cite{NagumoAriYosh62}.

Usually, controlled parabolic equations can be written as a nonautonomous evolutionary system in the abstract form
\begin{equation}\label{sys-y}
      \dot y + Ay +A_{\rm rc}(t)y+\NN(t,y)-\sum_{i=1}^M u_i(t)\Psi_i=0,\quad    y(0)=y_0,
\end{equation}
where~$y$ is the  state, $y_0$ and~$\Psi_i$, $i\in\{1,\,2,\,\dots,\,M\}$, are given in a  Hilbert space~$H$,
and~$u(t)=(u_1,\dots,u_M)(t)$ is a control function at our disposal, taking values in~$\R^M$.
The linear operator~$A$ is a diffusion-like operator and the linear operator~$A_{\rm rc}$ is a time-dependent
reaction-convection-like operator. The operator
$\NN$ is a time-dependent nonlinear operator. The general properties
asked for~$A$, $A_{\rm rc}$, and~$\NN$ will be precised later on.

In the linear case, $\NN=0$, 
is has been proven in~\cite{KunRod18-cocv} that the closed-loop system
\begin{equation}\label{sys_FeedKy}
      \dot y + Ay +A_{\rm rc}(t)y-\KK_{U_M}^{\FF,\M}(t,y)=0,\quad    y(0)=y_0\in H,
\end{equation}
is {\em globally} exponentially stable, with the feedback control operator
\begin{equation}\label{FeedKy}
 y\mapsto\KK_{U_M}^{\FF,\M}(t,y)\coloneqq P_{U_M}^{E_\M^\perp}\left(Ay +A_{\rm rc}(t)y-\FF(y)\right),
\end{equation}
where
\begin{equation}\label{FF=lam1}
 \FF(y)=\lambda\Id y,
\end{equation}
provided the condition
\begin{align}
 \overline\mu_M&\coloneqq\alpha_{M+1}-\left(6+4\norm{P_{{U_M}}^{E_\M^{\perp}}}{\LL(H)}^2\right)
 \norm{A_{\rm rc}}{L^\infty((0,+\infty),\LL(H,V'))}^2
 >0\label{suffalpha.lin}
 \end{align}
holds true.
In~\eqref{FeedKy} and~\eqref{FF=lam1}, $\Id$ is the identity operator, ~$\lambda>0$ 
is an arbitrary constant, and~$P_{U_M}^{E_\M^\perp}$ stands for the oblique projection
in~$H$ onto the closed subspace~$U_M$ along
the closed subspace~$E_\M^\perp$. Where~$U_M\coloneqq\linspan\{\Psi_i\mid i\in\{1,\,2,\,\dots,\,M\}\}$ 
is the linear span of our~$M$ linearly
independent actuators
and~$E_\M\coloneqq\linspan\{e_i\mid i\in\M\}$, with~$\M=\{1,\,2,\,\dots,\,M\}$, is the linear span of
``the'' first~$M$ linearly independent eigenfunctions of the diffusion operator~$A\colon\D(A)\to H$,
with domain $\D(A)\xhookrightarrow{\rm d,c} H$.
Further, $\alpha_{M+1}$ is the~$(M+1)$st eigenvalue of~$A$. 
The eigenvalues of~$A$, denoted by~$\alpha_i$, are supposed to satisfy
\[
 Ae_i=\alpha_ie_i,\qquad 0<\alpha_1\le\alpha_2\le\alpha_3\le\dots,\qquad \lim_{i\to+\infty}\alpha_i=+\infty.
\]

\begin{remark}
 Note that~$\KK_{U_M}^{\FF,\M}(t,y)=\sum_{i=1}^Mu_i(t)\Psi_i$ for suitable $u_i(t)\in\R$.
\end{remark}

It is not difficult to see that we can follow the arguments in~\cite[Thms.~3.5, 3.6,
and Rem.~3.8]{KunRod18-cocv} to conclude that
system~\eqref{sys_FeedKy} is still stable if we replace~\eqref{FF=lam1} by
\begin{equation*}
 \FF(y)=Ay+\lambda\Id y.
\end{equation*}

Observe that~\eqref{suffalpha.lin} concerns a single~$M\in\N$ and a single pair~$(U_M,E_\M)$.
The following result, which follows straightforwardly from the sufficiency of~\eqref{suffalpha.lin},
concerns a sequence of pairs~$(U_M,E_\M)_{M\in\N}$.
\begin{theorem}\label{T:L.intro}
 Assume that we can construct a sequence~$(U_M,E_\M)_{M\in\N}$ such that $\norm{P_{{U_M}}^{E_\M^{\perp}}}{\LL(H)}\le C_P$
 remains bounded, with~$C_P>0$
 independent of~$M$. Then system~\eqref{sys_FeedKy} is globally exponentially stable for large enough~$M$, 
 with~$\FF(y)\in\{\lambda \Id y,Ay+\lambda \Id y\}$.
\end{theorem}

Our main goal is to prove that an analogous explicit feedback allow us to {\em semiglobally} stabilize 
nonlinear systems as~\eqref{sys-y},
for a suitable class of nonlinearities. We underline that we shall not
assume any condition on the sign of the nonlinearity~$\NN$,
which means that the uncontrolled solution may blow up in finite time. For results concerning blow up
of solutions, see~\cite{Ball77,MerleZaag98,Levine73}. In particular,
this means that we will have to guarantee that
the controlled solution does not blow up, which is a nontrivial task/problem. 
This is a problem we do not meet when
dealing with linear systems, because solutions of linear systems do not blow up in finite time. 

In the linear case the number~$M$ of actuators that allow us to stabilize the system does not depend on the initial
condition, while in the nonlinear case it does. We shall prove that~$M$ 
depends only on a suitable norm of the initial condition,
this dependence is what motivates the terminology ``{\em semiglobal} stability'' we use throughout the paper.

For nonlinear systems, previous results on the related literature are 
concerned with {\em local} stabilization, and such results are often derived through a suitable
nontrivial fixed point argument. In such situation the feedback operator
is linear and is such that it {\em globally} stabilizes  the linearized system, with~$\NN=0$.
In general, such linearization based feedback will be able to stabilize the
nonlinear system only if the initial condition is {\em small enough}, in a suitable norm.
Here, in order to cover {\em arbitrary large} initial conditions, and thus obtain the semiglobal
stabilization result for~\eqref{sys-y}, we will use a nonlinear feedback operator.
Instead of starting by constructing a feedback stabilizing the linearized system, we deal
directly with the nonlinear system.

\subsection{The main result}
We show that, for a suitable Hilbert space
~$V\xhookrightarrow{\rm d,c}H$, and for an arbitrary given~$R>0$,
system~\eqref{sys-y}
\begin{subequations}\label{sys_FeedKy-NF}
 \begin{equation}\label{sys_FeedKy-N}
      \dot y + Ay +A_{\rm rc}(t)y+\NN(t,y)-\KK_{U_M}^{\FF,\M,\NN}(t,y)=0,\quad    y(0)=y_0,
\end{equation}
with the feedback
\begin{equation}\label{FeedKy-N}
 y\mapsto\KK_{U_M}^{\FF,\M,\NN}(t,y)\coloneqq P_{U_M}^{E_\M^\perp}\bigl(Ay +A_{\rm rc}(t)y+\NN(t,y)-\FF(y)\bigr)
\end{equation}
\end{subequations}
is stable, provided the initial condition is in the ball~$\{v\in V\mid\norm{v}{V}<R\}$
and the pair~$(U_M,E_\M)$ satisfies  a suitable ``nonlinear version'' of~\eqref{suffalpha.lin}.
The number~$M$ of actuators needed
to stabilize the system will (or may) increase with~$R$. 
A precise statement of the main stability result concerning a single pair~$(U_M,E_\M)$, together a ``nonlinear version'' of the sufficient
stability condition~\eqref{suffalpha.lin}
is given hereafter, once we have introduced some notation and terminology.
A consequence of that result will be the following ``nonlinear version'' of Theorem~\ref{T:L.intro}.
\begin{theorem}\label{T:N.intro}
Assume that we can construct a sequence~$(U_M,E_\M)_{M\in\N}$ such that $\norm{P_{{U_M}}^{E_\M^{\perp}}}{\LL(H)}\le C_P$
 remains bounded, with~$C_P>0$
 independent of~$M$. Then, with $\FF(y)=Ay+\lambda \Id y$, system~\eqref{sys_FeedKy-NF}
 is  exponentially stable, for large enough~$M$ depending on~$\norm{y_0}{V}$.
\end{theorem}

The operator choice~$\FF(y)=\lambda \Id y$, used in previous works for linear systems,
will not necessarily satisfy the assumptions hereafter (Assumption~\ref{A:FFM}, in particular). That is, we cannot
conclude/guarantee (from our results) that such choice will {\em semiglobally} stabilize the nonlinear system. 
To better understand the differences between the two choices, we will consider a general
operator~$\FF(y)=\FF_\M(P_{E_\M}y)$ depending only on the orthogonal projection~$P_{E_\M}y$ of the state~$y$ in~$H$
onto~$E_{\M}$. Further $\FF_\M\colon E_\M\to E_\M$ is a continuous operator. Notice that, with~$\varsigma\in\{0,1\}$ 
we have that $P_{E_\M}(\varsigma A+\lambda\Id) P_{E_\M}$ is continuous, and the feedback in~\eqref{FeedKy-N}
satisfies~$\KK_{U_M}^{\varsigma A+\lambda\Id,\M,\NN}=\KK_{U_M}^{P_{E_\M}(\varsigma A+\lambda\Id) P_{E_\M},\M,\NN}$,
because ~$P_{E_\M}$ commutes with both~$A$ and~$\Id$
and because
$P_{U_M}^{E_\M^\perp}P_{E_\M}=P_{U_M}^{E_\M^\perp}$.
Notice also that when~$\FF$ is linear and~$\NN\ne0$, then~$y\mapsto \KK_{U_M}^{\FF,\M}(t,y)$ is linear, while
$y\mapsto \KK_{U_M}^{\FF,\M,\NN}(t,y)$ is nonlinear.

\subsection{Motivation and short comparison to previous works} We find systems in form~\eqref{sys-y} when, for example,
we want to stabilize a system
to a trajectory~$\hat z$. That is,
suppose~$\hat z$ solves the nonlinear system
\[
 \dot{\hat z}+A{\hat z}+f({\hat z})=0,\qquad \hat z(0)=\hat z_0,
\]
and that~$\hat z$ has suitable desired properties (e.g., it is essentially bounded and regular).
In many situations, it may happen that the solution issued from a different initial condition~$z_0$ 
may present a nondesired behavior
(e.g., not remaining bounded, or even blowing up in finite time). In such situation,
we would like to find a control $\mathbf u(t)=\sum_{i=1}^M u_i(t)\Psi_i$, such that the solution
of
\begin{equation}\label{sys-zeta}
  \dot z+Az+f(z)+\mathbf u=0,\qquad z(0)=z_0,
\end{equation}
approaches the desired behavior $\hat z$. More precisely, we would like to have
\begin{equation}\label{goalz}
 \norm{z(t)-\hat z(t)}{\mathfrak H}\le C\ex^{-\mu t}\norm{z(0)-\hat z(0)}{\mathfrak H}, 
\end{equation}
for some normed space~$\mathfrak H$. Now we observe that the difference~$y\coloneqq z-\hat z$
satisfies a dynamics as~\eqref{sys-y}, because
from Taylor expansion (for regular enough~$f$)
we may write $f(z)-f(\hat z)\eqqcolon A_{\rm r,c}(t)y+\NN(t,y)$, with~$A_{\rm r,c}(t)=\frac{\ed}{\ed z}f(\hat z)$ and with a
remainder~$\NN(t,y)$. Notice that $\NN$ vanishes if, and only if,~$f$
is affine, otherwise~$\NN(t,y)$ is nonlinear.
Therefore, stabilizing~\eqref{sys-zeta} to the targeted trajectory,
is equivalent to stabilizing system~\eqref{sys-y} (to zero), because
\eqref{goalz} reads~$\norm{y(t)}{\mathfrak H}\le C\ex^{-\mu t}\norm{y(0)}{\mathfrak H}$.

In previous works on internal stabilization of nonautonomous parabolic-like systems
including~\cite{BarRodShi11,KroRod15,KroRod-ecc15,BreKunRod17,PhanRod18-mcss}, the exact null controllability of the
corresponding linearized systems
(by means of infinite dimensional controls, see~\cite{CoronNguyen17,DuyckZhangZua08,FerGonGuePuel06,FerGueImaPuel04,FurIma99,Ima01,Yama09})
played a key role in the proof of the existence of a stabilizing control.
See also~\cite{Ammari-Duyck-Shi16} for the weakly damped wave equation. We would like to underline that for
the proof of the stability of an oblique projection based closed-loop system, we do not 
need to assume the above null controllability result.

Our results are also true for the particular case of autonomous systems, which has been extensively studied. However, in such case
other tools may be, and have been, used. Among such tools we have the spectral properties of the  system operator~$A +A_{\rm rc}$. We refer to the
works~\cite{Nambu82,RaymThev10,BadTakah11,Barbu12,BarbuLasTri06,BarbuTri04,Barbu_TAC13,HalanayMureaSafta13,Munteanu17,ChowdhuryErvedoza19} and references therein.
See also the comments in~\cite[Sect.~6.5]{KunRod18-cocv}.
Finally we refer to the examples in~\cite{Wu74}, showing that in the nonautonomous case, the spectral
properties of~$A +A_{\rm rc}(t)$, at each time~$t\ge0$, are 
not appropriate for studying the stability
of the corresponding nonautonomous system.

Though we do not deal here with boundary controls,
we refer to~\cite{NgomSeneLeRoux15,Raymond07,Rod18} for works on the stabilization of the
Navier--Stokes equation, evolving in a bounded domain~$\Omega\subset\R^3$, to a targeted trajectory. In~\cite{NgomSeneLeRoux15,Raymond07} the targeted trajectory
is independent of time (autonomous case), while in~\cite{Rod18} it is time-dependent (nonautonomous case).
In~\cite{NgomSeneLeRoux15} the {\em global} stability of the closed-loop is shown to hold in~$L^2$-norm for
{\em at least one } (not necessarily unique) appropriately defined ``weak'' solution.
In~\cite{Raymond07} the {\em local} stability of 
the closed-loop system has been shown to hold in the Sobolev~$W^{s,2}$-norm, with~$s\in(\frac{1}{2},1]$, and
the solutions of the closed-loop
system are more regular and unique. In~\cite{Rod18} the {\em local} stability of 
the closed-loop system has been shown to hold in the~$W^{1,2}$-norm and the solutions are unique. Recall that
$L^2=W^{0,2}\supset W^{s_1,2}\supset W^{s_2,2}$, for~$0<s_1<s_2$.

Our results can be used to conclude the {\em semiglobal} stability of nonautonomous oblique projection based closed-loop
parabolic-like systems with internal controls,
where {\em semiglobal} stability lies
between {\em local} and {\em global} 
stability. The stability of the closed-loop system is shown to hold in the~$W^{1,2}$-norm, and the solutions are unique.
In previous results concerning {\em local} stability of parabolic systems, the control domain~$\omega$
can be arbitrary and fixed a priori. For our
results the volume of the support of the actuators can still be arbitrarily small and fixed a priori,
but the support itself is not fixed a priori. See Section~\ref{sS:exRect}.

Finally, though we consider here the case of parabolic-like systems and are particularly interested
in the case where blow up may occur for the free dynamics and on the case our control is finite dimensional, 
the stabilization problem is still an interesting problem for other types of evolution equations, where blow up does not occur,
like those conserving the energy and/or other quantities. For stabilization results
(by means of infinite-dimensional control) for nonparabolic-like systems we refer the reader
to~\cite{RussellZhang96,AzmiKunisch18,LaurentLinaresRosier15,Russell78} and references therein.

\subsection{Computational advantage}
We underline that the feedback operators in~\eqref{FeedKy} and~\eqref{FeedKy-N} are explicit and the essential step in their
practical realization involves the computation of the oblique projection. A classical approach to find a feedback stabilizing control
is to compute the solution of the Hamilton--Jacobi--Bellman equation, which is known to be a difficult numerical task,
being related with the so-called ``curse of dimensionality'', for example see the recent paper~\cite{KaliseKunisch18}
(for the autonomous case), where the authors, in order to compute the Hamilton--Jacobi--Bellman feedback, need to
approximate a parabolic equation by a 14-dimensional
ordinary differential equation (previous works deal with even lower-dimensional approximations).
This also means that standard discretization
methods as finite elements approximations are not appropriate for computing the Hamilton--Jacobi--Bellman solution, because a
14-dimensional finite elements approximation of a parabolic equation is hardly accurate enough.
In the linear case (and with quadratic cost) the Hamilton--Jacobi--Bellman feedback reduces
to the (algebraic) Riccati feedback. In this case finite elements approximations can be used, but the computational
effort increases considerably as we increase the number of
degrees of freedom.
For parabolic systems, the computation of the feedback in~\eqref{FeedKy} and in~\eqref{FeedKy-N} is considerably
cheaper, because the numerical computation of the oblique projection~$P_{U_M}^{E_\M^\perp}$ amounts to the
computation of the $M$~eigenfunctions~$\{e_i\mid i\in\M\}$, and the computation of the inverse of the
matrix~$\Theta_\M=\begin{bmatrix}
                   (E_\M,U_M)_{H}
                  \end{bmatrix}=\begin{bmatrix}
                   (e_i,\Phi_j)_{L^2}
                  \end{bmatrix}\in\R^{M\times M}
$, see~\cite{RodSturm18}. Note that the size of~$\Theta_\M$ is defined by the number~$M$ of actuators, 
and thus it is independent of
the number of degrees of freedom of the space discretization, that is, computing $\Theta_\M^{-1}$ does 
not become a harder task as we
refine our discretization.

Even in case we are able to compute an approximation of an Hamilton--Jacobi--Bellman based feedback control, such 
(approximated) feedback may not
guarantee stabilization for arbitrary initial conditions, as reported in~\cite[Sect.~5.2, Test~2]{KaliseKunisch18},
though we likely obtain a
neighborhood of attraction larger than that
of the Riccati closed-loop system.

Finally, the main idea behind solving the Riccati or
Hamilton--Jacobi--Bellman equations is that of finding a feedback (closed-loop) stabilizing control
or an optimal control, under the assumption/knowledge that a stabilizing (open-loop) control does exist.
Instead, in this paper,
the proof of existence of such a stabilizing control is included in the results.

\subsection{Contents and general notation}
The rest of the paper is organized as follows. In  Section~\ref{S:prelim} we recall suitable properties of oblique
projections, present an example of application of our results, and recall previous
global and local exponential stability results, which are related to the problem we address in this manuscript. 
In  Section~\ref{S:assump} we introduce the general properties asked for the
operators~$A$, $A_{\rm rc}$, and~$\NN$ in~\eqref{sys-y}, and also the properties asked for the triple~$(U_M,E_\M,\FF)$ defining
the feedback operator.
In  Section~\ref{S:feedback} we prove our main result. 
In  Section~\ref{S:examples} we show that our results can be applied to the stabilization of semilinear
parabolic equations with polynomial nonlinearities. 
In  Section~\ref{S:simul} we present the results of numerical simulations showing the
performance of the proposed nonlinear feedback. Finally,
the appendix gathers proofs of auxiliary results used in the main text.

\medskip

Concerning the notation, we write~$\R$ and~$\mathbb N$ for the sets of real numbers and nonnegative
integers, respectively, and we define $\R_r\coloneqq(r,\,+\infty)$ and~$\overline{\R_r}\coloneqq[r,\,+\infty)$, for $r\in\R$, and $\mathbb
N_0\coloneqq\mathbb N\setminus\{0\}$.

For an open interval $I\subseteq\R$ and two Banach spaces~$X,\,Y$, we write
$W(I,\,X,\,Y)\coloneqq\{y\in L^2(I,\,X)\mid \dot y\in L^2(I,\,Y)\}$,
where~$\dot y\coloneqq\frac{\ed}{\ed t}y$ is taken in the sense of
distributions. This space is endowed with the natural norm
$|y|_{W(I,\,X,\,Y)}\coloneqq\bigl(|y|_{L^2(I,\,X)}^2+|\dot y|_{L^2(I,\,Y)}^2\bigr)^{1/2}$.
In the case~$X=Y$ we write $H^1(I,\,X)\coloneqq W(I,\,X,\,X)$.

If the inclusions $X\subseteq Z$ and~$Y\subseteq Z$ are continuous, where~$Z$ is a Hausdorff topological space,
then we can define the Banach spaces $X\times Y$, $X\cap Y$, and $X+Y$,
endowed with the norms defined as
$|(a,\,b)|_{X\times Y}:=\bigl(|a|_{X}^2+|b|_{Y}^2\bigr)^{\frac{1}{2}}$,
$|a|_{X\cap Y}:=|(a,\,a)|_{X\times Y}$, and
$|a|_{X+Y}:=\inf_{(a^X,\,a^Y)\in X\times Y}\bigl\{|(a^X,\,a^Y)|_{X\times Y}\mid a=a^X+a^Y\bigr\}$,
respectively.
In case we know that $X\cap Y=\{0\}$, we say that $X+Y$ is a direct sum and we write $X\oplus Y$ instead.

If the inclusion
$X\subseteq Y$ is continuous, we write $X\xhookrightarrow{} Y$. We write
$X\xhookrightarrow{\rm d} Y$, respectively $X\xhookrightarrow{\rm c} Y$, if the inclusion is also dense, respectively compact.

The space of continuous linear mappings from~$X$ into~$Y$ will be denoted by~$\LL(X,Y)$. In case~$X=Y$ we 
write~$\LL(X)\coloneqq\LL(X,X)$.
The continuous dual of~$X$ is denoted~$X'\coloneqq\LL(X,\R)$.

The space of continuous functions from~$X$ into~$Y$ is denoted~$C(X,Y)$.
We consider the subspace of
increasing continuous functions, defined in~$\overline{\R_0}$ and vanishing at~$0$:
\begin{equation*}
 C_{0,\rm i}(\overline{\R_0},\R) \coloneqq \{\mathfrak n\mid \mathfrak n\in C(\overline{\R_0},\R),
 \quad \mathfrak n(0)=0,\quad\mbox{and}\quad
 \mathfrak n(\varkappa_2)\ge\mathfrak n(\varkappa_1)\;\mbox{ if }\; \varkappa_2\ge \varkappa_1\ge0\}.
\end{equation*}
Next,  we denote by~$\CC_{\rm b, i}(X, Y)$ the vector subspace 
\begin{equation*}
 \CC_{\rm b, i}(X, Y)\coloneqq
 \left\{f\in C(X,Y) \mid \exists\mathfrak n\in C_{0,\rm i}(\overline{\R_0},\R)\;\forall x\in X:\;
\norm{f(x)}{Y}\le \mathfrak n (\norm{x}{X})
\right\}.
\end{equation*} 

Given a subset~$S\subset H$ of a Hilbert space~$H$, with scalar product~$(\Bigcdot,\Bigcdot)_H$, the orthogonal complement of~$S$ is
denoted~$S^\perp\coloneqq\{h\in H\mid (h,s)_H=0\mbox{ for all }s\in S\}$.

Given a sequence~$(a_j)_{j\in\{1,2,\dots,n\}}$ of real constants, $n\in\N_0$, $a_i\ge0$, we
denote~$\|a\|\coloneqq\max\limits_{1\le j\le n} a_j$. Further, by
$\overline C_{\left[a_1,\dots,a_n\right]}$ we denote a nonnegative function that
increases in each of its nonnegative arguments.

Finally, $C,\,C_i$, $i=0,\,1,\,\dots$, stand for unessential positive constants.

\section{Preliminaries}\label{S:prelim}
We introduce/recall here specific notation and terminology concerning oblique projections and stability.

\subsection{Actuators and eigenfunctions}

In the stability condition~\eqref{suffalpha.lin}, as we increase~$M\in\N_0$ we have sequences of subspaces
\begin{align}\label{seqM-11}
 E_{\{1\}}, \; E_{\{1,2\}}, \; E_{\{1,2,3\}}, \;...\qquad\mbox{and}\qquad U_{1}, \; U_{2}, \; U_{3},\;...
 \end{align}
where the $M$th term of each sequence is an~$M$-dimensional space, $\dim E_{\M}=M=\dim U_{M}$.

Motivated by the results in~\cite[Sect.~4.8]{KunRod18-cocv} (see also~\cite[Rem.~3.9]{KunRod18-cocv}), in order to prove the
boundedness of the norm~$\norm{P_{{U_M}}^{E_\M^{\perp}}}{\LL(H)}\le C_P$, uniformly on~$M$, it may be
convenient to consider different sequences.

To simplify the exposition, we denote by~$\#Z\in\N$ the number of elements of a given
finite set~$Z\subseteq Y$. See~\cite[Sect.~13]{Halmos74}. For~$N\in\N_0$, ~$\#Z=N$ simply means that there exists a
one-to-one correspondence from $\{1,2,\dots,N\}$ onto $Z$. Of course~$\#Z=0$ means that $Z=\emptyset$, the empty set.
We also denote the collection 
\begin{equation*}
 \mathfrak P_N(Y)=\{Z\subseteq Y\mid \#Z=N\}.
\end{equation*}

Now, instead of~\eqref{seqM-11}, we consider a more general sequence as follows
\begin{align}
 E_{\{\sigma^1_1,\sigma^1_2,\dots,\sigma^1_{|\sigma^{1}|}\}}, \;
 E_{\{\sigma^2_1,\sigma^2_2,\dots,\sigma^2_{|\sigma^{2}|}\}}, \;
 E_{\{\sigma^3_1,\sigma^3_2,\dots,\sigma^3_{|\sigma^{3}|}\}}, \; ...,
 \quad\mbox{and}\quad U_{|\sigma^{1}|}, \; U_{|\sigma^{2}|}, \; U_{|\sigma^{3}|},\;...\notag
 \end{align}
 that is, denoting~$\M_\sigma\coloneqq \{\sigma^M_1,\sigma^M_2,\dots,\sigma^M_{|\sigma^{M}|}\}$, we have~$\#\M_\sigma=|\sigma^{M}|$ and the sequences
 \begin{align}\label{seqM-rs}
 E_{\M_\sigma}\coloneqq\linspan\{e_i\mid i\in\M_\sigma\} 
 \qquad\mbox{and}\qquad U_{\#\M_\sigma},\qquad\mbox{with}\qquad M\in\N_0,
 \end{align}
where for each $M\in\N_0$, the $M$th term of each sequence is a~$\#\M_\sigma$-dimensional space,
$\dim E_{\M_\sigma}=\#\M_\sigma=\dim U_{\#\M_\sigma}$,
and the function~$\sigma^M\colon\{1,2,\dots,\,\#\M_\sigma\}\to\M_\sigma\in\mathfrak P_{\#\M_\sigma}(\N_0)$, $i\mapsto\sigma^M_i$, is a bijection.

For a given~$M\in\N_0$, we will also need to underline two particular eigenvalues 
defined as
\begin{equation}\label{VDAalphaM}
 \alpha_{\M_{\sigma}}\coloneqq \max\{\alpha_i\mid i\in \M_{\sigma}\}, \qquad \alpha_{\M_{\sigma+}}\coloneqq \min\{\alpha_i\mid i\notin \M_{\sigma}\}. 
\end{equation}

Notice that, the sequence~\eqref{seqM-11} is the particular case of~\eqref{seqM-rs}
where~$\sigma^M_i=i$, $\#\M_\sigma=M$, $\M_\sigma=\{1,\dots,M\}$, $\alpha_{\M_{\sigma}}=\alpha_{M}$, and~$\alpha_{\M_{\sigma+}}=\alpha_{M+1}$.

Essentially, the results in~\cite{KunRod18-cocv} tell us that
the linear closed-loop system
\begin{equation*}
      \dot y + Ay +A_{\rm rc}(t)y-\KKK_{U_{\#\M_\sigma}}^{\FF_{\M_\sigma}}(t,y)=0,\quad    y(0)=y_0\in H,
\end{equation*}
is {\em globally} exponentially stable, with the feedback control operator
\begin{equation}\label{FeedKyx}
 y\mapsto\KKK_{U_{\#\M_\sigma}}^{\FF_{\M_\sigma}}(t,y)\coloneqq
 P_{U_{\#\M_\sigma}}^{E_{\M_\sigma}^\perp}\left(Ay +A_{\rm rc}(t)y-\FF_{\M_\sigma}(P_{E_{\M_\sigma}}y)\right),
\end{equation}
with~$\FF_{\M_\sigma}\in\{\lambda\Id , A+\lambda\Id\}$, provided the condition
\begin{align}
 \overline\mu_M&\coloneqq\alpha_{\M_{\sigma+}}-\left(6+4\norm{P_{U_{\#\M_\sigma}}^{E_{\M_\sigma}^\perp}}{\LL(H)}^2\right)\norm{A_{\rm rc}}{L^\infty((0,+\infty),\LL(H,V'))}^2
 >0\label{suffalpha.linx}
 \end{align}
holds true, which is a slightly relaxed version of~\eqref{suffalpha.lin}. In case we also have
that~$\alpha_{\M_{\sigma+}}\to+\infty$ as $M\to+\infty$, then we also have Theorem~\ref{T:L.intro} with~$\alpha_{\M_{\sigma+}}$
and~$(U_{\#\M_\sigma},E_{\M_\sigma})$ in the roles
 of~$\alpha_{M+1}$ and~$(U_{M},E_{\M})$, respectively.

\subsection{Example of application}\label{sS:exRect}
We recall here that we can choose~$U_{\#\M_\sigma}$ and~${E_{\M_\sigma}^\perp}$ so that~\eqref{suffalpha.linx} is
satisfied for parabolic equations evolving in rectangular domains.
Let~$\M=\{1,2,\dots,M\}$ and~$r\in(0,1)$. For {1D} parabolic equations, evolving in a nonempty interval~$\Omega^1=(0,L_1)\subset\R$, we have that the norm of the projection
$P_{U_\M}^{E_\M^\perp}$ remains bounded if we take for the actuators the indicator functions~$1_{\omega_j^1}(x_1)$, $j\in\{1,2,\dots,M\}$,
defined as follows,
\begin{equation}\label{Act-mxe}
 1_{\omega_j^1}(x_1)\coloneqq\begin{cases}
                            1,&\mbox{if }x_1\in\Omega^1\textstyle\bigcap\omega_j^1,\\
                            0,&\mbox{if }x_1\in\Omega^1\setminus\overline{\omega_j^1},
                           \end{cases}
\quad \omega_j^1\coloneqq(c_j-\fractx{rL_1}{2M},c_j+\fractx{rL_1}{2M}),\quad c_j\coloneqq\fractx{(2j-1)L_1}{2M}.
\end{equation}
This boundedness result holds true for both Dirichlet and Neumann boundary conditions, see~\cite[Thms.~4.4 and~5.2]{RodSturm18}. Then
from~\cite[Sect.~4.8]{KunRod18-cocv} we also know that for nonempty rectangular
domains~$\Omega^\times=\prod\limits_{n=1}^{\mathbf d}(0,L_n)\subset\R^{\mathbf d}$ the operator norm of
the projections~$P_{U_{\#\M_\sigma}}^{E_{\M_\sigma}^\perp}$ remains bounded, if we take~$\#\M_\sigma=M^{\mathbf d}$, and the
cartesian product actuators and eigenfunctions as follows
\begin{align*}
U_{\#\M_\sigma}=\linspan\{1_{\omega_{\mathbf j}^\times}\mid {\mathbf j}\in\M^{\mathbf d}\}&\quad\mbox{and}\quad
E_{\M_\sigma}=\linspan\{e_{\mathbf j}^\times\mid {\mathbf j}\in\M^{\mathbf d}\},
\end{align*}
with $\omega_{\bf j}^\times\coloneqq\{(x_1,x_2,\dots,x_{\mathbf d})\in\Omega^\times\mid x_n\in\omega_{\mathbf j_n}^n\}$
and $e_{\mathbf j}^\times(x_1,x_2,\dots,x_{\mathbf d})\coloneqq\textstyle\prod\limits_{n=1}^{\mathbf d}e_{\mathbf j_n}(x_n)$.
Notice that we can also write~$1_{\omega_{\mathbf j}^\times}=\textstyle\prod\limits_{n=1}^{\mathbf d}1_{\omega_{\mathbf j_n}^n}(x_n)$, and after ordering
the eigenpairs~$(\alpha_i,e_i)$ of~$-\Delta+\Id$ in~$\Omega^\times$, we
can find~$\sigma^M$ so that~$\{e_i\mid i\in \M_\sigma\}=\{e_{\mathbf j}^\times\mid{\mathbf j}\in\M^{\mathbf d}\}$, roughly
 speaking~$\M_\sigma\sim\M^{\mathbf d}$.
Furthermore, the total volume covered by the actuators 
is given by
$r^{\mathbf d} \mathop{\rm vol}\left(\Omega^\times\right)=\textstyle\prod\limits_{n=1}^{\mathbf d}rL_n$.
That is, the total volume covered by the actuators can be
fixed a priori and taken arbitrarily small. However, for smaller~$r$ we may need a 
larger number~$M$ of actuators, because the norm
of~$P_{U_{\#\M_\sigma}}^{E_{\M_\sigma}^\perp}$ will increase as~$r$ decreases, see~\cite[Sect.~4.8.1]{KunRod18-cocv}
and~\cite[Thms.~4.4 and~5.2]{RodSturm18}.

Observe that we have $\alpha_{\M_{\sigma+}}\ge\pi^2(\frac{1^2({\mathbf d}-1)+(M+1)^2}{\overline L^2})+1$ 
under Dirichlet boundary conditions, and
~$\alpha_{\M_{\sigma+}}\ge\pi^2(\frac{0^2({\mathbf d}-1)+M^2}{\overline L^2})+1$
under Neumann boundary conditions, where~$\overline L=\max\{L_j\mid1\le j\le n\}$, which implies
that in either case~$\alpha_{\M_{\sigma+}}\to+\infty$ as ~$M\to+\infty$, and so
condition~\eqref{suffalpha.linx} will be satisfied for large enough~$M$.
Recall that~$e_i^n(x_n)=(\frac{2}{L_n})^{\frac12}\sin(\frac{i\pi x_n}{L_n})$, $i\ge1$,
under Dirichlet boundary conditions,
and~$e_1^n(x_n)=(\frac{1}{L_n})^{\frac12}$ and~$e_i^n(x_n)=(\frac{2}{L_n})^{\frac12}\cos(\frac{(i-1)\pi x_n}{L_n})$, $i\ge2$,
under Neumann boundary conditions.

For nonrectangular domains~$\Omega\subset\R^{\mathbf d}$, with~${\mathbf d}\ge2$, we do not
know whether we can choose the actuators (as indicator functions) so that~\eqref{suffalpha.linx}
is satisfied (again, in case the total volume of actuators is fixed a priori and arbitrarily small).
This is an interesting open question. Numerical simulations in~\cite{KunRod18-cocv} and~\cite{KunRod-pp18} show
the stabilizing performance of a linear feedback~$\KK_{U_M}^{\lambda\Id,\M}$ in
a nonrectangular domain.

\begin{remark}\label{R:frac-eigs}
 For the nonlinear systems, to derive the semiglobal stability result hereafter we will also
 need that~$\frac{\alpha_{\M_\sigma}}{\alpha_{\M_{\sigma+}}}$ remains bounded. 
 This is again satisfied for the choice above for rectangular domains. Indeed, under Dirichlet boundary conditions we have
$\alpha_{\M_\sigma}\le{\mathbf d}\pi^2(\frac{M^2}{\underline L^2})+1$, where~$\underline L=\min\{L_j\mid 1\le j\le n\}$, which implies
 ~$\frac{\alpha_{\M_\sigma}}{\alpha_{\M_{\sigma+}}}
 \le\frac{{\mathbf d}\pi^2\frac{M^2}{\underline L^2}+1}{\pi^2\frac{(M+1)^2}{\overline L^2}+1}
 ={\mathbf d}\frac{\overline L^2}{\underline L^2}\frac{\pi^2M^2+\underline L^2}{\pi^2(M+1)^2+\overline L^2}$.
 Analogously, under Neumann boundary conditions we have
 $\alpha_{\M_\sigma}\le{\mathbf d}\pi^2(\frac{(M-1)^2}{\underline L^2})+1$ and
 ~$\frac{\alpha_{\M_\sigma}}{\alpha_{\M_{\sigma+}}}
 \le\frac{{\mathbf d}\pi^2\frac{(M-1)^2}{\underline L^2}+1}{\pi^2\frac{M^2}{\overline L^2}+1}
 ={\mathbf d}\frac{\overline L^2}{\underline L^2}\frac{\pi^2(M-1)^2+\underline L^2}{\pi^2M^2+\overline L^2}$.
 That is, for either boundary conditions we have~$\lim\limits_{M\to+\infty}\frac{\alpha_{\M_\sigma}}{\alpha_{\M_{\sigma+}}}
 \le{\mathbf d}\frac{\overline L^2}{\underline L^2}$, which implies
  that~$\frac{\alpha_{\M_\sigma}}{\alpha_{\M_{\sigma+}}}\le\Lambda$ for a suitable~$\Lambda>0$ independent of~$M$.
\end{remark}

\subsection{Global, local, and semiglobal exponential stability}
We  recall 3 different exponential stability concepts, in order to better explain the result.
Let~$\mathfrak K\ge1$, $\mathfrak l>0$, and let~$\mathfrak H$
 be a normed space. Let us consider the dynamics in~\eqref{sys-y},
\begin{equation}\label{sys_F}
      \dot y +  Ay +A_{\rm rc}(t)y+\NN(t,y)-\mathfrak F(t,y)=0,\quad y(0)=y_0,\qquad t\ge0.
\end{equation}
with a general feedback control operator~$\mathfrak F$ taken from a suitable class~$\mathbb F$.
\begin{definition}\label{D:expstab.G}
 Let us fix~$\mathfrak F\in\mathbb F$. We say that system~\eqref{sys_F}
 is {\em globally} $(\mathfrak F,\mathfrak K,\mathfrak l,\mathfrak H)$-exponentially stable if for arbitrary given~$y_0\in\mathfrak H$, the
  corresponding solution~$y_\mathfrak F$
 is defined for all~$t\ge0$ and satisfies
 $
 \norm{y_\mathfrak F(t)}{\mathfrak H}^2\le \mathfrak K\ex^{-\mathfrak l t}\norm{y_0}{\mathfrak H}^2. 
 $
\end{definition}

\begin{definition}\label{D:expstab.L}
 Let us fix~$\mathfrak F\in\mathbb F$. We say that system~\eqref{sys_F}
 is {\em locally} $(\mathfrak F,\mathfrak K,\mathfrak l,\mathfrak H)$-exponentially stable if there exists~$\epsilon>0$, such that
 for arbitrary given~$y_0\in\mathfrak H$ with~$\norm{y_0}{\mathfrak H}<\epsilon$,  the corresponding solution~$y_\mathfrak F$
 is defined for all~$t\ge0$ and satisfies
  $
 \norm{y_\mathfrak F(t)}{\mathfrak H}^2\le \mathfrak K\ex^{-\mathfrak l t}\norm{y_0}{\mathfrak H}^2. 
 $
\end{definition}

\begin{definition}\label{D:expstab.SG}
 Let us be given a class of operators~$\mathbb F$. We say that~\eqref{sys_F}
 is {\em semiglobally} $(\mathbb F,\mathfrak H)$-exponentially stable if for arbitrary given~$R>0$,
 we can find~$\mathfrak F\in\mathbb F$, $\mathfrak K\ge1$, and~$\mathfrak l>0$, such that: for
 arbitrary given~$y_0\in\mathfrak H$ with~$\norm{y_0}{\mathfrak H}<R$, 
 the corresponding  solution~$y_{\mathfrak F}$
 is defined for all~$t\ge0$ and satisfies
  $
 \norm{y_{\mathfrak F}(t)}{\mathfrak H}^2\le \mathfrak K\ex^{-\mathfrak l t}\norm{y_0}{\mathfrak H}^2. 
 $
\end{definition}

We will consider system~\eqref{sys_F} evolving in a Hilbert~$H$, which will be considered as a pivot space, $H=H'$.
Let~$\D(A)\xhookrightarrow{\rm d,c} H$ be the domain of the diffusion-like operator, and denote 
$V\coloneqq\D(A^\frac12)\xhookrightarrow{\rm d,c} H$, and its dual by $V'$.
From the results in~\cite{KunRod18-cocv} we know that
 if~$\NN=0$ and~\eqref{suffalpha.linx} holds true, then 
there exist suitable constants~$C_1\ge1$, ~$\mu_1>0$, and~$M>0$ so that system~\eqref{sys_F} is
globally $(\KKK_{U_{\#\M_\sigma}}^{\FF_{\M_\sigma}},C_1,\mu_1,H)$-exponentially stable,
with~$\FF_{\M_\sigma}\in\{\lambda\Id ,A+\lambda\Id\}$.

Note that~$A_{\rm rc}\in L^\infty((0,+\infty),\LL(H,V'))$ is assumed in~\eqref{suffalpha.linx}.
If we (also) have that~$A_{\rm rc}\in L^\infty((0,+\infty),\LL(V,H))$, then we will (also) have strong solutions for system~\eqref{sys_F}
which will lead to the smoothing property
\[
 \norm{y(s+1)}{V}^2\le C_2\norm{y(s)}{H}^2,\quad\mbox{for all}\quad s\ge0,
\]
for a suitable constant~$C_2>0$, independent of~$s$. Hence, by standard estimates 
(e.g., following~\cite[Sect.~3]{PhanRod18-mcss}, see also~\cite[Sect.~4]{KunRod-pp18}),
we can conclude that
there is ~$C_3>0$ such that
system~\eqref{sys_F}, again with~$\NN=0$,
is again globally $(\KKK_{U_{\#\M_\sigma}}^{\FF_{\M_\sigma}},C_3,\mu_1,V)$-exponentially stable.

Afterwards, by a rather standard, still nontrivial, fixed point argument, we can derive that for a suitable constant~$C_4>0$, the perturbed system
\begin{equation}\label{sys-FeedKy-non}
      \dot y + Ay +A_{\rm rc}(t)y+\NN(t,y)-\KKK_{U_{\#\M_\sigma}}^{\FF_{\M_\sigma}}(t,y)=0,\quad    y(0)=y_0\in V,
\end{equation}
is locally $(\KKK_{U_{\#\M_\sigma}}^{\FF_{\M_\sigma}},C_4,\mu_1,V)$-exponentially stable, for a general class of nonlinearities~$\NN$.

Let us now consider the nonlinear feedback operator (cf.~\eqref{FeedKy-N}),
\begin{equation}\label{FeedKy-Nx}
 y\mapsto\KKK_{U_{\#\M_\sigma}}^{\FF_{\M_\sigma},\NN}(t,y)
 \coloneqq P_{U_{\#\M_\sigma}}^{E_{\M_\sigma}^\perp}\left(Ay +A_{\rm rc}(t)y+\NN(t,y)-\FF_{\M_\sigma}(P_{E_{\M_\sigma}}y)\right)
\end{equation}
and the class
\begin{equation}\label{F-family-N}
 \mathbb F\coloneqq\left\{\left.\KKK_{U_{\#\M_\sigma}}^{\FF_{\M_\sigma},\NN}\,\right|\begin{array}{l}
 \;M\in\N,\;U_{\#\M_\sigma}\mbox{ is a~$\#\M_\sigma$-dimensional subspace of }H,\\
 \M_\sigma\in\mathfrak{P}_{\#\M_\sigma}(\N_0),
 \mbox{ and }\FF_{\M_\sigma}\in \CC_{\rm b,i}(E_{\M_\sigma},E_{\M_\sigma})\end{array}\right\}.
\end{equation}

We will prove that the closed-loop system~\eqref{sys_F}
is semiglobally $(\mathbb F,V)$-exponentially stable,
with~$\mathbb F$ as in~\eqref{F-family-N}  and under  general
conditions on the state operators~$A$, $A_{\rm rc}$,
and~$\NN$, in~\eqref{sys_F}, under  general conditions on~$\FF_{\M_\sigma}$, and
under a particular condition on the oblique projections~$P_{U_{\#\M_\sigma}}^{E_{\M_\sigma}^\perp}$, i.e., under
a suitable ``nonlinear version'' of condition~\eqref{suffalpha.linx} 
(see condition~\eqref{suffalpha} hereafter).
In other words, for arbitrary given~$R>0$ we want to find $M\in\N$, $\M_\sigma\in\mathfrak{P}_{\#\M_\sigma}(\N_0)$, and
a set of $\#\M_\sigma$~actuators spanning~$U_{\#\M_\sigma}$ such that the solution of
system~\eqref{sys_F} with~$\mathfrak F=\KKK_{U_{\#\M_\sigma}}^{\FF_{\M_\sigma},\NN}$
satisfies
\begin{equation}\label{goal.stab-N}
 \norm{y(t)}{V}^2\le C_5\ex^{-\mu_2 t}\norm{y_0}{V}^2,\quad\mbox{for all}\quad t\ge 0
 \quad\mbox{and all}\quad y_0\in \{v\in V\mid \norm{v}{V}<R\},
\end{equation}
with $(C_5,\mu_2,M)$ independent of~$y_0$. Note that here~$(C_5,\mu_2,M)$ may depend on~$R$, though.

The assumptions on the state operators, on the ``partial feedback''~$\FF_{\M_\sigma}$, and on the
oblique projection are given in the following sections.
Such assumptions will lead to the following relaxed/generalized version of Theorem~\ref{T:N.intro},
with~$\FF_{\M_\sigma}=A+\lambda\Id$, whose proof is given in
Section~\ref{sS:proofT:N.goal}.
\begin{theorem}\label{T:N.goal}
 Suppose we can construct a sequence~$(U_{\#\M_\sigma},E_{\M_\sigma})_{M\in N}$ so 
 that both the norm $\norm{P_{U_{\#\M_\sigma}}^{E_{\M_\sigma}^\perp}}{\LL(H)}\le C_P$
 and the ratio~$\frac{\alpha_{\M_\sigma}}{\alpha_{\M_{\sigma+}}}\le \Lambda$ remain
 bounded, with both~$C_P$ and~$\Lambda>0$
independent of~$M$.
Then, for arbitrary given~$R>0$ we can find $M\in\N$ large enough so that the solution of
system~\eqref{sys_F},
with~$\mathfrak F=\KKK_{U_{\#\M_\sigma}}^{A+\lambda\Id,\NN}$,
satisfies~\eqref{goal.stab-N},
with $(C_5,\mu_2,M)$ independent of~$y_0$. That is,
system~\eqref{sys_F} is semiglobally~$(\mathbb F,V)$-exponentially stable.
\end{theorem}

\section{Assumptions and mathematical setting}\label{S:assump}
Here we present the mathematical setting and the sufficient conditions for stability of the closed-loop system.

\subsection{Assumptions on the state operators}\label{sS:assumOp}
Let~$H$ and~$V$ be separable Hilbert spaces, with~$V\subseteq H$. We will consider~$H$ as pivot space, $H'=H$.

\begin{assumption}\label{A:A0sp}
 $A\in\LL(V,V')$ is an isomorphism from~$V$ onto~$V'$, $A$ is symmetric, and $(y,z)\mapsto\langle Ay,z\rangle_{V',V}$
 is a complete scalar product on~$V.$
\end{assumption}

From now on we suppose that~$V$ is endowed with the scalar product~$(y,z)_V\coloneqq\langle Ay,z\rangle_{V',V}$,
which still makes~$V$ a Hilbert space. 
Therefore, $A\colon V\to V'$ is an isometry.
\begin{assumption}\label{A:A0cdc}
The inclusion $V\subseteq H$ is continuous, dense, and compact. 
\end{assumption}

Necessarily, we have that
the operator $A$ is densely defined in~$H$, with domain $\D(A)\coloneqq\{u\in V\mid Au\in H\}$ endowed with the scalar product
$(y,z)_{\D(A)}\coloneqq (Ay,Az)_H$, and the inclusions
\[
\D(A)\xhookrightarrow{\rm d,c} V\xhookrightarrow{\rm d,c} H\xhookrightarrow{\rm d,c} V'\xhookrightarrow{\rm d,c}\D(A)'.
\]
Further,~$A$ has compact inverse~$A^{-1}\colon H\to \D(A)$, and we can find a nondecreasing
system of (repeated) eigenvalues $(\alpha_i)_{i\in\N_0}$ and a corresponding complete basis of
eigenfunctions $(e_i)_{i\in\N_0}$:
\begin{equation*}
 0<\alpha_1\le\alpha_2\le\dots\le\alpha_i\le\alpha_{i+1}\to+\infty \quad\mbox{and}\quad Ae_i=\alpha_i e_i.
\end{equation*}

For every $\beta\in\R$, the power~$A^\beta$ of~$A$ is defined by
\[
A^\beta \sum_{i=1}^{+\infty}y_ie_i\coloneqq\sum_{i=1}^{+\infty}\alpha_i^\beta y_i e_i, 
\]
and the corresponding domains~$\D(A^{|\beta|})\coloneqq\{y\in H\mid A^{|\beta|} y\in H\}$, and 
$\D(A^{-|\beta|})\coloneqq \D(A^{|\beta|})'$.
We have~$\D(A^{\beta})\xhookrightarrow{\rm d,c}\D(A^{\beta_1})$, for all $\beta>\beta_1$,
and we can see that~$\D(A^{0})=H$, $\D(A^{1})=\D(A)$, $\D(A^{\frac{1}{2}})=V$.

For the time-dependent operators we assume the following:
\begin{assumption}\label{A:A1}
For  all~$t>0$ we have~$A_{\rm rc}(t)\in\LL(V,H)$, and there is a nonnegative constant~$C_{\rm rc}$
such that, $\norm{A_{\rm rc}}{L^\infty(\R_0,\LL(V,H))}\le C_{\rm rc}.$
\end{assumption}

\begin{assumption}\label{A:NN}
 We have~$\NN(t,\Bigcdot)\in\CC_{\rm b,i}(\D(A),H)$ 
and
there exist constants $C_\NN\ge 0$, $n\in\N_0$, 
$\zeta_{1j}\ge0$, $\zeta_{2j}\ge0$,
 $\delta_{1j}\ge 0$, ~$\delta_{2j}\ge 0$, 
 with~$j\in\{1,2,\dots,n\}$, such that
 for  all~$t>0$ and
 all~$(y_1,y_2)\in H\times H$, we have
\begin{align*}
&\norm{\NN(t,y_1)-\NN(t,y_2)}{H}\le C_\NN\textstyle\sum\limits_{j=1}^{n}
  \left( \norm{y_1}{V}^{\zeta_{1j}}\norm{y_1}{\D(A)}^{\zeta_{2j}}+\norm{y_2}{V}^{\zeta_{1j}}\norm{y_2}{\D(A)}^{\zeta_{2j}}\right)
   \norm{y_1-y_2}{V}^{\delta_{1j}}\norm{y_1- y_2}{\D(A)}^{\delta_{2j}},
\end{align*}
with~$\zeta_{2j}+\delta_{2j}<1$ and~$\delta_{1j}+\delta_{2j}\ge1$.
\end{assumption}

\bigskip\noindent
\textit{\large Examples.}
We can show that our Assumptions~\ref{A:A0sp}--\ref{A:NN} on the linear and nonlinear operators will be satisfied for parabolic equations
evolving in a bounded smooth, or rectangular, domain~$\Omega\in\R^{\mathbf d}$, $\mathbf d\in\{1,2,3\}$, as
\[
 \fractx{\p}{\p t}y+(-\nu\Delta +\Id)y +(a-1)y+b\cdot\nabla y -\textstyle\sum\limits_{i=2}^n\hat{a}_iy^i +(\hat b(t)\cdot\nabla y)y=0,\qquad \GG y\rest{\p\Omega}=0,\qquad y(0)=y_0,
\]
with~$n\le4$ if $\mathbf d=3$, and~$n\in\N$ if~$\mathbf d\in\{1,2\}$. Under either Dirichlet or Neumann homogeneous  boundary conditions.
For example, here we may take~$A=-\nu\Delta +\Id$ as the shifted Laplacian. The same
assumptions are also satisfied for the Navier--Stokes equations under homogeneous Dirichlet boundary conditions, where
we may take~$A=P_H(-\nu\Delta +\Id)$ as the shifted Stokes operator, where~$P_H$ is the
orthogonal projection in $L^2(\Omega,\R^{\mathbf d})$ onto the space~$H$ of
divergence free vector fields which are tangent to the boundary~$\p\Omega$ of~$\Omega$.
More comments and details on these examples are given later in Section~\ref{S:examples}. 

\subsection{Auxiliary estimates for the nonlinear terms}
Besides the assumptions on the state operators, presented in  Section~\ref{sS:assumOp}, we will
need also assumptions on the triple~$(\FF_{\M_\sigma},E_{\M_\sigma}, U_{\#\M_\sigma})$, which defines the feedback operator.
Before, we need to present suitable estimates resulting from Assumption~\ref{A:NN}. These are the content of the
following Proposition, whose proof follows by straightforward computations. The proof is, however, not trivial and is given in the
Appendix,  Section~\ref{Apx:proofP:NN}.

Recall the notation~$\|a\|\coloneqq\max_{1\le j\le n}\{a_j\}$,
for a sequence of constants $a_j\ge0$. We will also denote
\[
 \|P\|_\LL\coloneqq\norm{P_{E_{\M_\sigma}^\perp}^{U_{\#\M_\sigma}}}{\LL(H)},
\]
which will not lead to ambiguity, as soon as the pair~$(E_{\M_\sigma}, U_{\#\M_\sigma})$ is fixed.

\begin{proposition}\label{P:NN}
 If Assumptions~\ref{A:A0sp}, \ref{A:A0cdc}, and~\ref{A:NN} hold true, then 
there are constants $\overline C_{\NN1}>0$, and~$\overline C_{\NN2}>0$
such that:
  for all~$\widehat\gamma_0>0$, all~$t>0$, 
  all~$(y_1,y_2)\in H\times H$, and all~$(q,Q)\in E_{\M_\sigma}\times E_{\M_\sigma}^\perp$, we have
\begin{align}
 &2\Bigl( P_{E_{\M_\sigma}^\perp}^{U_{\#\M_\sigma}}\left(\NN(t,y_1)-\NN(t,y_2)\right),A(y_1-y_2)\Bigr)_{H}
 \le \widehat\gamma_0 \norm{y_1- y_2}{\D(A)}^{2}\notag\\
  &\hspace*{.5em}
  +\left(1+\widehat\gamma_0^{-\frac{1+\|\delta_2\|}{1-\|\delta_2\|} }\right)
 \overline C_{\NN1}\sum\limits_{j=1}^n\left( \norm{y_1}{V}^\frac{2\zeta_{1j}}{1-\delta_{2j}} \norm{y_1}{\D(A)}^\frac{2\zeta_{2j}}{1-\delta_{2j}}
 +\norm{y_2}{V}^\frac{2\zeta_{1j}}{1-\delta_{2j}} \norm{y_2}{\D(A)}^\frac{2\zeta_{2j}}{1-\delta_{2j}}
 \right)\norm{y_1-y_2}{V}^{\frac{2\delta_{1j}}{1-\delta_{2j}}},\label{NNyAy}\\
 &2\Bigl(  P_{E_{\M_\sigma}^\perp}^{U_{\#\M_\sigma}}\left(\NN(t,q+Q)-\NN(t,q)\right),AQ\Bigr)_{H} 
 \le \widehat\gamma_0 \norm{Q}{\D(A)}^{2}\notag\\
  &\hspace*{.5em}
  +\left(1+\widehat\gamma_0^{-\frac{1+\|\delta_2\|}{1-\|\zeta_2+\delta_2\|} }\right)
  \overline C_{\NN2}
 \left(1 +\norm{q}{V}^{2\|\frac{\zeta_1+\delta_1}{1-\zeta_2-\delta_2}\|}\right)
 \left(1+\norm{q}{\D(A)}^{\|\frac{\zeta_2}{1-\delta_2}\|}\right)\left(
 1+\norm{Q}{V}^{2\|\frac{\zeta_1+\delta_1}{1-\zeta_2-\delta_2}\|-2}
 \right)\norm{Q}{V}^{2}.\label{NNQAQ}
\end{align}
 Further, 
 $ \overline C_{\NN1}=\ovlineC{n,\frac{1}{1-\|\delta_{2}\|},C_\NN,\|P\|_\LL}$
 and~$\;\overline C_{\NN2}=\ovlineC{n,\|\zeta_1\|,\|\zeta_2\|,\frac{1}{1-\|\delta_{2}\|},\frac{1}{1-\|\zeta_2+\delta_{2}\|},
 C_\NN,\|P\|_\LL}.
 $
 \end{proposition}

Inequality~\eqref{NNQAQ} will be used to prove the existence of a solution for the closed-loop system, while~\eqref{NNyAy}
will be used to prove the uniqueness of the solution.

\subsection{Assumptions on the oblique projection based feedback}\label{sS:assumProj}
We present here the assumptions on the triple~$(\FF_{\M_\sigma},E_{\M_\sigma}, U_{\#\M_\sigma})$.
Observe that, from~\eqref{sys-FeedKy-non} and~\eqref{FeedKy-Nx},
the orthogonal projection $q\coloneqq P_{E_{\M_\sigma}}y$
satisfies
\begin{equation}\label{sys-q}
\dot q=-\FF_{\M_\sigma}(q), 
\end{equation}
For the exponential stability of~\eqref{sys-FeedKy-non} we need $q(t)$ to decrease exponentially to zero.
We will also ask for integrability of~$q$ and~$\dot q$ as follows.
\begin{assumption}\label{A:FFM}
 We have~$\FF_{\M_\sigma}\in\CC_{\rm b,i}(E_{\M_\sigma},E_{\M_\sigma})$ and there are
 constants~$C_{q0}\ge0$,~$C_{q1}\ge1$,~$C_{q2}\ge0$, $C_{q3}\ge0$, $\xi\ge1$,~$\lambda>0$,
 $\beta_{0}\ge0$,
 $\beta_{1}\ge0$, 
 $\beta_{2}\ge0$, $\eta_{1}\ge0$, 
 $\eta_{2}\ge0$, and~$1<\mathfrak r<\frac{1}{\|\zeta_2+\delta_2\|}$, all independent of~$\M_\sigma$, such that:
\begin{align*}
&\norm{\FF_{\M_\sigma}(\widehat q)}{H}\le C_{q_0}\norm{\widehat q}{\D(A)}^\xi,\quad\mbox{for all}\quad \widehat q\in E_{\M_\sigma},\\
&\beta_1+\beta_2\ge1,\qquad\mathfrak r\|\zeta_1+\delta_1+(\eta_1+\eta_2)(\zeta_2+\delta_2)\|\ge1,\\
&\left(\dnorm{\fractx{\zeta_1+\delta_1}{1-\zeta_2-\delta_2}}{}-1\right)\beta_2<1-\dnorm{\fractx{\zeta_2}{1-\delta_2}}{}
,\qquad
\left(\dnorm{\fractx{\zeta_1+\delta_1}{1-\zeta_2-\delta_2}}{}-1\right)
\dnorm{\zeta_{2}+\delta_{2}}{}\mathfrak r\eta_2<1-\dnorm{\fractx{\zeta_2}{1-\delta_2}}{},
\end{align*}
and
 every solution~$q$ of system~\eqref{sys-q} satisfies, for all~$t\ge0$,
\begin{align*}
    &\norm{q(t)}{V}\le C_{q1}\ex^{-\lambda t}\norm{q(0)}{V},\qquad\norm{q}{L^2(\R_{0},\D(A))}\le
 C_{q2}\norm{q(0)}{V}^{\eta_1}\norm{q(0)}{\D(A)}^{\eta_2},\qquad\mbox{and}\\
  &\norm{Aq-\FF_{\M_\sigma}(q)}{L^{2\mathfrak r}(\R_{0},H)}\le
  C_{q3}\norm{q(0)}{V}^{\beta_1}\norm{q(0)}{\D(A)}^{\beta_2}.
\end{align*}
\end{assumption}

As we see, when~$\dnorm{\fractx{\zeta_1+\delta_1}{1-\zeta_2-\delta_2}}{}\ne1$,
we ask for small enough magnitudes of~$\beta_2$.
Observe that with~$\FF_{\M_\sigma}=\lambda\Id $ we can take~$({\beta_1},{\beta_2})=(0,1)$,
but we cannot take~${\beta_2}<1$, that is, Assumption~\ref{A:FFM} does not necessarily hold true, being satisfied only if
$\dnorm{\frac{\zeta_2}{1-\delta_2}}{}+\dnorm{\fractx{\zeta_1+\delta_1}{1-\zeta_2-\delta_2}}{}<2$.
Instead, with~$\FF_{\M_\sigma}=A+\lambda\Id$ the Assumption~\ref{A:FFM} is always satisfied, because we can take
~$({\beta_1},{\beta_2})=({\eta_1},{\eta_2})=(1,0)$. This is why in Theorem~\eqref{T:N.goal} we
write only~$A+\lambda\Id$, and exclude~$\lambda\Id $.

Finally, we present the assumptions involving~$P_{U_{\#\M_\sigma}}^{E_{\M_\sigma}^\perp}$.
Note that both~$U_{\#\M_\sigma}$ and~$E_{\M_\sigma}^\perp$ are closed subspaces. Thus,
the {\em oblique}
projection~$P_{U_{\#\M_\sigma}}^{E_{\M_\sigma}^\perp}\colon H\mapsto U_{\#\M_\sigma}$ is well defined if, and only if, we have the direct sum
$H=U_{\#\M_\sigma}\oplus E_{\M_\sigma}^\perp$.
In particular, by considering the feedback~\eqref{FeedKy}, we are necessarily assuming the following.
\begin{assumption}\label{A:DS}
 We have the direct sum~$H=U_{\#\M_\sigma}\oplus E_{\M_\sigma}^\perp$.
\end{assumption}
Recall that~$\#\M_\sigma={\rm dim}(U_{\#\M_\sigma})={\rm dim}(E_{\M_\sigma})$. 
Recall also that Assumption~\ref{A:DS} means that for every  given~$h\in H$ there exists one, and only one,
pair $(h_{U_{\#\M_\sigma}},h_{E_{\M_\sigma}^\perp})$ satisfying
\[
 h=h_{U_{\#\M_\sigma}}+h_{E_{\M_\sigma}^\perp}
 \qquad\mbox{with}\qquad (h_{U_{\#\M_\sigma}},h_{E_{\M_\sigma}^\perp})\in {U_{\#\M_\sigma}}\times E_{\M_\sigma}^\perp.
\]
Hence we simply take
$
P_{U_{\#\M_\sigma}}^{E_{\M_\sigma}^\perp}h\coloneqq h_{U_{\#\M_\sigma}}.
$
Similarly, the oblique projection in~$H$ onto~${E_{\M_\sigma}^\perp}$ along~$U_{\#\M_\sigma}$ is defined by
$
P_{E_{\M_\sigma}^\perp}^{U_{\#\M_\sigma}}h\coloneqq h_{E_{\M_\sigma}^\perp}.
$
Observe that~$P_{U_{\#\M_\sigma}}^{E_{\M_\sigma}^\perp}h$ is the only element in the set~$(h+{E_{\M_\sigma}^\perp}){\textstyle\bigcap} U_{\#\M_\sigma}$, and
~$P_{E_{\M_\sigma}^\perp}^{U_{\#\M_\sigma}}h$ is the only element in the set~$(h+U_{\#\M_\sigma}){\textstyle\bigcap} {E_{\M_\sigma}^\perp}$.

The oblique projection~$P_{U_{\#\M_\sigma}}^{E_{\M_\sigma}^\perp}$ is {\em orthogonal} if, and only if, ${U_{\#\M_\sigma}}={E_{\M_\sigma}}$. The operator
norm of an orthogonal projection
onto a closed subspace~$F\subset H$ is always equal to~$1$, if~$F\ne\{0\}$, that is, ~$\norm{P_{F}^{F^\perp}}{\LL(H)}=1$. If~$F=\{0\}$,
then~$\norm{P_{F}^{F^\perp}}{\LL(H)}=0$.
The operator norm of an oblique {\em nonorthogonal} projection is strictly larger than~$1$. In particular,
in case $U_{\#\M_\sigma}\ne{E_{\M_\sigma}}$ we have that
$\norm{P_{U_{\#\M_\sigma}}^{E_{\M_\sigma}^\perp}}{\LL(H)}>1$. 

Orthogonal projections~$P_{F}^{F^\perp}$ will be denoted by~$P_{F}$, for simplicity. 
We have the following properties, which are useful in the computations hereafter.
\begin{subequations}\label{propProj}
 \begin{align}
  P_{E_{\M_\sigma}}&=P_{E_{\M_\sigma}}P_{U_{\#\M_\sigma}}^{E_{\M_\sigma}^\perp},&
   P_{U_{\#\M_\sigma}}^{E_{\M_\sigma}^\perp}&=P_{U_{\#\M_\sigma}}^{E_{\M_\sigma}^\perp}P_{E_{\M_\sigma}},\\
  P_{E_{\M_\sigma}^\perp}^{U_{\#\M_\sigma}}&=P_{E_{\M_\sigma}^\perp}+P_{E_{\M_\sigma}^\perp}^{U_{\#\M_\sigma}}P_{E_{\M_\sigma}},&
  P_{E_{\M_\sigma}^\perp}^{U_{\#\M_\sigma}}&=P_{E_{\M_\sigma}^\perp}-P_{E_{\M_\sigma}^\perp}P_{U_{\#\M_\sigma}}^{E_{\M_\sigma}^\perp}P_{E_{\M_\sigma}}.
 \end{align}
\end{subequations}

For further comments on oblique projections we refer to~\cite[Sect.~2.2]{KunRod18-cocv} and~\cite[Sect.~3]{RodSturm18}.

The next assumption is less trivial and it is the one that gives us the stability condition. 
In order to state the assumption we start
by recalling the particular eigenvalues~$\alpha_{\M_{\sigma}}$ and~$\alpha_{\M_{\sigma+}}$, defined in~\eqref{VDAalphaM}.
Then we define suitable functions as follows.
For a given triple~$\gamma=(\gamma_1,\gamma_2,\gamma_3)\in\R_0^3$ with positive coordinates, and a given
function~$q\in L^\infty(\R_{0}, E_{\M_\sigma})$, we define 
\begin{subequations}\label{frak-apr}%
\begin{align}
 &\mathfrak a_0\coloneqq 2-\gamma_1-\gamma_2-\gamma_3,\qquad
  \mathfrak a_1\coloneqq\gamma_1^{-1}\norm{P_{E_{\M_\sigma}^\perp}^{U_{\#\M_\sigma}}}{\LL(H)}^2 C_{\rm rc}^2,\qquad
  \mathfrak a_2\coloneqq \gamma_3^{-\frac{2}{1-\|\zeta_2+\delta_2\|} }
  \overline C_{\NN2},\\
  &\mathfrak q\coloneqq\max_{t\ge 0}\left( 1 +\norm{q}{V}^{2\|\frac{\zeta_1+\delta_1}{1-\zeta_2-\delta_2}\|}\right)
 \left(1+\norm{q}{\D(A)}^{\|\frac{\zeta_2}{1-\delta_2}\|}\right)
 ,\qquad
   \mathfrak p\coloneqq\|\fractx{\zeta_1+\delta_1}{1-\zeta_2-\delta_2}\|-1,\qquad
   \mathfrak h\coloneqq\gamma_2^{-1}\norm{\MM(q)}{H}^2,
 \end{align}
where the constants~$C_{\rm rc}$ and~$C_\NN$ are as in Assumptions~\ref{A:A1} and~\ref{A:NN}, respectively, and
\end{subequations}
\begin{align}\label{MM}
    \MM(q)&\coloneqq-P_{E_{\M_\sigma}^\perp}^{U_{\#\M_\sigma}}\Bigl(Aq +A_{\rm rc}(t)q+\NN(t,q)\Bigr)
    - P_{E_{\M_\sigma}^\perp} P_{U_{\#\M_\sigma}}^{E_{\M_\sigma}^\perp}\FF_{E_{\M_\sigma}}(q)\notag\\
    &=-P_{E_{\M_\sigma}^\perp}^{U_{\#\M_\sigma}}\Bigl(Aq +A_{\rm rc}(t)q+\NN(t,q)-\FF_{E_{\M_\sigma}}(q)\Bigr),\qquad 
    q\in E_{\M_\sigma}.
\end{align}

\begin{assumption}\label{A:suffalpha}
With~$\mathfrak r>1$ as in Assumption~\ref{A:FFM}, we have that 
\begin{align}\label{suffalpha}
 \alpha_{\M_{\sigma+}}>\inf_{\begin{subarray}{c}
                       (\gamma_1,\gamma_2,\gamma_3)\in\R_0^3,\\
                       \mathfrak a_0>0,\hspace*{.7em}\overline\e>0,\\
                       \mathfrak a_0\alpha_{\M_{\sigma+}}-\mathfrak a_1-\mathfrak a_2\mathfrak q-\overline\e>0
                      \end{subarray}}
\frac{1}{\mathfrak a_0}\left(\mathfrak a_1+\mathfrak a_2\mathfrak q+\overline\e
+(\mathfrak p+1)\mathfrak a_2\mathfrak q\left(\norm{Q_0}{V}^2
  +(\tfrac{\mathfrak r}{\mathfrak r-1}\overline\e)^{-\frac{\mathfrak r-1}{\mathfrak r}}
  \norm{\mathfrak h}{L^{\mathfrak r}(\R_{0},\R)}\right)^{\mathfrak p}\right).
  \end{align}
\end{assumption}

\bigskip\noindent
\textit{\large Remarks and examples.}
Note that Assumption~\ref{A:FFM}
holds true with, for example,~$\FF_{\M_\sigma}=A+\lambda\Id$.
Of course it would also hold true with~$\FF_{\M_\sigma}=\lambda\Id$ if we would not ask for the constants in
there to be independent of~$\M_\sigma$. 
Such independence is helpful to prove that, in particular situations as in Corollary~\ref{C:bddratP=suffalpha}
below, Assumption~\ref{A:suffalpha}
will be satisfied for large enough~$M$. It is also helpful to prove, later on, that
the number of actuators depend only in the~$V$-norm of the initial condition~$y(0)=q(0)+Q(0)$, with~$(q(0),Q(0))\in
E_{\M_\sigma}\times E_{\M_\sigma}^\perp$ (cf.~Thm.~\ref{T:N.goal}).

Concerning Assumption~\ref{A:DS}, it is needed to define the oblique projection
$P_{U_{\#\M_\sigma}}^{E_{\M_\sigma}^\perp}$ and it is not difficult to find the actuators such that it holds true.
What is not clear is whether we can find the actuators, for example a finite number of
indicator functions~$1_{\omega_i}$ in the setting
of parabolic equations,
so that Assumption~\ref{A:suffalpha} also holds true. Indeed,
recalling~\eqref{MM} and~\eqref{frak-apr}, and using Assumption~\ref{A:NN}, we obtain
\begin{align*}
    &\quad\norm{\mathfrak h}{{L^{\mathfrak r}(\R_{0},\R)}}
    =\gamma_2^{-1}\norm{\MM(q)}{{L^{2\mathfrak r}(\R_{0},H)}}^2
     =\gamma_2^{-1}
     \norm{-P_{E_{\M_\sigma}^\perp}^{U_{\#\M_\sigma}}\Bigl(Aq +A_{\rm rc}(t)q+\NN(t,q) -\FF_{\M_{\sigma}}(q)\Bigr)
     }{{L^{2\mathfrak r}(\R_{0},H)}}^2\\
     &\le\ovlineC{\|P\|_\LL}
     \left(\norm{Aq-\FF_{\M_{\sigma}}(q)}{L^{2\mathfrak r}(\R_{0},H)}^2 
     +\norm{A_{\rm rc}(t)q+\NN(t,q)}{L^{2\mathfrak r}(\R_{0},H)}^2\right)\\
     &\le\ovlineC{n,\|P\|_\LL}
     \!\!\left(\!\!C_{q3}\!\norm{q(0)}{V}^{2\beta_1}\!\norm{q(0)}{\D(A)}^{2\beta_2}
     +C_{\rm rc}^2C_{q1}^2(4\lambda)^{-\frac12}\!\norm{q(0)}{V}^2
     +C_\NN^2\sum\limits_{j=1}^n\!\norm{\norm{q}{V}^{\zeta_{1j}+\delta_{1j}}
    \!\norm{q}{\D(A)}^{\zeta_{2j}+\delta_{2j}}}{L^{2\mathfrak r}(\R_{0},\R)}^2\!\right)\!.
    \end{align*}
  Observe that from Assumption~\ref{A:FFM} we have 
  \begin{align*}\int_{\R_{0}}\norm{q}{V}^{2\mathfrak r(\zeta_{1j}+\delta_{1j})}
    \norm{q}{\D(A)}^{2\mathfrak r(\zeta_{2j}+\delta_{2j})}\,\ed s
    &\le\left(\int_{\R_{0}}\norm{q}{V}^{\frac{2\mathfrak r(\zeta_{1j}+\delta_{1j})}{1-\mathfrak r(\zeta_{2j}+\delta_{2j})}}
    \,\ed s\right)^{1-\mathfrak r(\zeta_{2j}+\delta_{2j})}
    \left(\int_{\R_{0}}
    \norm{q}{\D(A)}^2\,\ed s\right)^{\mathfrak r(\zeta_{2j}+\delta_{2j})}\\
    &\le\ovlineC{C_{q1},C_{q2},\frac1{\lambda}}\norm{q(0)}{V}^{2\mathfrak r(\zeta_{1j}+\delta_{1j})}
         \norm{q(0)}{V}^{2\eta_1\mathfrak r(\zeta_{2j}+\delta_{2j})}
         \norm{q(0)}{\D(A)}^{2\eta_2\mathfrak r(\zeta_{2j}+\delta_{2j})},
    \end{align*}
  which allow us to derive that, with~$\overline C_1=\ovlineC{n,\|P\|_\LL,C_{\rm rc},C_\NN,C_{q1},C_{q2},C_{q3},\frac1{\lambda}}$,
  \begin{align}
    \norm{\mathfrak h}{{L^{\mathfrak r}(\R_{0},\R)}}
      &\le \overline C_1
     \left(\norm{q(0)}{V}^{2}+\alpha_{\M_\sigma}^{\beta_2}\norm{q(0)}{V}^{2(\beta_1+\beta_2)}
          +\sum\limits_{j=1}^n\alpha_{\M_\sigma}^{\mathfrak r\eta_2(\zeta_{2j}+\delta_{2j})}
          \norm{q(0)}{V}^{2\mathfrak r(\zeta_{1j}+\delta_{1j}+(\eta_1+\eta_2)(\zeta_{2j}+\delta_{2j}))}\right)\notag\\
       &\le\ovlineC{n,\|P\|_\LL,C_{\rm rc},C_\NN,C_{q1},C_{q2},C_{q3},\frac1{\lambda},\norm{q(0)}{V}}
           \left(1+\alpha_{\M_\sigma}^{\beta_2}+\alpha_{\M_\sigma}^{\mathfrak r\eta_2\|\zeta_{2}+\delta_{2}\|}
      \right)\norm{q(0)}{V}^{2},\label{normhLr}
     \end{align}
    Recall also that~$\beta_1+\beta_2\ge1$
  and~$\mathfrak r\|\zeta_{1}+\delta_{1}+(\eta_1+\eta_2)(\zeta_{2}+\delta_{2}\|
  \ge1$.

\begin{corollary}\label{C:bddratP=suffalpha}
Suppose that the ratio~$\frac{\alpha_{\M_{\sigma}}}{\alpha_{\M_{\sigma+}}}\le\Lambda$ and the projection norm
$\norm{P_{E_{\M_\sigma}^\perp}^{U_{\#\M_\sigma}}}{\LL(H)}\le C_P$ both remain bounded,
with~$(\Lambda,C_P)$ independent of~$M$. Then Assumption~\ref{A:suffalpha} is satisfied.
\end{corollary}
\begin{proof} 
We know that~$\lim\limits_{M\to+\infty}\alpha_{\M_{\sigma+}}=+\infty$, then
for fixed~$\gamma\in\R_0^3$, such that~$\mathfrak a_0>0$, and~$\overline\e>0$, we see that~\eqref{suffalpha} will
be satisfied for large enough~$M$, 
because~$0\le \max\{\beta_2\mathfrak p,\mathfrak r\eta_2\|\zeta_{2}+\delta_{2}\|\mathfrak p\}<1$,
due to Assumption~\ref{A:FFM}. Note that the constant $\overline C$ in~\eqref{normhLr} is independent of~$\alpha_{\M_\sigma}$.
 \qed\end{proof}

The boundedness of the ratio~$\frac{\alpha_{\M_{\sigma}}}{\alpha_{\M_{\sigma+}}}\le\Lambda$ is
assumed in Theorem~\ref{T:N.goal}. This ratio depends only on the choice of~$\M_\sigma$ (cf.~Rem.~\ref{R:frac-eigs}). 
On the other hand, the boundedness of~$\norm{P_{E_{\M_\sigma}^\perp}^{U_{\#\M_\sigma}}}{\LL(H)}\le C_P$
is generally a nontrivial
assumption. However, as we have seen in Section~\ref{sS:exRect} for parabolic equations
evolving in a bounded rectangular domain we can choose~$(E_{\M_\sigma},U_{\#\M_\sigma})$ so that the
operator norm of the projection remains
bounded, furthermore the total volume~${\rm vol}({\bigcup_{i=1}^M\omega_i})=r{\rm vol}(\Omega)$ can be arbitrarily small.
On the other hand we have also mentioned in Section~\ref{sS:exRect} that for general domains such choice of~$(E_{\M_\sigma},U_{\#\M_\sigma})$
is an open interesting question.

\section{Stability of the closed-loop system}\label{S:feedback}
Here we prove that system~\eqref{sys_F} is exponentially stable with the feedback in~\eqref{FeedKy-Nx}, provided
the above assumptions are satisfied by
the state operators and the triple~$(\FF_{\M_\sigma},U_{\#\M_\sigma},E_{\M_\sigma})$.

Let $\mathfrak a_0,\mathfrak a_1,\mathfrak a_2,\mathfrak p$, $\mathfrak q$, and~$\mathfrak h$, be as in~\eqref{frak-apr},
and~$\overline\e$ be as in~\eqref{suffalpha}. Note that, if Assumption~\ref{A:suffalpha} is satisfied, then
\begin{subequations}\label{eps.tildeeps}
\begin{align}\label{suffalpha.gamma}
 &\mathfrak a_0\alpha_{\M_{\sigma+}}>
\mathfrak a_1+\mathfrak a_2\mathfrak q+\overline\e+(\mathfrak p+1)\mathfrak a_2\mathfrak q\left(\norm{Q_0}{V}^2
  +(\tfrac{\mathfrak r}{\mathfrak r-1}\overline\e)^{-\frac{\mathfrak r-1}{\mathfrak r}}
  \norm{\mathfrak h}{L^{\mathfrak r}(\R_{0},\R)}\right)^{\mathfrak p},\quad\mbox{for some}\quad\gamma\in\R_0^3.
\intertext{We define the constants}
 &\e\coloneqq\mathfrak a_0\alpha_{\M_{\sigma+}}-\mathfrak a_1-\mathfrak a_2\mathfrak q
-\mathfrak a_2\mathfrak q\norm{Q_0}{V}^{2\mathfrak p},\\
&\widetilde\e\coloneqq\mathfrak a_0\alpha_{\M_{\sigma+}}-\mathfrak a_1-\mathfrak a_2\mathfrak q 
-\mathfrak a_2\mathfrak q(\mathfrak p+1)\left(\norm{Q_0}{V}^2
+(\tfrac{r}{r-1}\overline\e)^{-\frac{r-1}{r}}\norm{\mathfrak h}{L^r(\R_{0},H)}\right)^{\mathfrak p}
\intertext{and observe that, for~$\gamma$ as in~\eqref{suffalpha.gamma}, we have}
&\e\ge\widetilde\e\ge\overline\e>0. 
\end{align}
\end{subequations}
\begin{theorem}\label{T:main.1pair}
Let Assumptions~\ref{A:A0sp}--\ref{A:NN} and~\ref{A:FFM}--\ref{A:suffalpha} be satisfied.
Then system~\eqref{sys_F} is exponentially stable,
with~$\mathfrak F=\KKK_{U_{\#\M_\sigma}}^{\FF_{\M_\sigma},\NN}$ as in~\eqref{FeedKy-Nx}.  The solution~$y$
satisfies~$\norm{y(t)}{V}
 \le \overline C\ex^{-\frac{\mu}{2} t}\norm{y(0)}{V}$, for all~$t\ge0$, where~$\mu<\min\{\widetilde\e,2\lambda\}$, $\widetilde\e$ is
 as in~\eqref{eps.tildeeps}, and~$\overline C=\ovlineC{n,\|P\|_\LL,C_{\rm rc},C_\NN,\frac{1}{\widetilde\e},C_{1q},
   \norm{q(0)}{V},\frac{1}{\gamma_3},\frac{1}{\widetilde\e-\mu},\alpha_{\M_\sigma}}
   $.
\end{theorem}

\subsection{Orthogonal decomposition of the solution}
Observe that we may write system~\eqref{sys_F}, with~$\mathfrak F=\KKK_{U_{\#\M_\sigma}}^{\FF_{\M_\sigma},\NN}$,
for time~$t\ge 0$, as
\begin{align}\label{sys_F-dec}
  \dot y + \Bigl(\Id-P_{U_{\#\M_\sigma}}^{E_{\M_\sigma}^\perp}\Bigr)\Bigl(Ay +A_{\rm rc}(t)y+\NN(t,y)\Bigr)
  + P_{U_{\#\M_\sigma}}^{E_{\M_\sigma}^\perp}\FF_{\M_\sigma}(P_{E_{\M_\sigma}}y)&=0,\qquad y(0)=y_0.
\end{align}
Splitting~$y$ as
$
 y=P_{E_{\M_\sigma}}y+P_{E_{\M_\sigma}^\perp}y,
$
with~$(q,Q)\coloneqq(P_{E_{\M_\sigma}}y,P_{E_{\M_\sigma}^\perp}y)$ we obtain, using the properties in~\eqref{propProj},
\begin{align*}
  \dot q +\FF_{E_{\M_\sigma}}(q)&=0,\\
  \dot Q + P_{E_{\M_\sigma}^\perp}^{U_{\#\M_\sigma}}\Bigl(Ay +A_{\rm rc}(t)y+\NN(t,y)\Bigr)
  +P_{E_{\M_\sigma}^\perp} P_{U_{\#\M_\sigma}}^{E_{\M_\sigma}^\perp}\FF_{E_{\M_\sigma}}(q)&=0.
\end{align*}

Now we write
\[
 \NN(t,q,Q)\coloneqq\NN(t,y)-\NN(t,q)
\]
and
$
    \dot Q + P_{E_{\M_\sigma}^\perp}^{U_{\#\M_\sigma}}\Bigl(AQ +A_{\rm rc}(t)Q+\NN(t,q,Q)\Bigr)
    =\MM(q)
$
with~$\MM$ as in~\eqref{MM}.
Note that we have~$P_{E_{\M_\sigma}^\perp} P_{U_{\#\M_\sigma}}^{E_{\M_\sigma}^\perp}\FF_{E_{\M_\sigma}}(q)
=P_{E_{\M_\sigma}^\perp} (1-P_{E_{\M_\sigma}^\perp}^{U_{\#\M_\sigma}})\FF_{E_{\M_\sigma}}(q)
=-P_{E_{\M_\sigma}^\perp} P_{E_{\M_\sigma}^\perp}^{U_{\#\M_\sigma}}\FF_{E_{\M_\sigma}}(q)$,
due to~$\FF_{E_{\M_\sigma}}(q)\in E_{\M_\sigma}$.
Hence, system~\eqref{sys_F-dec} splits as follows, by setting $y_0\eqqcolon q_0+Q_0\in V$,
\begin{subequations}\label{sys-split-qQ}
\begin{align}
  \dot q +\FF_{E_{\M_\sigma}}(q)&=0,&\quad  q(0)&=q_0\in E_{\M_\sigma},\label{sys-split-q}\\
  \dot Q + P_{E_{\M_\sigma}^\perp}^{U_{\#\M_\sigma}}\Bigl(AQ +A_{\rm rc}(t)Q+\NN(t,q,Q)\Bigr)
    &=\MM(q),&\quad  Q(0)&=Q_0\in E_{\M_\sigma}^\perp{\textstyle\bigcap}V.\label{sys-split-Q}
\end{align}
\end{subequations}

We will start by studying~\eqref{sys-split-Q}, for a given function~$q$ taking values in~$E_{\M_\sigma}$.
\begin{theorem}\label{T:wsol_QqV}
Let Assumptions~\ref{A:A0sp}--\ref{A:NN} and~\ref{A:FFM}--\ref{A:suffalpha} hold true, for a 
suitable space ${U_{\#\M_\sigma}}\subset H$ and a subset~$\M_\sigma\in\mathfrak P_{\#\M_\sigma}(\N_0)$.
Let also $Q_0\in E_{\M_\sigma}^\perp{\textstyle\bigcap}V$, and
 $q\in L^\infty(\R_{0},E_{\M_\sigma})$. Then
there exists a global strong solution~$Q\in W_{\rm loc}(\R_{0},\D(A),H)$,  for the system~\eqref{sys-split-Q},
taking its values in~$E_{\M_\sigma}^\perp{\textstyle\bigcap}V$.
Moreover the solution is unique and satisfies, for all~$t\ge 0$,
\begin{subequations}\label{sys-Q0-est}
\begin{align}
 \norm{Q(t)}{V}^2&\le \ex^{-\e t} \norm{Q_0}{V}^2
 + \int_{0}^{t} \ex^{-\widetilde\e(t-s)}\norm{\mathfrak h(s)}{\R}\,\ed s,\\
  \norm{Q(t)}{V}^2&\le\norm{Q_0}{V}^2
  +(\fractx{\mathfrak r}{\mathfrak r-1}\overline\e)^{-\fractx{\mathfrak r-1}{\mathfrak r}}
  \norm{\mathfrak h}{L^{\mathfrak r}(\R_{0},\R)},
\end{align}
 \end{subequations}
with~$\e$, $\widetilde\e$, and~$\overline\e$ as in~\eqref{eps.tildeeps}.
\end{theorem}

The proof is given hereafter in Section~\ref{sS:ProofT:wsol_QqV}, where the
local stability of~\eqref{sys-split-Q} is reduced to the local stability of
a suitable scalar {\sc ode} system in the form
\begin{align}\label{dyn-odewh}
 \dot w &=-\bigl(\breve C_1 -\breve C_2\norm{w}{\R}^{p}\bigr) w +\norm{h}{\R},\qquad
 w(0)=w_0\in\R.
 \end{align}
 where~$\breve C_1>0$, ~$\breve C_2>0$, and~$w$ takes its values in~$\R$, say
 for some given $\tau>0$ we have~$w(t)\in\R$ for~$t\in[0,\tau)$.
 
 \subsection{Auxiliary {\sc ode} stability results}
  Below ~$\breve C_1>0$ and~$\breve C_2>0$ are positive constants.
 We will look at~\eqref{dyn-odewh}
 as a perturbation of the system
 \begin{align}\label{dyn-odewh0}
 \dot w &=-\bigl(\breve C_1 -\breve C_2\norm{w}{\R}^{p}\bigr) w ,\qquad
 w(0)=w_0\in\R.
 \end{align}
 
 \begin{proposition}\label{P:odeh0}
 Let~$p>0$. If~$\norm{w_0}{\R}<(\frac{\breve C_1}{\breve C_2})^\frac1{p}$, then the solution of
 system~\eqref{dyn-odewh0} satisfies
 \begin{equation}\label{est.wh0}
 \ex^{-\breve C_1 (t-s)} \norm{w(s)}{\R}\le\norm{w(t)}{\R}\le \ex^{-\e (t-s)} \norm{w(s)}{\R},\quad\mbox{for all}\quad t\ge s \ge0, 
 \end{equation}
 with~$\e\coloneqq \breve C_1 -\breve C_2\norm{w_0}{\R}^{p}>0$.
 \end{proposition}

 The proof is straightforward. For the sake of completeness we give it in the Appendix, Section~\ref{Apx:proofP:odeh0}.

Next, for the perturbed {\sc ode} we have the following.
\begin{lemma}\label{L:odeh}
 Let~$p>0$, $r>1$, and~$h\in L^r(\R_{0},\R)$. If there exists~$\overline\e>0$ such that the inequality
 \begin{align}\label{bar.eps}
 \norm{w_0}{\R}+(\fractx{r}{r-1}\overline\e)^{-\fractx{r-1}{r}}\norm{h}{L^r(\R_{0},\R)}
  &<\left(\tfrac{\breve C_1-\overline\e}{\breve C_2(p+1)}\right)^\frac1{p}
  \end{align}
is satisfied, then the solution~$w=w^h$ of
 system~\eqref{dyn-odewh} satisfies, for all~$t\ge0$
 \begin{subequations}\label{est.wh0a}
 \begin{align}
 \norm{w^h(t)}{\R}&\le \ex^{-\e t} \norm{w_0}{\R}+ \int_{0}^{t} \ex^{-\widetilde\e(t-s)}\norm{h(s)}{\R}\,\ed s,
 \\
  \norm{w^h(t)}{\R}&\le\norm{w_0}{\R}+(\fractx{r}{r-1}\overline\e)^{-\fractx{r-1}{r}}\norm{h}{L^r(\R_{0},\R)},
 \end{align}
  \end{subequations}
   with~$\e\coloneqq \breve C_1 -\breve C_2\norm{w_0}{\R}^{p}$
  and~$\widetilde\e\coloneqq \breve C_1-\breve C_2(p+1)\left(\norm{w_0}{\R}
  +(\fractx{r}{r-1}\overline\e)^{-\fractx{r-1}{r}}\norm{h}{L^r(\R_{0},\R)}\right)^p$.
  Note that~$0<\overline\e<\widetilde\e\le\e$.
 \end{lemma}
\begin{proof}
The linearization of system~\eqref{dyn-odewh0} around a 
constant function~${\widehat w}$, ${\widehat w}(t)={\widehat w}(0)\in\R$ for all~$t\in\R$, reads 
\begin{align}\label{Linwh0}
 \dot z &=-\bigl(\breve C_1 -\breve C_2(p+1)\norm{\widehat w}{\R}^p\bigr) z ,\qquad
 z(0)=z_0\in\R,
 \end{align}
 which is exponentially stable if~$\breve C_1 >\breve C_2(p+1)\norm{\widehat w}{\R}^p$.
 That is, denoting the solution of~\eqref{Linwh0} by
 \[
  z(t)\eqqcolon \ZZ_{\widehat w}(t;0)z_0=\ZZ_{\widehat w}(t;s)\ZZ_{\widehat w}(s;0)z_0
  =\ZZ_{\widehat w}(t;s)z(s),\qquad t\ge s\ge0,\quad z_0\in\R,\quad {\widehat w}\in\R,
 \]
 we have that, with~$z(s)=z_1\in\R$,
 \begin{equation}\label{z-stab}
 \norm{\ZZ_{\widehat w}(t;s)z_1}{\R}= \ex^{-\widehat\e (t-s)} \norm{z_1}{\R},\quad\mbox{for all}\quad t\ge s\ge0, 
 \quad\mbox{with}\quad\widehat\e\coloneqq \breve C_1 -\breve C_2(p+1)\norm{\widehat w}{\R}^p.
 \end{equation}

 Let us also denote the solutions of systems~\eqref{dyn-odewh0} and~\eqref{dyn-odewh}, for~$t\ge s\ge0$, respectively by
 \begin{align*}
  w^0(t)\eqqcolon w^0(t;s,w_1),\qquad t\ge s\ge 0,\quad w^0(s)=w_1\in\R,\\
  w^h(t)\eqqcolon w^h(t;s,w_1),\qquad t\ge s\ge 0,\quad w^h(s)=w_1\in\R.
 \end{align*}
 Notice that by the assumption~\eqref{bar.eps} the initial condition $w_0$ satisfies
 \begin{equation*}
  0\le \norm{w_0}{\R}<\left(\fractx{\breve C_1}{\breve C_2}\right)^\frac1{p},
 \end{equation*}
which due to Proposition~\ref{P:odeh0} implies that~$w^0(t)$ is defined for all~$t\ge0$ and satisfies~\eqref{est.wh0}.
 We also know that~$w^h(t)$ will be defined for $t\ge 0$ in a maximal time interval, say
 for $t\in(0,\tau^h)$ with~$\tau^h>0$. We show now
 that~$\tau^h=+\infty$. Indeed if~$\tau^h\ne+\infty$ then we would have that
 \begin{equation}\label{blowup-wh}
  \lim_{t\nearrow\tau^h}\norm{w^h(t)}{\R}=+\infty. 
 \end{equation}
 Thus we want to show that~\eqref{blowup-wh} does not hold with (finite)~$\tau^h\in\R_{0}$.
 Let us fix an arbitrary~$\tau_1\in(0,\tau^h)$, then both
 solutions remain bounded in~$[0,\tau_1]$. That is, for a suitable large enough~$\rho>0$,
 \[
 \{w^0(t),w^h(t)\}\in(w_0-\rho,w_0+\rho), \quad\mbox{for all}\quad t\in[0,\tau_1].
 \]

 From~\cite[Lem.~3]{Brauer66}, since~\eqref{Linwh0} is the linearization of~\eqref{dyn-odewh0}, we know that we can write
 \begin{align*}
  w^h(t;0,w_0)&=w^0(t;0,w_0)+\int_{0}^t \ZZ_{w^h(s;0,w_0)}(t;s)h(s)\,\ed s,\qquad t\in[0,\tau_1].
 \end{align*}
Next we prove that  we actually have
\begin{equation}\label{norm-tau1}
 \norm{w^h(\tau_1)}{\R}\le\norm{w_0}{\R}+(\fractx{r}{r-1}\overline\e)^{-\fractx{r-1}{r}}\norm{h}{L^r(\R_{0},\R)}
 .
 \end{equation}
 For this purpose, let~$h\ne0$ and suppose that there exists~$\tau_2\in(0,\tau_1)$ such that
 \begin{subequations}\label{n-tau2}
  \begin{align}
 \norm{w^h(\tau_2)}{\R}&=\norm{w_0}{\R}+(\fractx{r}{r-1}\overline\e)^{-\fractx{r-1}{r}}\norm{h}{L^r(\R_{0},\R)}
 ,\label{n-tau2a}\\
 \norm{w^h(t)}{\R}&<\norm{w_0}{\R}+(\fractx{r}{r-1}\overline\e)^{-\fractx{r-1}{r}}\norm{h}{L^r(\R_{0},\R)}
 ,\quad\mbox{for all}\quad t\in[0,\tau_2).
 \end{align}
 \end{subequations}

 From~\eqref{z-stab}, we find that
 \begin{align}\label{norm-tau2}
  \norm{w^h(\tau_2;0,w_0)}{\R}&\le \ex^{-\e\tau_2}\norm{w_0}{\R}+\int_{0}^{\tau_2} \ex^{-\widetilde\e(\tau_2-s)}\norm{h(s)}{\R}\,\ed s
  \le \ex^{-\e\tau_2}\norm{w_0}{\R}+(\fractx{r}{r-1}\widetilde\e)^{-\fractx{r-1}{r}}\norm{h}{L^r((0,\tau_2),\R)},
 \end{align}
 which combined with~\eqref{n-tau2a} and with the fact that~$\overline\e<\widetilde\e$, gives
 us~$(\fractx{r}{r-1}\overline\e)^{-\fractx{r-1}{r}}>(\fractx{r}{r-1}\widetilde\e)^{-\fractx{r-1}{r}}$ and
 \[
 \norm{w_0}{\R}-\ex^{-\e\tau_2}\norm{w_0}{\R}
 \le-(\fractx{r}{r-1}\overline\e)^{-\fractx{r-1}{r}}\left(\norm{h}{L^r(\R_{0},\R)}-\norm{h}{L^r((0,\tau_2),\R)}\right),
 \]
which in turn implies
\[
w_0=0 \quad\mbox{and}\quad\norm{h}{L^r(\R_{\tau_2},\R)}=0.
\]
That is,~$h$ vanishes for~$t>\tau_2$ and we obtain~$w^h(\tau_1;\tau_2,w^h(\tau_2))=w^0(\tau_1;\tau_2,w^h(\tau_2))$, and
\[
 \norm{w^h(\tau_1;\tau_2,w^h(\tau_2))}{\R}
 \le \ex^{-\e(\tau_1-\tau_2)}\norm{w^h(\tau_2)}{\R}
 \le\norm{w_0}{\R}+(\fractx{r}{r-1}\overline\e)^{-\fractx{r-1}{r}}\norm{h}{L^r(\R_{0},\R)}
 .
\]

Therefore, with~$h\ne0$, if there is~$\tau_2\in(0,\tau_1)$ satisfying~\eqref{n-tau2}, then~\eqref{norm-tau1} holds true.
Of course, if there is no such~$\tau_2$ then necessarily~\eqref{norm-tau1} holds true.
 
Since~\eqref{norm-tau1} holds true for arbitrary~$\tau_1<\tau^h$, and
since~$\norm{w_0}{\R}+(\fractx{r}{r-1}\overline\e)^{-\fractx{r-1}{r}}\norm{h}{L^r(\R_{0},\R)}$
is independent of~$\tau_1$. it follows that necessarily~\eqref{blowup-wh} cannot hold with~$\tau^h\in\R$.
Therefore~$\tau^h=+\infty$, and \eqref{norm-tau1} and~\eqref{norm-tau2} hold true for all~$\tau_1\ge 0$ and all~$\tau_2\ge0$.
That is, the estimates in~\eqref{est.wh0a} are satisfied.
   \qed\end{proof}

\subsection{Proof of Theorem~\ref{T:wsol_QqV}}\label{sS:ProofT:wsol_QqV}
We can show the existence of the solution as a weak limit of Galerkin approximations of the system,
following a standard argument.
By taking the scalar product, in~$H$, with $2AQ$ in~\eqref{sys-split-Q}, we obtain 
\begin{align*}
\frac{\ed}{\ed t}\norm{Q}{V}^2&=-2 \norm{Q}{\D(A)}^2
-2\Bigl( P_{E_{\M_\sigma}^\perp}^{U_{\#\M_\sigma}}\Bigl(A_{\rm rc}(t)Q+\NN(t,q,Q)\Bigr),AQ\Bigr)_H  +2(\MM(q),AQ)_H,\qquad t>0.
\end{align*}

Using Assumption~\ref{A:suffalpha}, we fix
 a quadruple~$\gamma=(\gamma_1,\gamma_2,\gamma_3,\overline\e)\in\R_0^4$ satisfying~\eqref{suffalpha}.
 From Assumption~\ref{A:A1},
 \begin{subequations}\label{Young-qQ}
 \begin{align}
  2\Bigl( P_{E_{\M_\sigma}^\perp}^{U_{\#\M_\sigma}}A_{\rm rc}(t)Q,AQ\Bigr)_H&\le\gamma_1\norm{Q}{\D(A)}^2
  +\gamma_1^{-1}\norm{P_{E_{\M_\sigma}^\perp}^{U_{\#\M_\sigma}}}{\LL(H)}^2 C_{\rm rc}^2\norm{Q}{V}^2,\\
  2(\MM(q),AQ)_H&\le\gamma_2\norm{Q}{\D(A)}^2+\gamma_2^{-1}\norm{\MM(q)}{H}^{2}
  \end{align}
 and, from~\eqref{NNQAQ}, with~$\widehat\gamma_0=\gamma_3$, we find
 \begin{align}
 &2\Bigl(  P_{E_{\M_\sigma}^\perp}^{U_{\#\M_\sigma}}\NN(t,q,Q),AQ\Bigr)_{H} \le \gamma_3 \norm{Q}{\D(A)}^{2}\notag\\
  &\hspace*{3em}
  +\gamma_3^{-\frac{2}{1-\|\zeta_2+\delta_2\|} }
  \overline C_{\NN2} \left( 1 +\norm{q}{V}^{2\|\frac{\zeta_1+\delta_1}{1-\zeta_2-\delta_2}\|}\right)
 \left(1+\norm{q}{\D(A)}^{\|\frac{\zeta_2}{1-\delta_2}\|}\right)
 \left( 1+\norm{Q}{V}^{2\|\frac{\zeta_1+\delta_1}{1-\zeta_2-\delta_2}\|-2}
 \right)\norm{Q}{V}^{2}.
 \end{align}
 \end{subequations}
Hence,
 the estimates in~\eqref{Young-qQ} lead us to 
 \begin{align}\label{dynnormQDA}
\frac{\ed}{\ed t}\norm{Q}{V}^2
&\le-\mathfrak a_0 \norm{Q}{\D(A)}^2
+\left(\mathfrak a_1+\mathfrak a_2\mathfrak q(1+\norm{Q}{V}^{2\mathfrak p})
\right)\norm{Q}{V}^{2}+\mathfrak h,
\end{align}
with~$\mathfrak a_0$, $\mathfrak a_1$, $\mathfrak a_2$, $\mathfrak q$, $\mathfrak p$,
and~$\mathfrak h$ as in~\eqref{frak-apr}.

Recall that for~$(q,Q)\in E_{\M_\sigma}\times E_{\M_\sigma}^\perp$,
we have~$\norm{Q}{\D(A)}^2\ge\alpha_{\M_{\sigma+}}\norm{Q}{V}^2$ and~$\norm{q}{\D(A)}^2\le\alpha_{\M_\sigma}\norm{q}{V}^2$,
see~\eqref{VDAalphaM}.
Thus,
\begin{align*}
 \frac{\ed}{\ed t}\norm{Q}{V}^2
 &\le-\bigl(\mathfrak a_0\alpha_{\M_{\sigma+}}-\mathfrak a_1-\mathfrak a_2\mathfrak q
 -\mathfrak a_2\mathfrak q\norm{Q}{V}^{2\mathfrak p}\bigr) \norm{Q}{V}^2
 +\mathfrak h,
 \end{align*}
and we conclude that~$\norm{Q(t)}{V}^2\le w(t)$, where $w$ solves
\begin{align}\label{dyn-odew}
 \dot w
 &=-\bigl(\mathfrak a_0\alpha_{\M_{\sigma+}}-\mathfrak a_1-\mathfrak a_2\mathfrak q
 -\mathfrak a_2\mathfrak q\norm{w}{\R}^{\mathfrak p}\bigr) w +\mathfrak h
,\qquad
 w(0)=\norm{Q_0}{V}^2.
 \end{align}

 Note that, from Assumption~\ref{A:suffalpha},
 we have~$\mathfrak a_0\alpha_{\M_{\sigma+}}-\mathfrak a_1-\mathfrak a_2\mathfrak q-\overline\e>0$ and
 \begin{align*}
   \norm{Q_0}{V}^2+(\fractx{r}{r-1}\overline\e)^{-\frac{r-1}{r}}\norm{\mathfrak h}{L^r(\R_{0},\R)}&
  <\left(\frac{\mathfrak a_0\alpha_{\M_{\sigma+}}-\mathfrak a_1-\mathfrak a_2\mathfrak q
  -\overline\e}{(\mathfrak p+1)\mathfrak a_2\mathfrak q}\right)
  ^\frac1{\mathfrak p},
  \end{align*}
  which shows that the requirements in Lemma~\ref{L:odeh} are fulfilled with 
\begin{align*}
 &w_0=\norm{Q_0}{V}^2,\qquad p=\mathfrak p,
 \qquad h=\mathfrak h,\qquad
 C_1=\mathfrak a_0\alpha_{\M_{\sigma+}}-\mathfrak a_1-\mathfrak a_2\mathfrak q,\qquad\mbox{and}\qquad
 C_2=\mathfrak a_2\mathfrak q.
\end{align*}
Thus, with $\e\coloneqq C_1 -C_2\norm{w_0}{\R}^{p}$ and
$\widetilde\e=C_1-C_2(p+1)(\norm{w_0}{\R}+(\fractx{r}{r-1}\overline\e)^{-\frac{r-1}{r}}\norm{h}{L^r(\R_{0},\R)}
)^p$, we arrive at~\eqref{sys-Q0-est}.

We have just proven that~\eqref{sys-Q0-est} holds true, for any given strong solution. 
The existence of a strong solution follows from the fact that the previous estimates hold true for Galerkin
approximations~$Q^N$ taking values in the finite-dimensional
space~$E_{\M_{\sigma}}^{\perp,N}\coloneqq P_{E_N}E_{\M_{\sigma}}^{\perp,N}= E_{\M_{\sigma}}^{\perp}\bigcap E_N$,
\[
 N\in\N,\qquad E_N\coloneqq\linspan\{e_i\mid i\in\{1,2,\dots,N\}\}
\]
and $P_{E_N}\colon H\to E_N$ is the orthogonal projection in~$H$ onto~$E_N$, which solve the finite-dimensional system
\begin{align*}
    \dot Q^N + P_{E_N}P_{E_{\M_\sigma}^\perp}^{U_{\#\M_\sigma}}\Bigl(AQ^N +A_{\rm rc}(t)Q^N+\NN(t,q,Q^N)\Bigr)
    &=P_{E_N}\MM(q),&\quad  Q^N(0)&=Q^N_0\coloneqq P_{E_N}Q_0.
\end{align*}

Let us fix an arbitrary~$s>0$. 
Hence, from (the analogous to)~\eqref{sys-Q0-est} we find~$\norm{Q^N}{L^\infty((0,s),V)}\le C_3$,
where~$C_3$ can be taken independent of~$N$
and~$s$. Then, by integrating~\eqref{dynnormQDA} we obtain that
 \begin{align*}
\norm{Q(s)}{V}^{2}+\mathfrak a_0\norm{Q^N}{L^2((0,s),\D(A))}^2&\le\norm{Q(0)}{V}^{2}
+\norm{\mathfrak h}{L^1((0,s),\R)}\\
&\quad+s\left((\mathfrak a_1+\mathfrak a_2\mathfrak q)\norm{Q^N}{L^\infty((0,s),V)}^2
+\mathfrak a_2\mathfrak q\norm{Q^N}{L^\infty((0,s),V)}^{2\mathfrak p+2}\right).
 \end{align*}
Since~$\mathfrak a_0>0$, because~$\mathfrak a_0>\frac{\mathfrak a_1+\mathfrak a_2\mathfrak q}{\alpha_{\M_{\sigma+}}}>0$, 
we conclude that~$\norm{Q^N}{L^2((0,s),\D(A))}^2\le C_4$, where $C_4$ can be taken independent of~$N$.
Finally, from Assumption~\ref{A:A1}, ~\eqref{sys-split-Q}, ~\eqref{NNyAy}, and~$q\in L^\infty(\R_{0},\D(A))$, it follows that
 \begin{align*}
    \norm{\dot Q^N}{H} &\le C_5 \left(1+\norm{Q^N}{\D(A)} +\norm{\MM(q)}{H}\right),
\end{align*}
from which we have that $\norm{\dot Q^N}{L^2((0,s),H)}^2\le C_7$, with $C_7$ independent of~$N$.
Thus, we can conclude the existence of the weak limit~$Q^\infty$ of a suitable subsequence of~$Q^N$ (that we still denote~$Q^N$):
 \begin{align*}
  Q^N\xrightharpoonup[L^2((0,s),\D(A))]{} Q^\infty\qquad\mbox{and}\qquad
  \dot Q^N\xrightharpoonup[L^2((0,s),H)]{} \dot Q^\infty,
 \end{align*}
then we can assume the strong convergence $Q^N\xrightarrow[L^2((0,s),V)]{} Q^\infty$,
because~$W((0,s),\D(A),H)\xhookrightarrow{\rm c}L^2((0,s),V)$,
see~\cite[Ch.~3, Sect..~2.2, Thm.~2.1]{Temam01}.  
Next, we show that
\begin{align}
&\qquad\dot Q^N + P_{E_N}P_{E_{\M_\sigma}^\perp}^{U_{\#\M_\sigma}}\Bigl(AQ^N +A_{\rm rc}(t)Q^N+\NN(t,q,Q^N)\Bigr)
-P_{E_N}\MM(q)\notag\\
\xrightharpoonup[L^2((0,s),H)]{}&\qquad\dot Q^\infty 
+ P_{E_{\M_\sigma}^\perp}^{U_{\#\M_\sigma}}\Bigl(AQ^\infty +A_{\rm rc}(t)Q^\infty+\NN(t,q,Q^\infty)\Bigr)-\MM(q)\,
\label{Gal.limit}
\end{align}
from which we can conclude that the limit~$Q^\infty$ solves~\eqref{sys-split-Q}.
We know that~$\dot Q^N\xrightharpoonup[L^2((0,s),H)]{} \dot Q^\infty$, and 
\[
AQ^N +A_{\rm rc}(t)Q^N
\xrightharpoonup[L^2((0,s),H)]{} AQ^\infty +A_{\rm rc}(t)Q^\infty
\]
follows straightforwardly. Since $q\in L^2((0,s),\D(A))$ is fixed, we also have~$P_{E_N}\MM(q)
\xrightarrow[L^2((0,s),H)]{} \MM(q)$.

Hence, since~$P_{E_N}P_{E_{\M_\sigma}^\perp}^{U_{\#\M_\sigma}}\in\LL(H)$, to show~\eqref{Gal.limit} it remains to show
\[
\NN(t,q,Q^N)\xrightharpoonup[L^2((0,s),H)]{} \NN(t,q,Q^\infty).
\]
Actually, we  have strong convergence
$\NN(t,q,Q^N)
\xrightarrow[L^2((0,s),H)]{} \NN(t,q,Q^\infty)$. Indeed,
from 
$Q^\infty\in W((0,s),\D(A),H)\xhookrightarrow{}C([0,s],V)$ and the fact that the
sequence~$Q^N$ is uniformly bounded in the space $W((0,s),\D(A),H)$, from Assumption~\ref{A:NN} and
the H\"older inequality, 
with~$y_1=q+Q^N$ and~$y_2=q+Q^\infty$, it follows that, with~$D^N\coloneqq Q^N-Q^\infty$, and since~$\delta_{2j}+\zeta_{2j}<1$,
\begin{align*}
 &\norm{\NN(t,q+Q^N)-\NN(t,q+Q^\infty)}{L^2((0,s),H)} \notag\\
&\hspace*{.8em}\le C_{\NN}\textstyle\sum\limits_{j=1}^n\norm{\left( \norm{y_1}{V}^{\zeta_{1j}}\norm{y_1}{\D(A)}^{\zeta_{2j}}
 +\norm{y_2}{V}^{\zeta_{1j}}\norm{y_2}{\D(A)}^{\zeta_{2j}}\right)
 \norm{D^N}{\D(A)}^{\delta_{2j}}}{L^\frac{2}{\zeta_{2j}+\delta_{2j}}((0,s),\R)}
 \norm{\norm{D^N}{V}^{\delta_{1j}}}{L^\frac{2}{1-\zeta_{2j}-\delta_{2j}}((0,s),\R)} \\
 &\hspace*{.8em}\le C_8\textstyle\sum\limits_{j=1}^n\textstyle\sum\limits_{k=1}^2\norm{\norm{y_k}{\D(A)}^{\zeta_{2j}}
  \norm{D^N}{\D(A)}^{\delta_{2j}}}{L^\frac{2}{\zeta_{2j}+\delta_{2j}}((0,s),\R)}
 \norm{D^N}{L^{2}((0,s),V)}^{1-\zeta_{2j}-\delta_{2j}}.
\end{align*}
From $\delta_{1j}+\delta_{2j}\ge1$, it follows that~$\frac{2\delta_{1j}}{1-\zeta_{2j}-\delta_{2j}}\ge 2$
and~$\norm{\norm{D^N}{V}^{\delta_{1j}}}{L^\frac{2}{1-\zeta_{2j}-\delta_{2j}}((0,s),\R)}
\le C_9\norm{D^N}{L^{2}((0,s),V)}^{1-\zeta_{2j}-\delta_{2j}}$,
because~$D^N$ is uniformly bounded in~$L^{\infty}((0,s),V)$.
Observe also that by the Young inequality
\[
 \norm{y_k}{\D(A)}^{\frac{2\zeta_{2j}}{\zeta_{2j}+\delta_{2j}}}
  \norm{D^N}{\D(A)}^{\frac{2\delta_{2j}}{\zeta_{2j}+\delta_{2j}}}
  \le
  \norm{y_k}{\D(A)}^{2}
  +\norm{D^N}{\D(A)}^{2},
\]
which leads us to
\begin{align*}
\norm{\norm{y_k}{\D(A)}^{\zeta_{2j}}
  \norm{D^N}{\D(A)}^{\delta_{2j}}}{L^\frac{2}{\zeta_{2j}+\delta_{2j}}((0,s),\R)}
  &\le \norm{\norm{y_k}{\D(A)}^{\frac{2\zeta_{2j}}{\zeta_{2j}+\delta_{2j}}}
  \norm{D^N}{\D(A)}^{\frac{2\delta_{2j}}{\zeta_{2j}+\delta_{2j}}}}{L^1((0,s),\R)}^\frac{\zeta_{2j}+\delta_{2j}}{2}\\
  &\le \left(\norm{y_k}{L^2((0,s),\D(A))}^{2}+\norm{D^N}{L^2((0,s),\D(A))}^2\right)^{\frac{\zeta_{2j}+\delta_{2j}}{2}}
  \end{align*}
and consequently to
\begin{align*}
 &\norm{\NN(t,q+Q^N)-\NN(t,q+Q^\infty)}{L^2((0,s),H)} \le C_{10}\norm{D^N}{L^{2}((0,s),V)}^{1-\zeta_{2j}-\delta_{2j}}
 \xrightarrow[N\to+\infty]{} 0.
\end{align*}

To finish the proof of Theorem~\ref{T:wsol_QqV}, it remains to prove the uniqueness  in~$W((0,s),\D(A),H)$.
For this purpose, observe
that given two solutions~$Q_1$ and~$Q_2$ in~$W((0,s),\D(A),H)$, we find that $G=Q_2-Q_1\in W((0,s),\D(A),H)$ solves
\begin{align*}
    \dot G + P_{E_{\M_\sigma}^\perp}^{U_{\#\M_\sigma}}\Bigl(AG +A_{\rm rc}(t)G+\NN(t,q,Q_2)-\NN(t,q,Q_1)\Bigr)
    &=0,&\quad  G(0)&=0\in E_{\M_\sigma}^\perp.
\end{align*}
Thus, from~\eqref{NNyAy} with~$\widehat\gamma_0=1$, and the Young inequality,
with~$y_1=q+Q^N$ and~$y_2=q+Q^\infty$,  it follows
\begin{align*}
 &2\Bigl( P_{E_{\M_\sigma}^\perp}^{U_{\#\M_\sigma}}\left(\NN(t,y_1)-\NN(t,y_2)\right),AG\Bigr)_{H} \notag\\
&\hspace*{2em}\le  \norm{G}{\D(A)}^{2}
  +
 \overline C_{\NN1}\sum\limits_{j=1}^n\left( \norm{y_1}{V}^{\frac{2\zeta_{1j}}{1-\delta_{2j}-\zeta_{2j}}}
 +\norm{y_2}{V}^{\frac{2\zeta_{1j}}{1-\delta_{2j}-\zeta_{2j}}}
 + \norm{y_1}{\D(A)}^2+\norm{y_2}{\D(A)}^2
 \right)\norm{G}{V}^{\frac{2\delta_{1j}}{1-\delta_{2j}}}\\
 &\hspace*{2em}=  \norm{G}{\D(A)}^{2}
  +
 \Phi(t)\norm{G}{V}^{2} ,
\end{align*}
with~$\Phi(t)\coloneqq\overline C_{\NN1}\sum\limits_{j=1}^n\left( \norm{y_1}{V}^{\frac{2\zeta_{1j}}{1-\delta_{2j}-\zeta_{2j}}}
 +\norm{y_2}{V}^{\frac{2\zeta_{1j}}{1-\delta_{2j}-\zeta_{2j}}}
 + \norm{y_1}{\D(A)}^2+\norm{y_2}{\D(A)}^2
 \right)\norm{G}{V}^{\frac{2\delta_{1j}}{1-\delta_{2j}}-2}$. Recall that $\frac{2\delta_{1j}}{1-\delta_{2j}}\ge2$.

By using Assumption~\ref{A:A1} and~\eqref{Young-qQ} with~$\gamma_1=1$, we find
\begin{align*}
    \frac{\ed}{\ed t}\norm{G}{V}^2 
    &\le -2\norm{G}{\D(A)}^2+\norm{G}{\D(A)}^{2}+ \Phi(t)\norm{G}{V}^{2}+\norm{G}{\D(A)}^{2}
    + \norm{P_{E_{\M_\sigma}^\perp}^{U_{\#\M_\sigma}}}{\LL(H)}^2C_{\rm rc}^2\norm{G}{V}^{2}\le\Phi_2(t)\norm{G}{V}^{2}.
   \end{align*} 
with~$\Phi_2(t)\coloneqq\norm{P_{E_{\M_\sigma}^\perp}^{U_{\#\M_\sigma}}}{\LL(H)}^2C_{\rm rc}^2+\Phi(t)$. 
Now, we see that~$\Phi_2$ is integrable on~$(0,s)$, due
to~$q\in L^\infty([0,s],\D(A))$ and $\{Q_1,Q_2\}\subset C([0,s],V)\textstyle\bigcap
 L^2((0,s),\D(A))$. Hence, by the Gronwall inequality,
\[
 \norm{G(t)}{V}^2\le\ex^{\int_{0}^t\Phi_2(\tau)\,\ed \tau}\norm{G(0)}{V}^2=0, \quad\mbox{for all}\quad t\in[0,s].
\]
That is, $G=0$ and~$Q_2=Q_1+G=Q_1$. We have shown that for arbitrary~$s> 0$ there exists one,
and only one, strong solution
$Q\in W((0,s),\D(A),H)$ for~\eqref{sys-split-Q}. In other words,
there exists one, and only one, global solution~$Q\in W_{\rm loc}(\R_0,\D(A),H)$ for~\eqref{sys-split-Q}. This finishes
the proof of Theorem~\ref{T:wsol_QqV}.
\qed

\subsection{Proof of Theorem~\ref{T:main.1pair}}
Theorem~\ref{T:main.1pair} 
follows from the following Theorem~\ref{T:main.1pairF}.\qed

\begin{theorem}\label{T:main.1pairF}
If Assumptions~\ref{A:A0sp}--\ref{A:NN} and~\ref{A:FFM}--\ref{A:suffalpha} are satisfied, then 
system~\eqref{sys-split-qQ} is exponentially stable.  The solution~$y=q+Q$, satisfies~$\norm{y(t)}{V}
 \le \overline C\ex^{-\frac{\mu}{2}t}\norm{y(0)}{V}$, for all~$t\ge0$,
 where~$\mu<\min\{\widetilde\e,2\lambda\}$ and~$\widetilde\e$ is
 as in~\eqref{eps.tildeeps}. Furthermore, ~$\overline C=\ovlineC{n,\|P\|_\LL,C_{\rm rc},C_\NN,\frac{1}{\widetilde\e},C_{1q},
   \norm{q(0)}{V},\frac{1}{\gamma_3},\frac{1}{\widetilde\e-\mu},\alpha_{\M_\sigma}}
   $.
 \end{theorem}
 \begin{proof}
We have $q\in L^\infty(\R_{0},\D(A))$ because $q\in L^\infty(\R_{0},H)$ and~$E_{\M_\sigma}$ is 
finite dimensional,~$E_{\M_\sigma}\subset\D(A)\subset H$.
 By Theorem~\ref{T:wsol_QqV},  we conclude that~$Q$ satisfies, for all $t\ge 0$,
\begin{align*}
 \norm{Q(t)}{V}^2&\le \ex^{-\e t} \norm{Q(0)}{V}^2
 + \int_{0}^{t} \ex^{-\widetilde\e(t-s)}\gamma_3^{-1}\norm{\MM(q(s))}{H}^2\,\ed s.
   \end{align*}
Hence we obtain,  
   using Assumptions~\ref{A:NN} and~\ref{A:FFM},
 \begin{align*}
 &\quad\norm{Q(t)}{V}^2-\ex^{-\e t} \norm{Q(0)}{V}^2\\
  &\le
  \ovlineC{n,\|P\|_\LL,C_{\rm rc},C_\NN}\gamma_3^{-1}\int_{0}^{t} \ex^{-\widetilde\e(t-s)}
  \biggl(\norm{Aq(s)}{H}^2+\norm{q(s)}{V}^{2}
 +\textstyle\sum\limits_{j=1}^n\norm{q(s)}{V}^{2( \delta_{1j}+\zeta_{1j} )}\norm{q(s)}{\D(A)}^{2(\delta_{2j}+\zeta_{2j})}\\
&\hspace*{25em}+\norm{\FF_{\M_\sigma}(q(s)) }{H}^2\biggr)\,\ed s\\
  &\le  \ovlineC{n,\|P\|_\LL,C_{\rm rc},C_\NN,C_{q0},C_{q3}}\gamma_3^{-1}
  \int_{0}^{t} \ex^{-\widetilde\e(t-s)}\biggl(\alpha_{\M_\sigma}+1
    +\textstyle\sum\limits_{j=1}^n\alpha_{\M_\sigma}^{\delta_{2j}+\zeta_{2j}}
    \norm{q(s)}{V}^{2( \delta_{1j}+\zeta_{1j}+\delta_{2j}+\zeta_{2j}-1)}\\
  &\hspace*{25em}+\alpha_{\M_\sigma}^{\xi}
    \norm{q(s)}{V}^{2(\xi-1)}\biggr)\norm{q(s)}{V}^{2}\,\ed s
  \\
 &\le  \ovlineC{n,\|P\|_\LL,C_{\rm rc},C_\NN,\frac{1}{\widetilde\e},C_{0q},C_{1q},C_{3q},\norm{q(0)}{V},
 \alpha_{\M_\sigma}}\gamma_3^{-1}
 \norm{q(0)}{V}^{2}\int_{0}^{t} \ex^{-\widetilde\e(t-s)}\ex^{-2\lambda s}\,\ed s.
   \end{align*}

Through straightforward computations we can obtain, with~$\mu<\min\{\widetilde\e,2\lambda\}$, the estimates
\begin{align*}
\int_{0}^{t} \ex^{-\widetilde\e(t-s)}\ex^{-2\lambda s}\,\ed s
\le\int_{0}^{t} \ex^{-\widetilde\e(t-s)}\ex^{-\mu s}\,\ed s
&\le\norm{\widetilde\e-\mu}{\R}^{-1}\ex^{-\mu t},
\end{align*}
which leads us to
\[
\norm{Q(t)}{V}^2\le \ex^{-\e t} \norm{Q(0)}{V}^2+
  \widehat D\ex^{-\mu t}\norm{q(0)}{V}^{2}
\]
  with~$\widehat D =\ovlineC{n,\|P\|_\LL,C_{\rm rc},C_\NN,\frac{1}{\widetilde\e},C_{0q},C_{1q},C_{3q},
   \norm{q(0)}{V},\alpha_{\M_\sigma}}\gamma_3^{-1}
   \norm{\widetilde\e-\mu}{\R}^{-1}$.
Hence,~$\norm{y(t)}{V}^2=\norm{Q(t)}{V}^2+\norm{q(t)}{V}^2$ satisfies
    \begin{align}
 \norm{y(t)}{V}^2
 &\le \ex^{-\e t} \norm{Q(0)}{V}^2
 +\widehat D\ex^{-\mu t}\norm{q(0)}{V}^{2}
 +C_{q1}^2\ex^{-2\lambda t}\norm{q(0)}{V}^{2}\notag\\
 &\le \ex^{-\e t} \norm{Q(0)}{V}^2
 +(C_{q1}^2+\widehat D)\ex^{-\mu t}\norm{q(0)}{V}^{2}
 \le (1+C_{q1}^2+\widehat D)\ex^{-\mu t}\norm{y(0)}{V}^{2},\qquad t\ge0,\label{stable-s_0}
 \end{align}
which finishes the proof.
\qed\end{proof}

\subsection{Proof of Theorem~\ref{T:N.goal}}\label{sS:proofT:N.goal}
We show that Theorem~\ref{T:N.goal} follows as a corollary of Theorem~\ref{T:main.1pair}.
 Indeed, let us suppose we have a
sequence~$(U_{\#\M_\sigma},E_{\M_\sigma})_{M\in\N}$ so that~$C_P^M\coloneqq\norm{P_{E_{\M_\sigma}^\perp}^{U_{\#\M_\sigma}}}{\LL(H)}\le C_P$
and~$\fractx{\alpha_{\M_\sigma}}{\alpha_{\M_{\sigma+}}}\le\Lambda$, with~$C_P$ and~$\Lambda$
independent of~$M$.
Let us also fix~$\gamma=\overline\gamma\in\R_0^3$ so that~$\mathfrak a_0=\mathfrak a_0^{\overline\gamma}>0$, and fix also~$\overline\e>0$.

Recalling~\eqref{frak-apr} and~\eqref{MM}, we see
that~${\mathfrak a}_1^{\overline\gamma}$, ~${\mathfrak a}_2^{\overline\gamma}$, and~$\mathfrak h^{\overline\gamma}$, are the only
terms in~\eqref{suffalpha} depending on~$C_P^M$. 
However, these terms remain bounded if~$C_P^M$ does. Hence, defining
\begin{align*}
 \widetilde{\mathfrak a}_1^{\overline\gamma}&\coloneqq{\mathfrak a}_1^{\overline\gamma}(C_P)
 =\overline\gamma_1^{-1}C_P^2 C_{\rm rc}^2>\mathfrak a_1^{\overline\gamma}(C_P^M),\\
 \widetilde{\mathfrak a}_2^{\overline\gamma}&\coloneqq{\mathfrak a}_2^{\overline\gamma}(C_P)
 =\overline\gamma_3^{-\frac{2}{1-\|\zeta_{2}+\delta_{2}\|}}\overline{C}_{\NN2}(C_P)
  >\mathfrak a_2^{\overline\gamma}(C_P^M),\\
  \widetilde{\mathfrak h}^{\overline\gamma}&
   \coloneqq{\mathfrak h}^{\overline\gamma}(C_P)\ge{\mathfrak h}^{\overline\gamma}(C_P^M), 
\end{align*}
we observe that Assumption~\ref{A:suffalpha}, taking~$r=\mathfrak r\in(1,\frac{1}{\|\zeta_2+\eta_2\|})$ as
in Assumption~\eqref{A:FFM},
follows from
\begin{align}\label{suffalpha-bddM}
 \alpha_{\M_{\sigma+}}>\inf_{\begin{subarray}{c}
                       (\overline\gamma_1,\overline\gamma_2,\overline\gamma_3)\in\R_0^3,\\
                       \mathfrak a_0^{\overline\gamma}>0,\hspace*{.7em}\overline\e>0,\\
                       \mathfrak a_0^{\overline\gamma}\alpha_{\M_{\sigma+}}
                       -\widetilde{\mathfrak a}_1^{\overline\gamma}-\widetilde{\mathfrak a}_2^{\overline\gamma}\mathfrak q
                       -\overline\e>0,
                                 \end{subarray}}
\frac{1}{\mathfrak a_0^{\overline\gamma}}\left(\widetilde{\mathfrak a}_1^{\overline\gamma}
+\widetilde{\mathfrak a}_2^{\overline\gamma}\mathfrak q+\overline\e
+(\mathfrak p+1)\widetilde{\mathfrak a}_2^{\overline\gamma}\mathfrak q\left(\norm{Q_0}{V}^2
  +(\fractx{\mathfrak r}{\mathfrak r-1}\overline\e)^{-\frac{\mathfrak r-1}{\mathfrak r}}
  \norm{\widetilde{\mathfrak h}^{\overline\gamma}}{L^{\mathfrak r}(\R_{0},\R)}
  \right)^{\mathfrak p}\right).
  \end{align}
Note that for~$M$ large enough it follows
that~$\mathfrak a_0^{\overline\gamma}\alpha_{\M_{\sigma+}}-\widetilde{\mathfrak a}_1^{\overline\gamma}
-\widetilde{\mathfrak a}_2^{\overline\gamma}\mathfrak q-\overline\e>0$.
 Now,
 with~$y_0=q_0+Q_0$, 
 \begin{subequations}\label{suffalpha-bddM3}
 \begin{align}
   &\lim_{M\to+\infty}\fractx{\widetilde{\mathfrak a}_1^{\overline\gamma}+\overline\e}{\alpha_{\M_{\sigma+}}}
                       =0,\\
   &\lim_{M\to+\infty}\fractx{ \mathfrak a_2^{\overline\gamma}\mathfrak q }{\alpha_{\M_{\sigma+}}}
      \le \ovlineC{C_P,\norm{q_0}{V}}\lim_{M\to+\infty}
      \fractx{ \alpha_{\M_{\sigma}}^{\|\frac{\zeta_2}{1-\delta_2}\|} }{\alpha_{\M_{\sigma+}}}
      \le \ovlineC{C_P,\norm{q_0}{V}}\lim_{M\to+\infty}
      \Lambda^{\|\frac{\zeta_2}{1-\delta_2}\|}\alpha_{\M_{\sigma+}}^{\|\frac{\zeta_2}{1-\delta_2}\|-1}=0,\quad                    
   \intertext{since~$\|\frac{\zeta_2}{1-\delta_2}\|<1$, and}
  &\lim_{M\to+\infty}\fractx{(\mathfrak p+1)\mathfrak a_2^{\overline\gamma}\mathfrak q\left(\norm{Q_0}{V}^2
  +(\fractx{\mathfrak r}{\mathfrak r-1}\overline\e)^{-\frac{\mathfrak r-1}{\mathfrak r}}\norm{{\mathfrak h}^{\overline\gamma}}{L^{\mathfrak r}(\R_{s_0},\R)}
  \right)^{\mathfrak p}}{\alpha_{\M_{\sigma+}}}
  \le\ovlineC{C_P,\norm{q_0}{V}}\lim_{M\to+\infty}
  \fractx{\mathfrak q\norm{\mathfrak h^{\overline\gamma}}{{L^{\mathfrak r}(\R_{0},\R)}}^{\mathfrak p}}
  {\alpha_{\M_{\sigma+}}}.
\end{align} 

From~\eqref{normhLr} we have that
\begin{align*}
    \norm{\mathfrak h^{\overline\gamma}}{{L^{\mathfrak r}(\R_{0},\R)}}
      &\le\ovlineC{n,C_P,C_{\rm rc},C_\NN,C_{q1},C_{q2},C_{q3},\frac1{\lambda},\norm{q(0)}{V}}
           \left(1+\alpha_{\M_\sigma}^{\beta_2}+\alpha_{\M_\sigma}^{\mathfrak r\eta_2\|\zeta_{2}+\delta_{2}\|}
      \right)\norm{q(0)}{V}^{2},
     \end{align*}
which leads us to
  \begin{align}   
 \lim_{M\to+\infty}
  \fractx{\mathfrak q\norm{\mathfrak h^{\overline\gamma}}{{L^2(\R_{0},\R)}}^{\mathfrak p}}
  {\alpha_{\M_{\sigma+}}}
  &\le\overline C\lim_{M\to+\infty}
  \left(\alpha_{\M_{\sigma+}}^{\|\frac{\zeta_2}{1-\delta_2}\|-1}
  \left(1+\alpha_{\M_\sigma}^{\beta_2\mathfrak p}
      +\alpha_{\M_\sigma}^{\mathfrak r\eta_2\|\zeta_{2}+\delta_{2}\|\mathfrak p}\right)\right)=0,
\end{align}     
  since, by Assumption~\ref{A:FFM}, we have~$\max\left\{\|\frac{\zeta_2}{1-\delta_2}\|-1+\beta_2\mathfrak p,
  \|\frac{\zeta_2}{1-\delta_2}\|-1+\mathfrak r\eta_2\|\zeta_{2}+\delta_{2}\|\mathfrak p\right\}<0$.
\end{subequations} 
 Therefore, from the inequalities in~\eqref{suffalpha-bddM3} we can conclude that necessarily~\eqref{suffalpha-bddM} holds
 true for large enough~$M$, with
  \begin{equation}\label{depM_Vnorm}
 M=\ovlineC{n,C_P,C_{\rm rc},C_\NN,C_{q1},C_{q2},C_{q3},\frac1{\lambda},\norm{q(0)}{V},\norm{Q(0)}{V}}. 
 \end{equation}
 In particular, \eqref{depM_Vnorm}
 means that~$M$ increases (or may increase) with the norm~$\norm{y(0)}{V}$, of the initial condition~$y(0)=q(0)+Q(0)$,
 but it also means that, for arbitrary given~$R>0$, ~$M$ can
 be taken the same for all initial initial conditions in the ball
 $\{z\in V\mid \norm{z}{V}\le R\}$.
 \qed

\subsection{Boundedness of the control}
In applications, besides the existence of a stabilizing feedback, it is important that the total
``energy'' spent to stabilize the system is finite.
We show here that the control given by our nonlinear feedback operator in~\eqref{FeedKy-Nx} is indeed bounded, with a bound
increasing with the norm of the initial condition.
Note that~\eqref{sys_F} and~\eqref{sys_F-dec} are the same system.

\begin{theorem}\label{T:main-ct}
 Let~$\mathbf u(t)\coloneqq\mathfrak F(t,y(t))=\KKK_{U_{\#\M_\sigma}}^{\FF_{\M_\sigma},\NN}(t,y)$
 be the control input given by the
 operator~\eqref{FeedKy-Nx}
 stabilizing system~\eqref{sys_F}, with initial condition~$y_0$ as in Theorem~\ref{T:main.1pair}. Then 
 \begin{align*}
&\sup_{t\ge0}\norm{\KKK_{U_{\#\M_\sigma}}^{\FF_{\M_\sigma},\NN}(t,z)}{H}
\le\ovlineC{\dnorm{P}{\LL},\norm{z}{V},C_{\rm rc},C_{\NN},C_{q_0}}\left(1+\norm{z}{\D(A)}^\xi\right),
\quad\mbox{for all}\quad z\in\D(A),\qquad\mbox{and}\\
&\norm{\mathbf u}{L^{2\mathfrak r}(\R_{0},H)}
\le \ovlineC{n,\dnorm{P}{\LL},C_{\rm rc},C_{\NN},\frac1{\widetilde\e},C_{q0},C_{q1},C_{q2},C_{q3},\norm{q_0}{V},\alpha_{\M_\sigma}}
\norm{q_0}{V}
\end{align*}
with~$1<\mathfrak r< \tfrac{1}{\|\zeta_2+\delta_2\|}$ and~$\xi$ as in~Assumption~\ref{A:FFM},
and~$\dnorm{P}{\LL}\coloneqq\norm{P_{U_{\#\M_\sigma}}^{E_{\M_\sigma}^\perp}}{\LL(H)}$.
\end{theorem}
\begin{proof}
Recalling~\eqref{FeedKy-Nx}, the boundedness of~$\KKK_{U_{\#\M_\sigma}}^{\FF_{\M_\sigma},\NN}$ follows simply from
 \begin{align*}
& \norm{A}{\LL(\D(A),H)}=1,\quad
\norm{A_{\rm rc}(t)}{\LL(\D(A),H)}\le\norm{\Id}{\LL(\D(A),V)}\norm{A_{\rm rc}(t)}{\LL(V,H)}\le
\ovlineC{C_{\rm rc}}\norm{\Id}{\LL(\D(A),V)},\\
& \norm{\NN(t,z)}{H}=\ovlineC{C_\NN}\left(1+\norm{z}{V}^{\|\zeta_1+\delta_1\|}\right)
 \left(1+\norm{z}{\D(A)}^{\|\zeta_2+\delta_2\|}\right),\qquad
\norm{\FF_{\M_\sigma}(P_{E_{\M_\sigma}}z)}{H}\le C_{q_0}\norm{z}{\D(A)}^\xi,
 \end{align*}
and from $\norm{\Id}{\LL(V,H)}=\alpha_1^{-\frac12}$, $\xi\ge1$, and~$\|\zeta_2+\delta_2\|<1$.
 
To show the boundedness of the spent ``energy''~$\norm{\mathbf u}{L^{2\mathfrak r}(\R_{0},H)}$, we start by observing that
\begin{align*}
\norm{\mathbf u}{L^{2\mathfrak r}(\R_{0},H)}
&\le \dnorm{P}{\LL}\norm{P_{E_{\M_\sigma}}\left(Ay+A_{\rm rc}(\Bigcdot)y
+\NN(\Bigcdot,y)-\FF_{\M_\sigma}(P_{E_{\M_\sigma}}y)\right)}{L^{2\mathfrak r}(\R_{0},H)}
\\
&\hspace*{-0em}\le \ovlineC{\dnorm{P}{\LL},C_{q3},C_{\rm rc},C_{\NN}}
\left(\norm{q_0}{V}^{\beta_1}\norm{q_0}{\D(A)}^{\beta_2}
+\norm{y}{L^{2\mathfrak r}(\R_{0},V)}
+\textstyle\sum\limits_{j=1}^n\norm{\norm{y}{V}^{\zeta_{1j}
+\delta_{1j}}\norm{y}{\D(A)}^{\zeta_{2j}+\delta_{2j}}}{L^{2\mathfrak r}(\R_{0},\R)}\right)\\
&\hspace*{-0em}\le \ovlineC{n,\dnorm{P}{\LL},C_{\rm rc},C_{\NN},\frac1{\widetilde\e},C_{q1},C_{q3},\norm{q_0}{V},\alpha_{\M_\sigma}}
\left(\norm{q_0}{V}^{\beta_1+\beta_2}
+\norm{y_0}{V}
+\norm{y_0}{V}^{\|\zeta_{1}+\delta_{1}+\zeta_{2}+\delta_{2}\|}\right)
\end{align*}
with~$q_0\coloneqq P_{E_{\M_\sigma}}y_0$, and where we have used~\eqref{stable-s_0}.
Observe that we have that~$\norm{y_0}{V}^{\zeta_{1j}+\delta_{1j}+\zeta_{2j}+\delta_{2j}}\le\norm{y_0}{V}
+\norm{y_0}{V}^{\|\zeta_{1}+\delta_{1}+\zeta_{2}+\delta_{2}\|}$,
because~$\zeta_{1j}+\delta_{1j}+\zeta_{2j}+\delta_{2j}\ge\delta_{1j}+\delta_{2j}\ge1$. Recall also that~$\beta_1+\beta_2\ge1$.
\qed\end{proof}

\subsection{Remark on the transient bound}
We have seen that, see~\eqref{depM_Vnorm} and~\eqref{stable-s_0}, in Theorem~\ref{T:N.goal} we may
take
\begin{align*}
 M&=\ovlineC{n,C_P,C_{\rm rc},C_\NN,C_{q1},C_{q2},C_{q3},\frac1{\lambda},\norm{y(0)}{V}},\\
 \mu_2&=\min\{\tfrac{\widetilde\e}{2},\lambda\},\quad\mbox{and}\quad C_5=\ovlineC{n,C_P,C_{\rm rc},C_\NN,C_{q0},C_{q1},C_{q3},
 \tfrac1{\widetilde\e},\norm{y(0)}{V},\alpha_{\M_\sigma}}.
 \end{align*}
Observe that 
by taking a larger~$M$ we still have a stable closed-loop system, but since the transient
bound~$C_5$ depends on~$\alpha_{\M_\sigma}$, the transient
 time~$t_{\rm tr}=\frac{\log C_5}{\mu_2}$ may also depend on~$\alpha_{\M_\sigma}$. Note also that, from~\eqref{eps.tildeeps}, ~$\mu_2$ will
depend on~$\alpha_{\M_\sigma}$ if $\norm{\mathfrak h}{L^r(\R_0,H)}$ does. We see that $C_5$ gives us an upper bound
for the norm of the closed-loop solution, $\max\{\norm{y(t)}{V}\mid t\ge0\}\le C_5\norm{y(0)}{V}$, and for time~$t\ge t_{\rm tr}$ 
we necessarily have that~$\norm{y(t)}{V}\le\norm{y(0)}{V}$. Therefore, it could be
interesting to understand whether we can make $C_5$ and $t_{\rm tr}$ as small as possible. Though we
do not study this possibility in here, we would like to say that a positive answer does not follow from above,  
due to the dependence on~$\alpha_{\M_\sigma}$. Finding a positive answer to this question
will likely require the derivation of new appropriate estimates.

\section{Parabolic equations with polynomial nonlinearities}\label{S:examples}

We consider parabolic equations, evolving in a bounded domain~$\Omega\subset\R^{\mathbf d}$, $\mathbf d\in\{1,2,3\}$ under
homogeneous Dirichlet or Neumann boundary conditions and with a general polynomial nonlinearity. We assume that~$\Omega$ is regular enough so that
$V\subset H^1(\Omega)$, $\D(A)\subset H^2(\Omega)$, with equivalent norms, say $C_1\norm{y}{V}\le \norm{y}{H^1(\Omega)}\le C_2\norm{y}{V}$
and $C_3\norm{y}{\D(A)}\le \norm{y}{H^2(\Omega)}\le C_4\norm{y}{\D(A)}$ for suitable positive constants~$C_1, C_2,C_3, C_4$.

We check Assumptions~\ref{A:A0sp}--\ref{A:NN} and Assumptions~\ref{A:FFM}--\ref{A:suffalpha}
for the system
\[
 \fractx{\p}{\p t}y+(-\nu\Delta +\Id)y +(a(t,x)-1)y+b(t,x)\cdot\nabla y +\NN(t,x,y)=0,\qquad \GG y\rest{\p\Omega}=0,\qquad y(0)=y_0,
\]
with either Dirichlet, $\GG=\Id$, or Neumann, $\GG=\frac{\p}{\p\nnn}$, homogeneous boundary conditions.
Where~$\nnn$ stands for the unit
outward normal vector to the boundary~${\p\Omega}$ of~$\Omega$.

\begin{remark}
 Below we use the Agmon embedding~$\D(A)\subseteq H^2(\Omega)\xhookrightarrow{}L^\infty(\Omega)$ which holds true
 for~$\mathbf d\in\{1,2,3\}$. For the case $\mathbf d\ge4$, which we do not consider here in this section, we would
 need a different argument because
 we do not necessarily have~$H^2(\Omega)\xhookrightarrow{}L^\infty(\Omega)$, see~\cite[Lem.~13.2]{Agmon65}.
\end{remark}

\subsection{The linear operators}
We check Assumptions~\ref{A:A0sp}--\ref{A:A1}.
We take~$A=-\nu\Delta +\Id\colon\D(A)\to H$, with $H=L^2(\Omega)$
and~$\D(A)=\{u\in H^2(\Omega)\mid \GG u\rest{\p\Omega}=0\}$. Under Dirichlet boundary conditions we have~$V=H^1_0(\Omega)$, 
and under Neumann boundary conditions~$V=H^1(\Omega)$. It is straightforward to see
that Assumptions~\ref{A:A0sp}--\ref{A:A0cdc} are satisfied. Assumption~\ref{A:A1} will be
satisfied if~$a\in L^\infty(\R_0\times\Omega,\R)$
and~$b\in L^\infty(\R_0\times\Omega,\R^3)$.

\subsection{Polynomial reactions and convections in case~$\Omega\subset\R^{3}$}
In case~$\mathbf d=3$ we show now that Assumption~\ref{A:NN} is satisfied for nonlinearities in the form
\begin{subequations}\label{ex.NNr}
 \begin{align}
&\NN(t,y)=\NN(t,x,y)=\textstyle\sum\limits_{j=1}^{n}\left(\hat a_j(t,x)\norm{y}{\R}^{r_j-1}y
+\left(\hat b_j(t,x)\cdot\nabla y\right)\norm{y}{\R}^{s_j-1}y\right),\quad y=y(t,x),\\
&\mbox{with}\quad r_j\in(1,5),\quad s_j\in[1,2),\quad\hat a_j\in L^\infty(\R_0\times\Omega,\R),
\quad \hat b_j\in L^\infty(\R_0\times\Omega,\R^3),\quad (t,x)\in\R_0\times\Omega.
\end{align}
\end{subequations}

\subsubsection{The reaction components}
We start by considereing the terms~$\hat a_j(t,x)\norm{y}{\R}^{r_j-1}y$.
Let us fix~$(t,j,r_j)$. Observe that~$\phi_j(x,s)\coloneqq\hat a_j(t,x)\norm{s}{\R}^{r_j-1}s$, $(t,s)\in\R^2$, is differentiable with respect to~$s$,
because~$\phi_j(x,\Bigcdot)\in C(\R,\R)$ and
$\frac{\p}{\p s}\phi_j(x,\Bigcdot)\in C(\R,\R)$, with $\frac{\p}{\p s}\phi_j(x,\tau)=r_j\hat a_j(t,x)\norm{\tau}{\R}^{r_j-1}$.
We also have the growth bounds
\[
 \phi_j(x,s)\le\norm{\hat a_j(t,\Bigcdot)}{L^\infty(\Omega)}\norm{s}{\R}^{r_j}\quad\mbox{and}\quad
 \textstyle\frac{\p}{\p s}\phi_j(x,\tau)\le r_j\norm{\hat a_j(t,\Bigcdot)}{L^\infty(\Omega)}\norm{\tau}{\R}^{r_j-1}.
\]
Thus, the Nemytskij operator
~$y\mapsto\NN_j(t,y)\coloneqq\hat a_j(t,x)\norm{y}{\R}^{r_j-1}y$ and its Fr\'echet derivative $\ed\NN_j\rest{y}$ satisfy:
\[
 \NN_j(t,\Bigcdot)\in C(L^{2r_j},L^2)\quad\mbox{and}\quad
 \ed\NN_j\rest{y}=r_j\hat a_j(t,x)\norm{y}{\R}^{r_j-1}\in C(L^{2r_j},\LL(L^{2r_j},L^2)).
\]
Indeed, with~$q>1$, we have
\begin{align*}
\norm{r_j\hat a_j\norm{y}{\R}^{r_j-1}h}{L^2}\le r_j\norm{\hat a_j(t,\Bigcdot)}{L^\infty}\norm{\norm{y}{\R}^{r_j-1}h}{L^2}
\le r_j\norm{\hat a_j(t,\Bigcdot)}{L^\infty}\norm{y}{L^{2q(r_j-1)}}^{r_j-1}\norm{h}{L^{2\frac{q}{q-1}}}.
\end{align*}
Setting $q=\frac{r_j}{r_j-1}$ we obtain
\begin{align*}
\norm{r_j\hat a_j\norm{y}{\R}^{r_j-1}h}{L^2}\le
r_j\norm{\hat a_j(t,\Bigcdot)}{L^\infty}\norm{y}{L^{2r_j}}^{r_j-1}\norm{h}{L^{2r_j}}.
\end{align*}

By the Mean Value Theorem
(see, e.g., \cite[Thm.~1.8]{AmbrosettiProdi93}) we can conclude that
\begin{align}
 \norm{\NN_j(t,y_1)-\NN_j(t,y_2)}{L^2}
 &\le
 r_j\norm{\hat a_j(t,\Bigcdot)}{L^\infty}\left(\norm{y_1}{L^{2r_j}}^{r_j-1}+\norm{y_2}{L^{2r_j}}^{r_j-1}\right)
 \norm{y_1-y_2}{L^{2r_j}}.\label{DifPhi1}
\end{align}
Shortening the notation as~$\DD_{\NN_j}\coloneqq\NN_j(t,y_1)-\NN_j(t,y_2)$ and~$\DD_y\coloneqq y_1-y_2$,
 we obtain
\begin{align*}
 \norm{\DD_{\NN_j}}{L^2}
 &\le C_4 r_j\norm{\hat a_j(t,\Bigcdot)}{L^\infty}\left(\norm{y_1}{V}^{r_j-1}+\norm{y_2}{V}^{r_j-1}\right)\norm{\DD_y}{V},
 &&\mbox{if}\quad r_j\le3.\\
\norm{\DD_{\NN_j}}{L^2}
  &\le C_5 r_j\norm{\hat a_j(t,\Bigcdot)}{L^\infty}
 \left(\textstyle\sum\limits_{k=1}^2\norm{y_k}{L^{\infty}}^{\frac{(2r_j-6)(r_j-1)}{2r_j}}
 \norm{y_k}{L^{6}}^{\frac{3(r_j-1)}{r_j}}
 \right)
 \norm{\DD_y}{L^{\infty}}^{\frac{2r_j-6}{2r_j}}\norm{\DD_y}{L^{6}}^{\frac{3}{r_j}}\\
 &\le C_6 r_j\norm{\hat a_j(t,\Bigcdot)}{L^\infty}
 \left(\textstyle\sum\limits_{k=1}^2\norm{y_k}{\D(A)}^{\frac{(2r_j-6)(r_j-1)}{4r_j}}
 \norm{y_k}{V}^{\frac{(2r_j+6)(r_j-1)}{4r_j}}
 \right)
 \norm{\DD_y}{\D(A)}^{\frac{2r_j-6}{4r_j}}\norm{\DD_y}{V}^{\frac{2r_j+6}{4r_j}},
 &&\mbox{if}\quad r_j>3.
\end{align*}
Where we have used the Sobolev
embedding inequality~$\norm{z}{L^6}\le C\norm{z}{V}$, see~\cite[Thm.~4.57]{DemengelDem12}, and the Agmon
inequality~$\norm{z}{L^\infty}\le C\norm{z}{V}^{\frac{1}{2}}\norm{z}{\D(A)}^{\frac{1}{2}}$,
see~\cite[Lem.~13.2]{Agmon65},\cite[Sect.~1.4]{Temam97}.

Therefore, we can see
that~$\NN_j$ satisfies the inequality in Assumption~\ref{A:NN} when~$1<r_j<5$, with
\begin{align*}
&\zeta_{1j}=r_j-1, 
\quad&&\zeta_{2j}=0,
\quad &&\delta_{1j}=1,\quad&&\delta_{2j}=0,&&\quad\mbox{if}\quad r_j\in(1,3].\\
&\zeta_{1j}=\fractx{(2r_j+6)(r_j-1)}{4r_j},
\quad&&\zeta_{2j}=\fractx{(2r_j-6)(r_j-1)}{4r_j},
\quad &&\delta_{1j}=\fractx{2r_j+6}{4r_j},\quad&&\delta_{2j}=\fractx{2r_j-6}{4r_j},&&\quad\mbox{if}\quad r_j\in(3,5).
\end{align*}
In either case
$\zeta_{2j}+\delta_{2j}<1$ and~$\delta_{1j}+\delta_{2j}=1$. For~$r_j\in(3,5)$ we
have~$\zeta_{2j}+\delta_{2j}=\fractx{2r_j-6}{4}=\fractx{r_j-3}{2}<1$.

\subsubsection{The convection components} 
We consider now the terms~$\left(\hat b_j(t,x)\cdot\nabla y\right)\norm{y}{\R}^{s_j-1}y$.
Observe that
\begin{align}
 &\quad\norm{\left(\hat b_j(t,\Bigcdot)\cdot\nabla y_1\right)\norm{y_1}{\R}^{s_j-1}y_1
 -\left(\hat b_j(t,\Bigcdot)\cdot\nabla y_2\right)\norm{y_2}{\R}^{s_j-1}y_2}{L^2}\notag\\
 &\le\norm{\left(\hat b_j(t,\Bigcdot)\cdot\nabla y_1\right)\left(\norm{y_1}{\R}^{s_j-1}y_1-\norm{y_2}{\R}^{s_j-1}y_2\right)}{L^2}
 +\norm{\left(\hat b_j(t,\Bigcdot)\cdot\nabla (y_1-y_2)\right)\norm{y_2}{\R}^{s_j-1}y_2}{L^2}\label{jConv0}
  \end{align}
  and
 \begin{align*} 
  &\norm{\left(\hat b_j(t,\Bigcdot)\cdot\nabla y_1\right)\left(\norm{y_1}{\R}^{s_j-1}y_1-\norm{y_2}{\R}^{s_j-1}y_2\right)}{L^2}\\
  &\hspace*{8em}\le Cs_j\norm{\hat b_j(t,\Bigcdot)}{L^\infty}\norm{\nabla y_1}{L^\frac{6}{4-s_j}}
  \norm{\norm{y_1}{\R}^{s_j-1}+\norm{y_2}{\R}^{s_j-1}}{L^\frac{6}{s_j-1}}\norm{y_1-y_2}{L^\infty},\\
  &\norm{\left(\hat b_j(t,\Bigcdot)\cdot\nabla (y_1-y_2)\right)\norm{y_2}{\R}^{s_j-1}y_2}{L^2}
  \le C\norm{\hat b_j(t,\Bigcdot)}{L^\infty}\norm{\nabla (y_1-y_2)}{L^\frac{6}{4-s_j}}
  \norm{\norm{y_2}{\R}^{s_j-1}}{L^\frac{6}{s_j-1}}\norm{y_2}{L^\infty}.
 \end{align*}
Recall that, since~$\mathbf d=3$, we have the  Sobolev embedding~$H^{\frac{s_j-1}{2}}\xhookrightarrow{}L^{\frac{6}{4-s_j}}$,
because~$2\frac{s_j-1}{2}<\mathbf d$, and~$\fractx{6}{4-s_j}=\fractx{2\mathbf d}{\mathbf d-2\frac{s_j-1}{2}}$,
see~\cite[Thm.~4.57]{DemengelDem12}.
Thus, ~$\norm{\nabla z}{H^{\frac{s_j-1}{2}}}\le C_1\norm{z}{H^{1+\frac{s_j-1}{2}}}
\le C_2\norm{z}{V}^{1-\frac{s_j-1}{2}}\norm{z}{\D(A)}^{\frac{s_j-1}{2}}$, where the latter inequality follows by an
interpolation inequality, see~\cite[Ch.~1, Sect.~2.5, Prop.~2.3]{LioMag72-I}. Therefore, we find
\begin{subequations}\label{jConv}
\begin{align}
 &\norm{\left(\hat b_j(t,\Bigcdot)\cdot\nabla y_1\right)\left(\norm{y_1}{\R}^{s_j-1}y_1-\norm{y_2}{\R}^{s_j-1}y_2\right)}{L^2}
 \notag\\
  &\hspace*{3em}\le s_jC_3\norm{\hat b_j(t,\Bigcdot)}{L^\infty}\norm{y_1}{V}^{\frac{3-s_j}{2}}\norm{y_1}{\D(A)}^{\frac{s_j-1}{2}}
  \left(\norm{y_1}{V}^{s_j-1}+\norm{y_2}{V}^{s_j-1}\right)\norm{y_1-y_2}{V}^\frac{1}{2}\norm{y_1-y_2}{\D(A)}^\frac{1}{2},
  \label{jConv1}\\
&\norm{\left(\hat b_j(t,\Bigcdot)\cdot\nabla (y_1-y_2)\right)\norm{y_2}{\R}^{s_j-1}y_2}{L^2} \notag\\
  &\hspace*{3em}\le C_3\norm{\hat b_j(t,\Bigcdot)}{L^\infty}
  \norm{y_2}{V}^{s_j-\frac12}\norm{y_2}{\D(A)}^\frac{1}{2}
  \norm{y_1-y_2}{V}^{\frac{3-s_j}{2}}\norm{y_1-y_2}{\D(A)}^{\frac{s_j-1}{2}}.
  \label{jConv2}
     \end{align}
  \end{subequations}
 Next, observe that in case~$s_j>1$, and setting~$\kappa\in(1,\frac{1}{s_j-1})$,
\begin{align*}
 &\norm{y_1}{V}^{\frac{3-s_j}{2}}\norm{y_1}{\D(A)}^{\frac{s_j-1}{2}}
  \left(\norm{y_1}{V}^{s_j-1}+\norm{y_2}{V}^{s_j-1}\right)
  = \norm{y_1}{V}^{\frac{s_j+1}{2}}\norm{y_1}{\D(A)}^{\frac{s_j-1}{2}}
  +\norm{y_1}{V}^{\frac{3-s_j}{2}}\norm{y_1}{\D(A)}^{\frac{s_j-1}{2}}\norm{y_2}{V}^{s_j-1}\\
  &\hspace*{1em}\le \norm{y_1}{V}^{\frac{s_j+1}{2}}\norm{y_1}{\D(A)}^{\frac{s_j-1}{2}}
  +\norm{y_1}{V}^{\kappa\frac{3-s_j}{2}}\norm{y_1}{\D(A)}^{\kappa\frac{s_j-1}{2}}+\norm{y_2}{V}^{(s_j-1)\frac{\kappa}{\kappa-1}}
  \le \textstyle\sum\limits_{i=1}^3
  \left(\norm{y_1}{V}^{\zeta_{1i}}\norm{y_1}{\D(A)}^{\zeta_{2i}}+\norm{y_2}{V}^{\zeta_{1i}}\norm{y_2}{\D(A)}^{\zeta_{2i}}
  \right)
\end{align*}
with~$(\zeta_{1i},\zeta_{2i})\in\{(\frac{s_j+1}{2},\frac{s_j-1}{2}),
(\kappa\frac{3-s_j}{2},\kappa\frac{s_j-1}{2}),(\frac{(s_j-1)\kappa}{\kappa-1},0)\}$. Hence, in case~$s_j\in(1,2)$,
the component in~\eqref{jConv1} is bounded by a sum as in Assumption~\ref{A:NN} with~$\delta_{1i}=\delta_{2i}=\frac12$ and
$\zeta_{2i}+\delta_{2i}\le\frac{1+\kappa(s_j-1)}{2}<1$.

We can see that the component in~\eqref{jConv2} can be bounded as in Assumption~\ref{A:NN},
namely with~$(\zeta_{1i},\zeta_{2i},\delta_{1i},\delta_{2i})=(s_j-\frac12,\frac12,\frac{3-s_j}2,\frac{s_j-1}2)$. Note that
$\delta_{1i}+\delta_{2i}=1$ and $\zeta_{2i}+\delta_{2i}=\frac{s_j}2<1$.

Finally, in case $s_j=1$, ~\eqref{jConv0} implies
\begin{align*}
 &\quad\norm{\left(\hat b_j(t,\Bigcdot)\cdot\nabla y_1\right)\norm{y_1}{\R}^{s_j-1}y_1
 -\left(\hat b_j(t,\Bigcdot)\cdot\nabla y_2\right)\norm{y_2}{\R}^{s_j-1}y_2}{L^2}\notag\\
 &\le\norm{\hat b_j(t,\Bigcdot)}{L^\infty}\Bigl(\norm{\nabla y_1}{L^2}\norm{y_1-y_2}{L^\infty}
 +\norm{\nabla (y_1-y_2)}{L^2}\norm{y_2}{L^\infty}\Bigr)\\
 &\le \norm{\hat b_j(t,\Bigcdot)}{L^\infty}\left(\norm{y_1}{V}\norm{y_1-y_2}{V}^\frac{1}{2}\norm{y_1-y_2}{\D(A)}^\frac{1}{2}
 +\norm{y_2}{V}^\frac{1}{2}\norm{y_2}{\D(A)}^\frac{1}{2}\norm{y_1-y_2}{V}\right).
  \end{align*}
 Again, we have a sum as in Assumption~\ref{A:NN}
 with~$(\zeta_{1i},\zeta_{2i},\delta_{1i},\delta_{2i})\in\{(1,0,\frac{1}{2},\frac{1}{2}),(\frac{1}{2},\frac{1}{2},1,0)\}$.
 Note that in either case~$\delta_{1i}+\delta_{2i}=1$ and $\zeta_{2i}+\delta_{2i}=\frac12$.

\begin{remark} Above in~\eqref{ex.NNr}, we may replace~$\norm{y}{\R}^{r_j-1}$ by~$y^{r_j-1}$ in case~$r_j\in\{2,3,4\}$ is an integer.
Analogously we may replace~$\norm{y}{\R}^{s_j-1}$ by~$y^{s_j-1}$ in case~$s_j=1$. The reason the absolute value
is taken in~\eqref{ex.NNr} is because we want $\NN(t,x,y(t,x))\in\R$, in order to have real valued
solutions~$y(t,x)\in\R$. 
\end{remark}

\subsection{The cases~$\Omega\subset\R^{1}$ and~$\Omega\subset\R^{2}$}\label{sS:d12}
In the cases~$\Omega\subset\R^{\mathbf d}$, $\mathbf d\in\{1,2\}$ we have the 
Sobolev embedding inequality~$\norm{z}{L^{q}}\le C\norm{z}{V}$ for all~$q<\infty$, see~\cite[Thm.~4.57]{DemengelDem12}.
In particular from~\eqref{DifPhi1} we can obtain
that the reaction terms satisfy an inequality in Assumption~\ref{A:NN} for all~$r_j>1$,
with~$(\zeta_{1i},\zeta_{2i},\delta_{1i},\delta_{2i})=(r_j-1,0,1,0)$, and from~\eqref{jConv0} the convection terms satisfy
\begin{align*}
 &\quad\norm{\left(\hat b_j(t,\Bigcdot)\cdot\nabla y_1\right)\norm{y_1}{\R}^{s_j-1}y_1
 -\left(\hat b_j(t,\Bigcdot)\cdot\nabla y_2\right)\norm{y_2}{\R}^{s_j-1}y_2}{L^2}\notag\\
 &\le\norm{\hat b_j(t,\Bigcdot)}{L^\infty}\!\!
 \Bigl(s_j\norm{\nabla y_1}{L^2}\left(\norm{\norm{y_1}{\R}^{s_j-1}+\norm{y_2}{\R}^{s_j-1}}{L^3}\right)\norm{y_1-y_2}{L^6}\\
 &\hspace*{1em}+\norm{\nabla (y_1-y_2)}{L^2}\norm{\norm{y_2}{\R}^{s_j-1}}{L^3}\norm{y_1-y_2}{L^6}\Bigr)\\
 &\le C\norm{\hat b_j(t,\Bigcdot)}{L^\infty}\!\!
 \Bigl(s_j\left(\norm{y_1}{V}^{s_j}+\norm{y_2}{V}^{s_j}\right)\norm{y_1-y_2}{V}
 +\norm{y_2}{V}^{s_j-1}\norm{y_1-y_2}{V}^2\Bigr).
  \end{align*}
Hence, they satisfy an
inequality as in Assumption~\ref{A:NN}, with~$(\zeta_{1i},\zeta_{2i},\delta_{1i},\delta_{2i})\in\{(s_j,0,1,0),(s_j-1,0,2,0)\}$.
Therefore, in the cases~$\Omega\subset\R^{\mathbf d}$, and~$\mathbf d\in\{1,2\}$, we can take nonlinearities as
 \begin{align*}
&\NN(t,y)=\NN(t,x,y)=\textstyle\sum\limits_{j=1}^{n}\left(\hat a_j(t,x)\norm{y}{\R}^{r_j-1}y
+\left(\hat b_j(t,x)\cdot\nabla y\right)\norm{y}{\R}^{s_j-1}y\right),\quad y=y(t,x),
\\
&\mbox{with}\quad
 r_j>1,\quad s_j\ge1,\quad\hat a_j\in L^\infty(\R_0\times\Omega,\R),\quad
\hat b_j(t)\in L^\infty(\R_0\times\Omega,\R^{\mathbf d}),\quad (t,x)\in\R_0\times\Omega.
\end{align*}

\begin{remark}
 Note that since~$\zeta_{2i}=\delta_{2i}=0$, by taking~$\FF_{\M_\sigma}=\lambda\Id$ we find that Assumption~\ref{A:FFM}
will follow with~$\eta_2=\beta_2=1$ and~$\dnorm{\zeta_1+\delta_1}{}<2$. In particular, Assumption~\ref{A:FFM} is satisfied, in
case~$\mathbf d\in\{1,2\}$, if we have no convection terms and the reaction terms satisfy~$1<r_j<2$. On the other hand
we underline that Assumption~\ref{A:FFM} is part of the sufficient conditions for stability of the closed system,
we do not claim that Assumption~\ref{A:FFM} is necessary. That is, our results do not show neither that
the feedback obtained by taking~$\FF_{\M_\sigma}=\lambda\Id$ is able to stabilize the system nor that it is not.
\end{remark}

\section{Numerical results}\label{S:simul}
We present here numerical results in the one dimensional case, showing the stabilizing
performance of the controller. Our parabolic equation evolving in the unit interval~$(0,L)$ reads
\[
 \fractx{\p}{\p t}y+(-\nu\Delta +\Id)y +(a-1)y+b\cdot\nabla y -c_\NN\norm{y}{\R}^{p-1}y
 =\KKK(y),\qquad y(t,0)=y(t,L)=0,\qquad y(0)=c_{\rm ic}y_0,
\]
where Dirichlet boundary conditions are imposed and where we have taken
\begin{align*}
 &\nu=0.1,\qquad L=1,\qquad y_0=\sin(2\pi x),\qquad p=\tfrac{13}{4},\qquad c_{\rm ic}\in\{2,4\},
 \qquad c_\NN\in\{0,1\},\\
  &a(t,x)\coloneqq -35\nu\pi^2-10\norm{\cos(4t)x\cos(xt)}{\R},\qquad\!
 b(t,x) \coloneqq -4\cos(3\pi t)-5(1-x)^2+2,\\
 &
 \KKK\in\left\{\KKK_{U_{\#\M_\sigma}}^{\FF_{\M_\sigma}},\KKK_{U_{\#\M_\sigma}}^{\FF_{\M_\sigma},\NN}\right\},
 \qquad\FF_{\M_\sigma}\in\left\{\lambda\Id,-\nu\Delta+\lambda\Id\right\},\qquad\lambda=1,\qquad\M_\sigma=\M=\{1,2,\dots,M\}.
 \end{align*}
Above~$(t,x)\in(0,+\infty)\times(0,1)$. Recall that~$\KKK_{U_{\#\M_\sigma}}^{\FF_{\M_\sigma}}$ and~$\KKK_{U_{\#\M_\sigma}}^{\FF_{\M_\sigma},\NN}$ are defined
 in~\eqref{FeedKyx} and~\eqref{FeedKy-Nx}, respectively.

For a given~$M\in\N_0$, the actuators were taken as in~\eqref{Act-mxe},
~$U_{\#\M_\sigma}=U_{M}=\linspan\{1_{\omega_i}\mid i\in\M\}$, with~$r=0.1$ and~$L_1=L=1$, that is,
the actuators cover $10\%$ of the domain~$(0,L)$.

 The simulations have been performed for a spatial finite element approximation of the equation, based on piecewise linear
 elements (hat functions). The interval domain~$[0,1]$
 have been discretized by~$N+1$ equidistant nodes~$[0,\frac{1}{N},\frac{2}{N},\dots,\frac{N-1}{N},1]$,
 with~$N=10000$.
 To solve the associated {\sc ode}s we followed a Crank-Nicolson scheme with the time interval~$(0,+\infty)$ discretized with
 timestep~$k=0.0001$, $[0,k,2k,3k,\dots)$. For further details, see~\cite{RodSturm18}.
In the figures below we are going to plot the behaviour of either~$\norm{y}{H}$ or~$\norm{y}{V}$.
Note that since~$V\xhookrightarrow{}H$, if~$\norm{y}{H}$ goes to~$+\infty$ then also~$\norm{y}{V}$ does.
Analogously, if~$\norm{y}{V}$ goes to~$0$ then also~$\norm{y}{H}$ does.
These norms have been computed/approximated
as~$\norm{y(t_j)}{H}^2=\overline y(t_j)^\top\Ma\overline y(t_j)$
and~$\norm{y(t_j)}{V}^2=\overline y(t_j)^\top(\nu\St+\Ma)\overline y(t_j)$.
Here~$\Ma$ and~$\St$ are, respectively, the Mass and Stiffness matrices, and~$\overline y(t_j)$
is the discrete solution at a given discrete time~$t_j=jk$. The simulations have been run for time~$t\in[0,5]$,
and have been performed in MATLAB.
 
In the figures below~$\FF_{M}=\FF_{\M_\sigma}$,
and~``${\rm Ktype}={\rm Klinz}$'' means that we have taken the linearization
based feedback~$\KKK=\KKK_{U_{\#\M_\sigma}}^{\FF_{\M_\sigma}}$, while~``${\rm Ktype}={\rm Knonl}$''
means that we have taken the nonlinear
feedback~$\KKK=\KKK_{U_{\#\M_\sigma}}^{\FF_{\M_\sigma},\NN}$.
Note that with~$c_\NN=0$ the system is linear, while with~$c_\NN=1$ the system is nonlinear.
Furthermore, $FeedOn$ stands for the time interval on which the control is switched on.
For example in Figure~\ref{Fig:UNC} the control is switched {\em off} on the entire time interval $[0,5)$, while in
Figure~\ref{Fig:CT.SF=LL} it is is switched {\em on} on the entire time interval $[0,5)$.

In Figure~\ref{Fig:UNC}, we observe
that both the linear and the nonlinear 
systems are unstable. The linear system is exponentially unstable and the nonlinear system blows up in finite time.
\begin{figure}[ht]
\centering
\subfigure
{\includegraphics[width=0.312\textwidth]{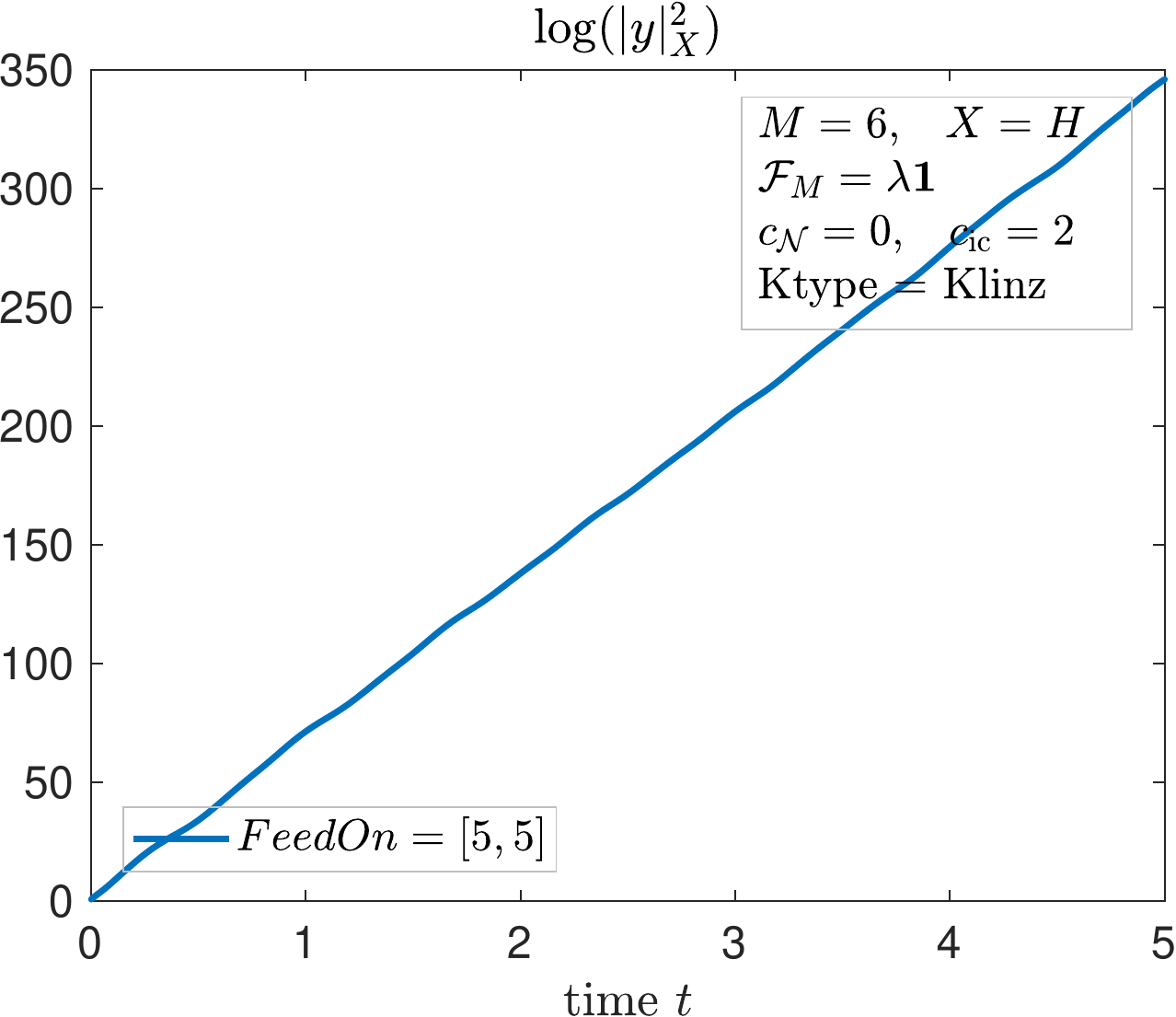}}
\qquad\subfigure
{\includegraphics[width=0.325\textwidth]{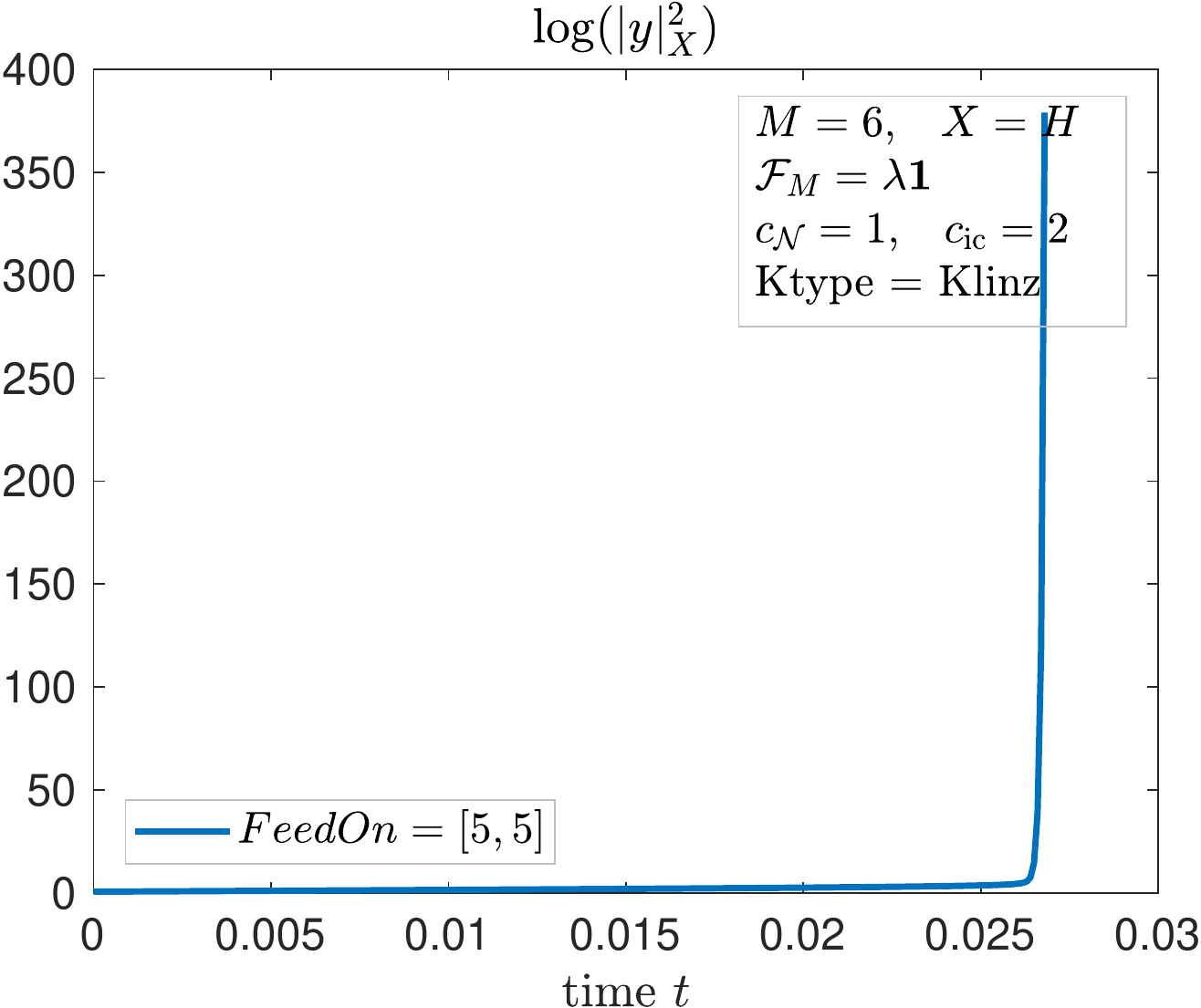}}
\caption{Uncontrolled solutions.  Linear and nonlinear systems.}
\label{Fig:UNC}
\end{figure}

In Figure~\ref{Fig:CT.SF=LL} we see that, with~$6$ actuators, the linear
feedback is able to stabilize the linear system, for both choices
of~$\FF_{M}$. In this example,
the choice of $\FF_{M}=-\nu\Delta+\lambda\Id$ leads to faster exponential decay rate
of the $V$-norm.
\begin{figure}[ht]
\centering
\subfigure
{\includegraphics[width=0.321\textwidth]{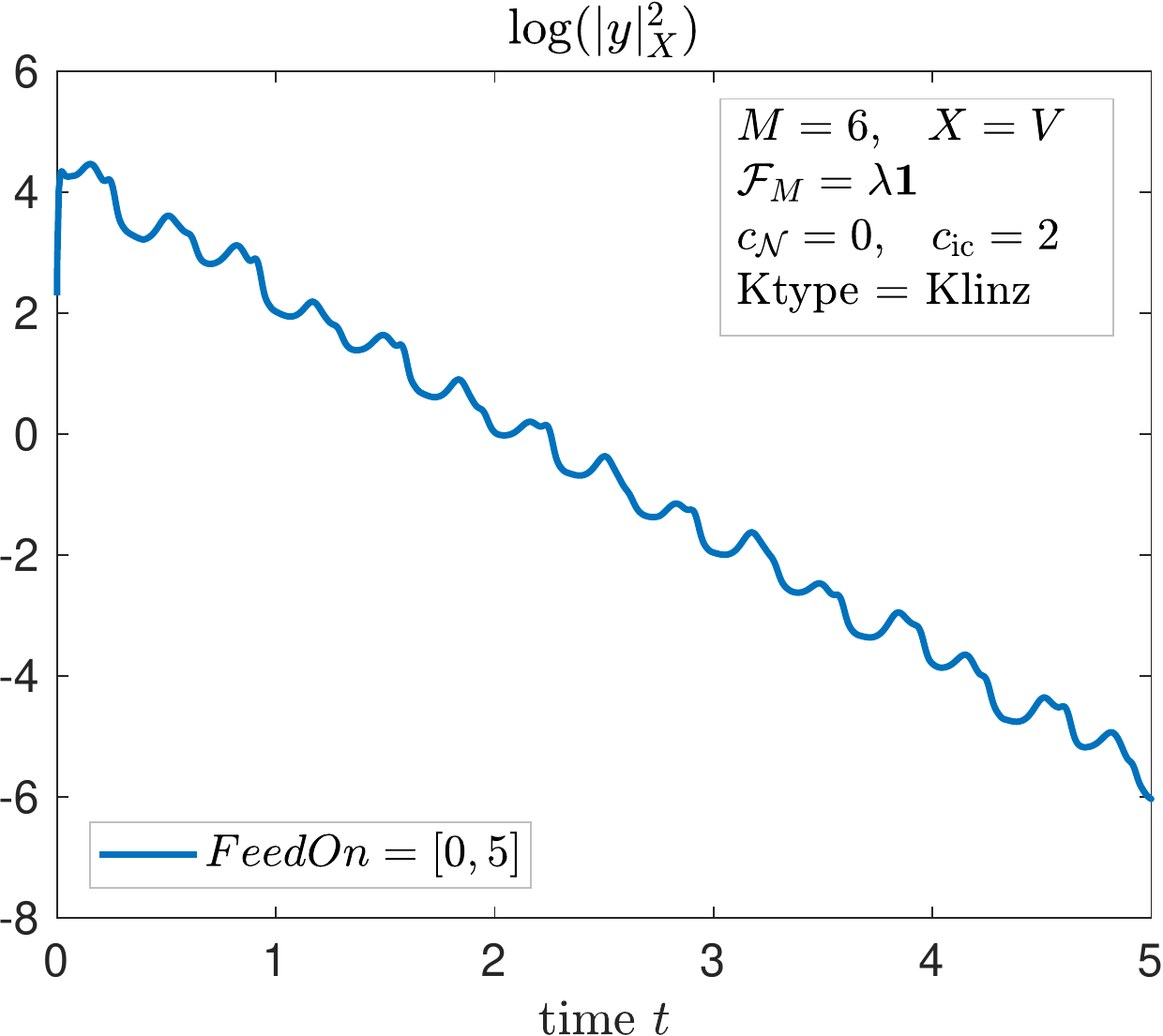}}
\qquad\subfigure
{\includegraphics[width=0.325\textwidth]{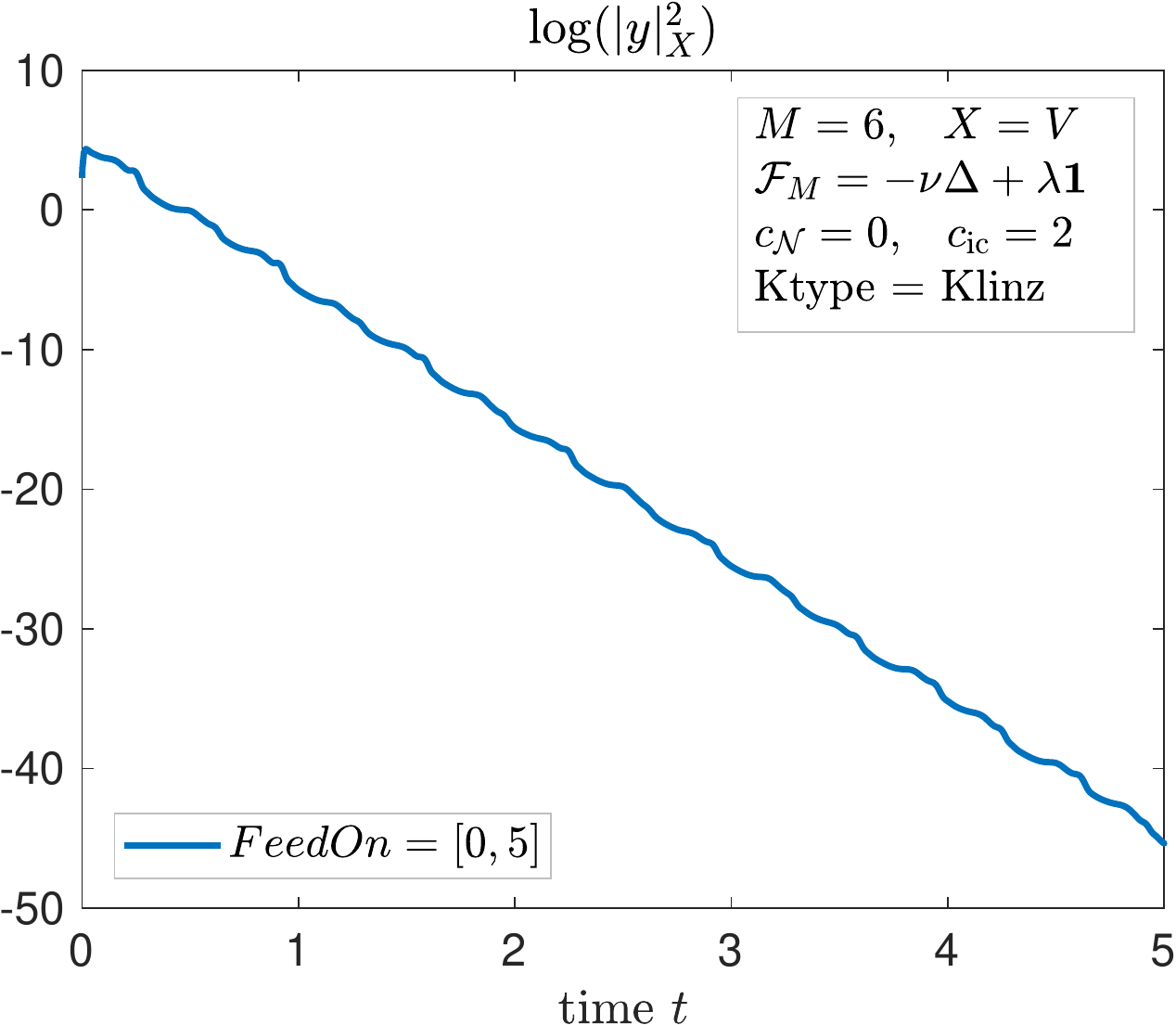}}
\caption{Linear systems and linear feedback.}
\label{Fig:CT.SF=LL}
\end{figure}

In Figure~\ref{Fig:CT.SF=NL} we see that the same linear feedback, is not
able to stabilize the nonlinear system. This is because the initial condition is too big. Recall that it is known that
we can expect such
linearization based feedback to be able to stabilize the nonlinear system only if the norm of the initial condition
is small enough (local stability).
\begin{figure}[ht]
\centering
\subfigure
{\includegraphics[width=0.32\textwidth]{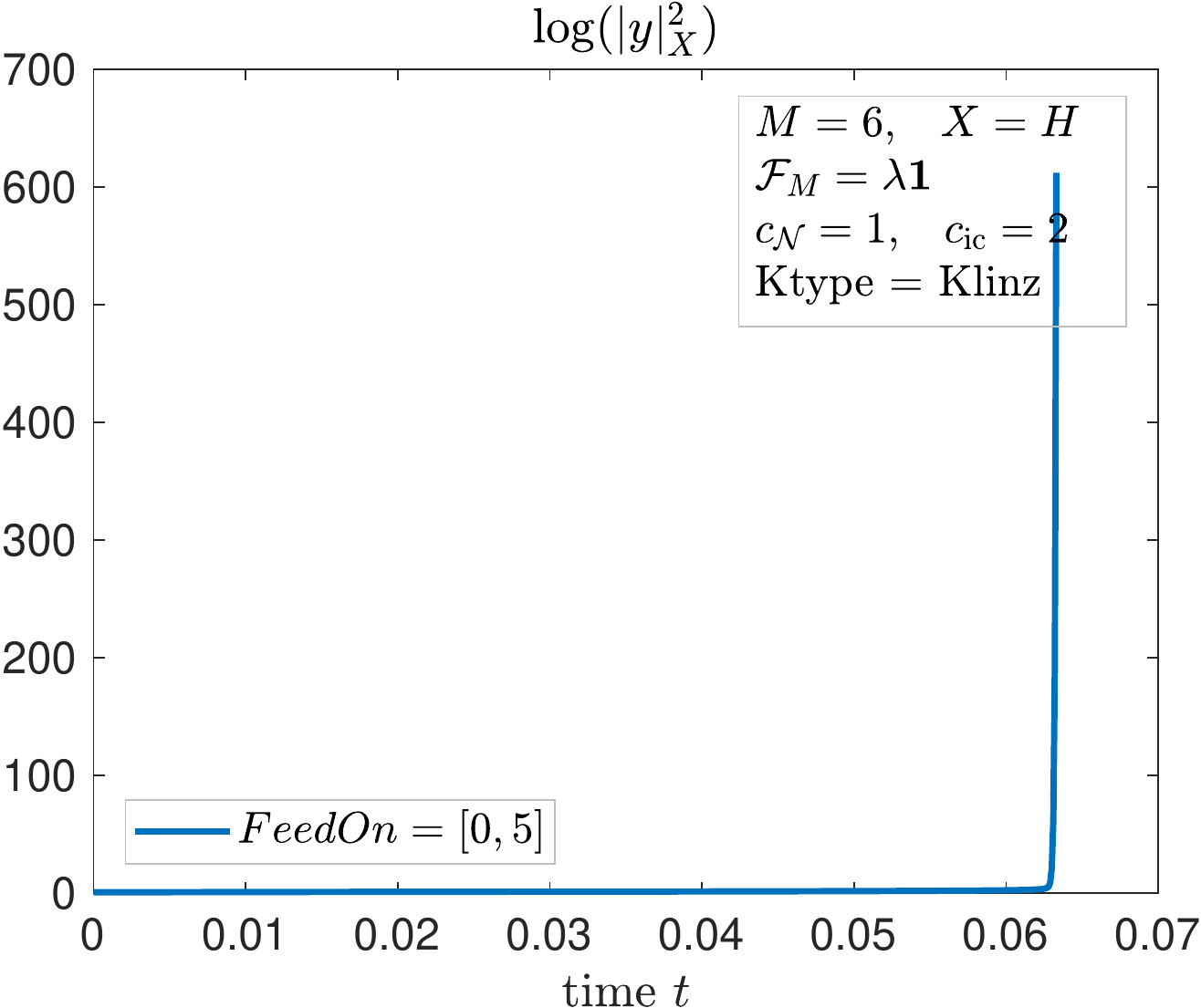}}
\qquad\subfigure
{\includegraphics[width=0.325\textwidth]{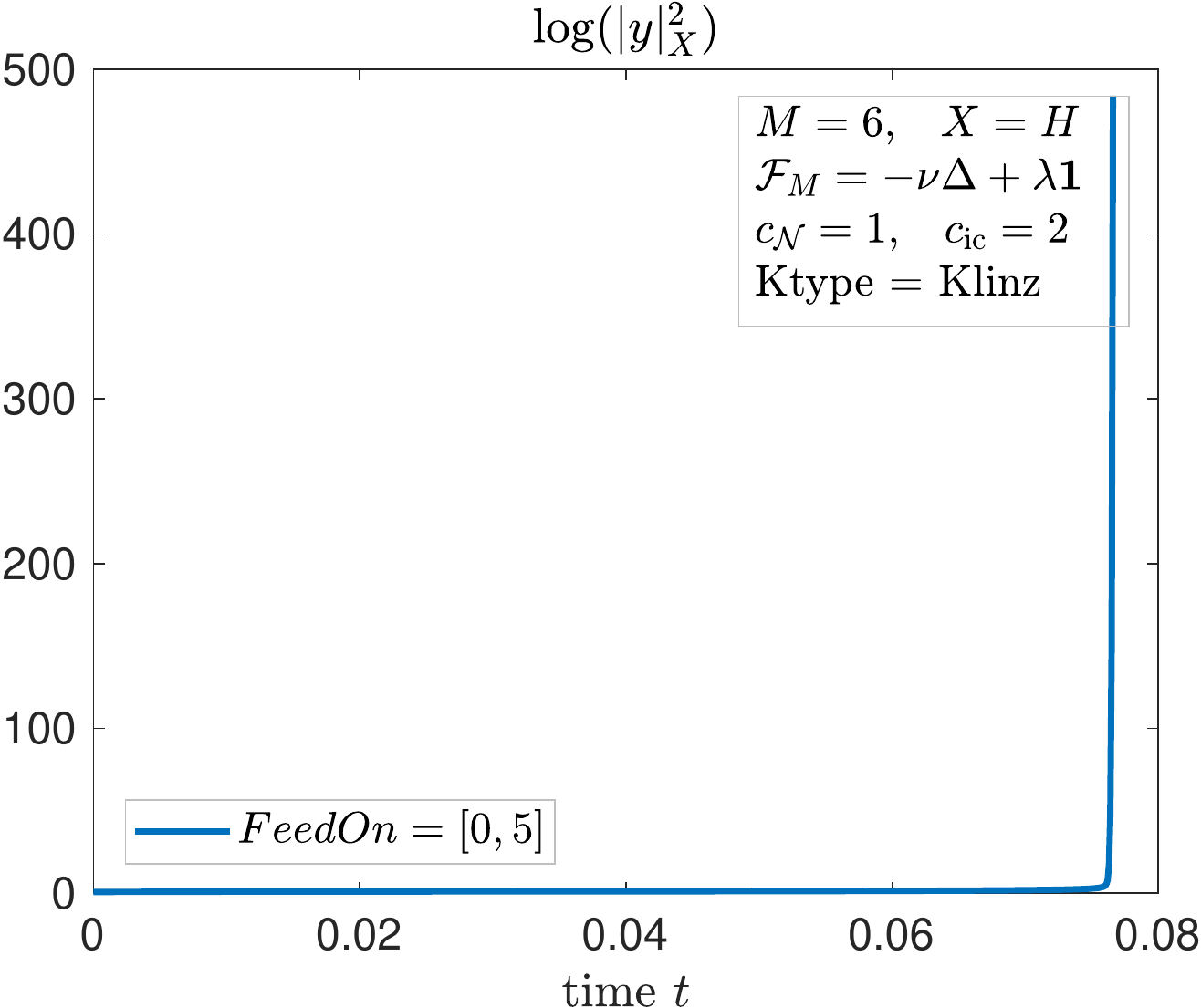}}
\caption{Nonlinear systems and linear feedback.}
\label{Fig:CT.SF=NL}
\end{figure}

In Figure~\ref{Fig:CT.SF=NN} we observe that the full nonlinear feedback with~$6$ actuators and with~$\FF_{M}=-\nu\Delta+\lambda\Id$ succeeds to
stabilize the solution, while with the choice~$\FF_{M}=\lambda\Id$ it fails to. The latter
choice~$\FF_{M}=\lambda\Id$ succeeds by taking~$7$ actuators.
\begin{figure}[ht]
\centering
\subfigure
{\includegraphics[width=0.325\textwidth]{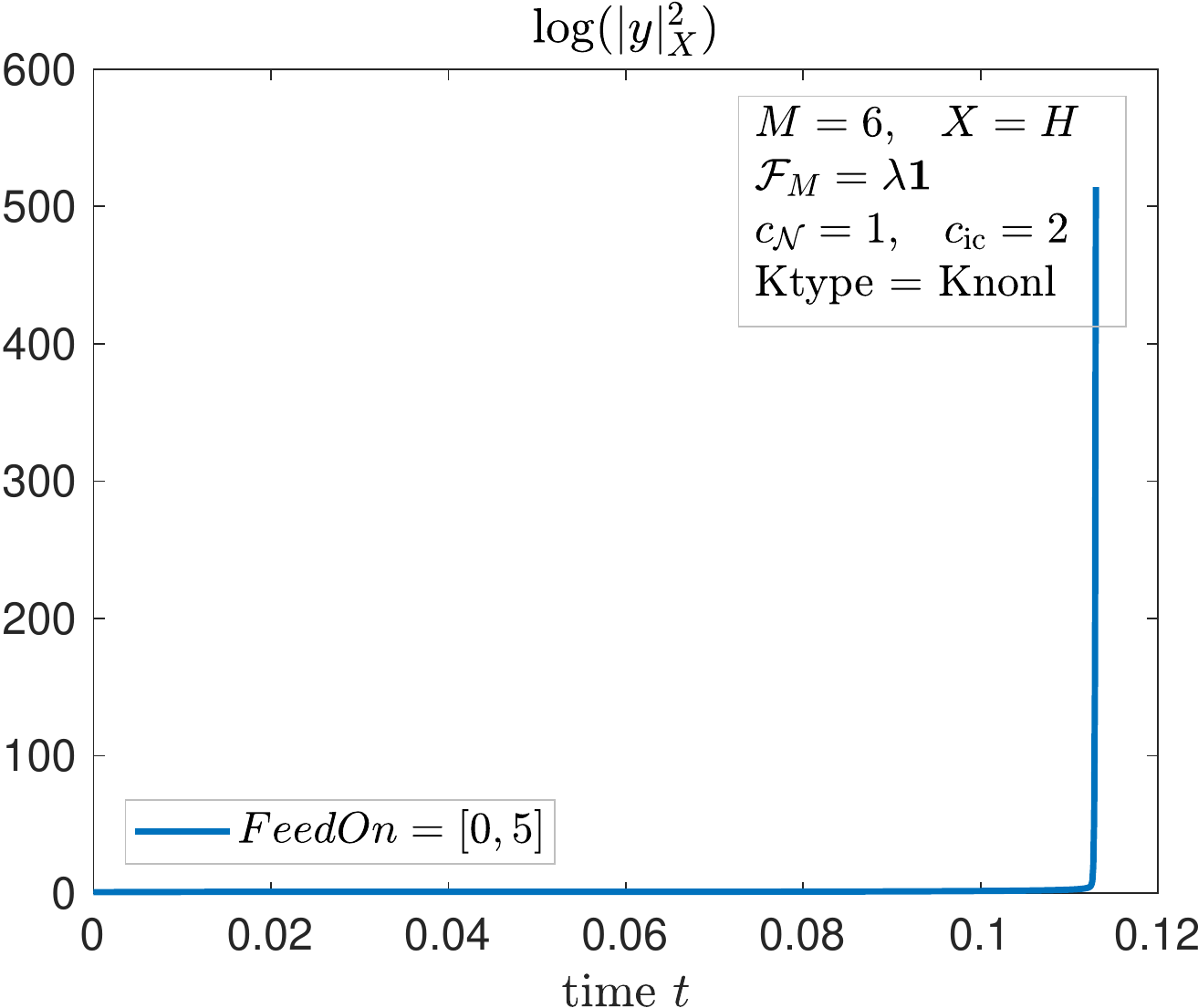}}
\subfigure
{\includegraphics[width=0.31\textwidth]{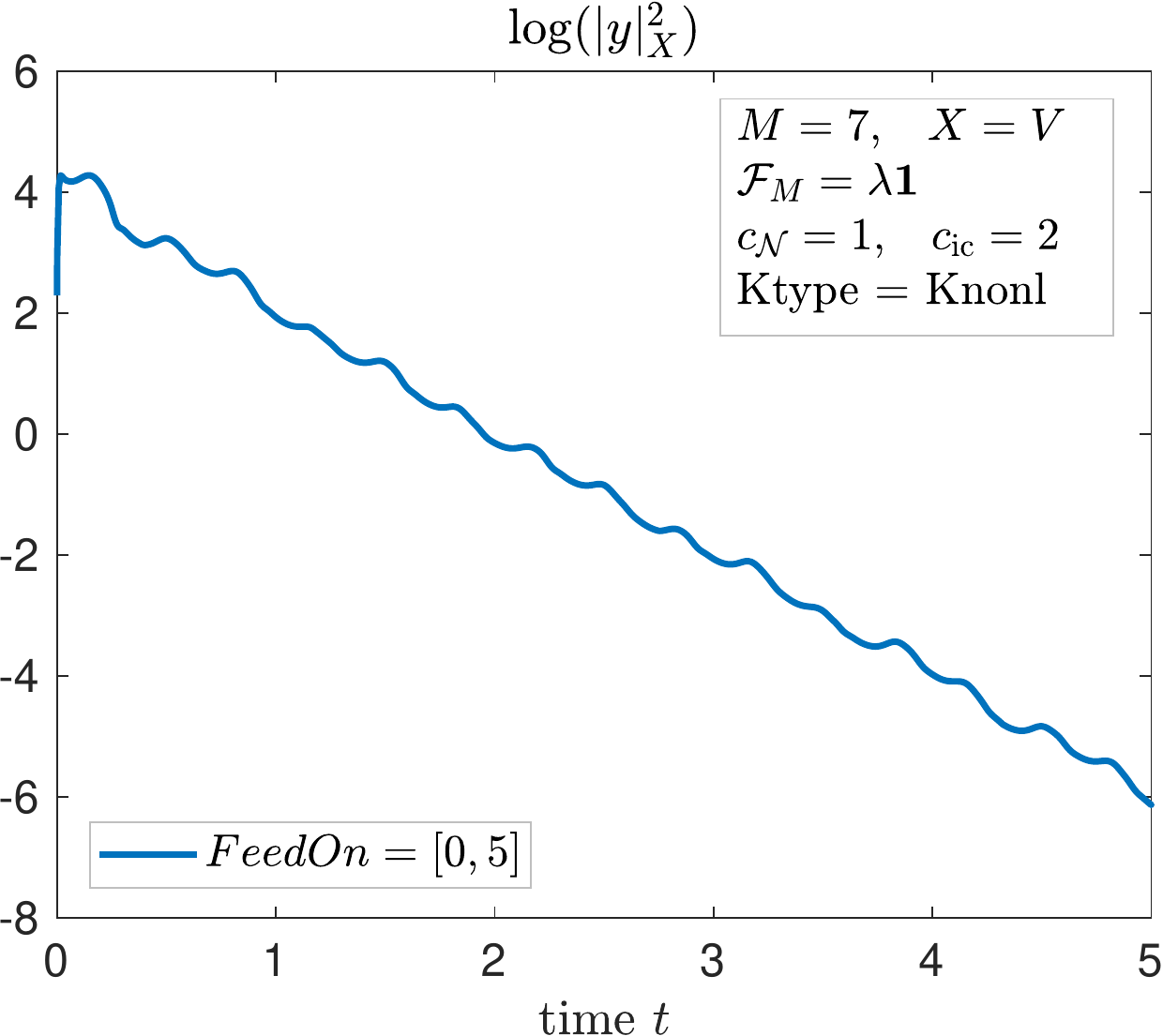}}
\subfigure
{\includegraphics[width=0.32\textwidth]{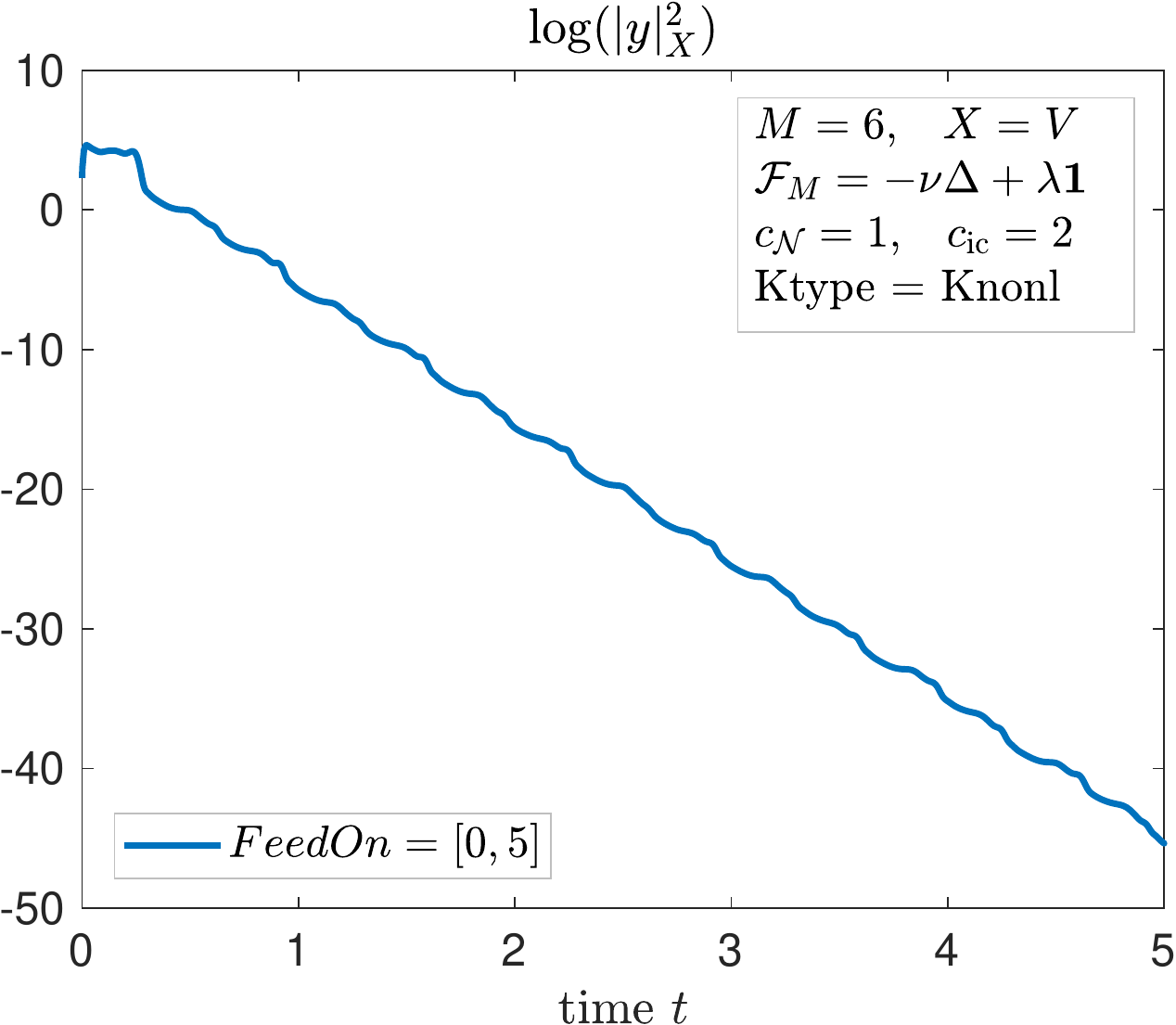}}
\caption{Nonlinear systems and nonlinear feedback.}
\label{Fig:CT.SF=NN}
\end{figure}
Figure~\ref{Fig:CT.SF=NN.fic4} shows that, for a bigger initial condition,
the same nonlinear feedback, with~$7$ actuators
is not anymore able to stabilize the system for both choices~$\FF_{M}=-\nu\Delta+\lambda\Id$ and~$\FF_{M}=\lambda\Id$.
Finally, in Figure~\ref{Fig:CT.SF=NN.fic4.M789}
we observe that by increasing the number~$M$ of actuators the nonlinear feedback is again able to stabilize
the system. This could give raise to the question on whether by incresing~$M$ would also lead to the stability of the linearization based
closed-loop system, Figure~\ref{Fig:CT.SF=NL102040} shows that this is not the case.
\begin{figure}[ht]
\centering
\subfigure
{\includegraphics[width=0.325\textwidth]{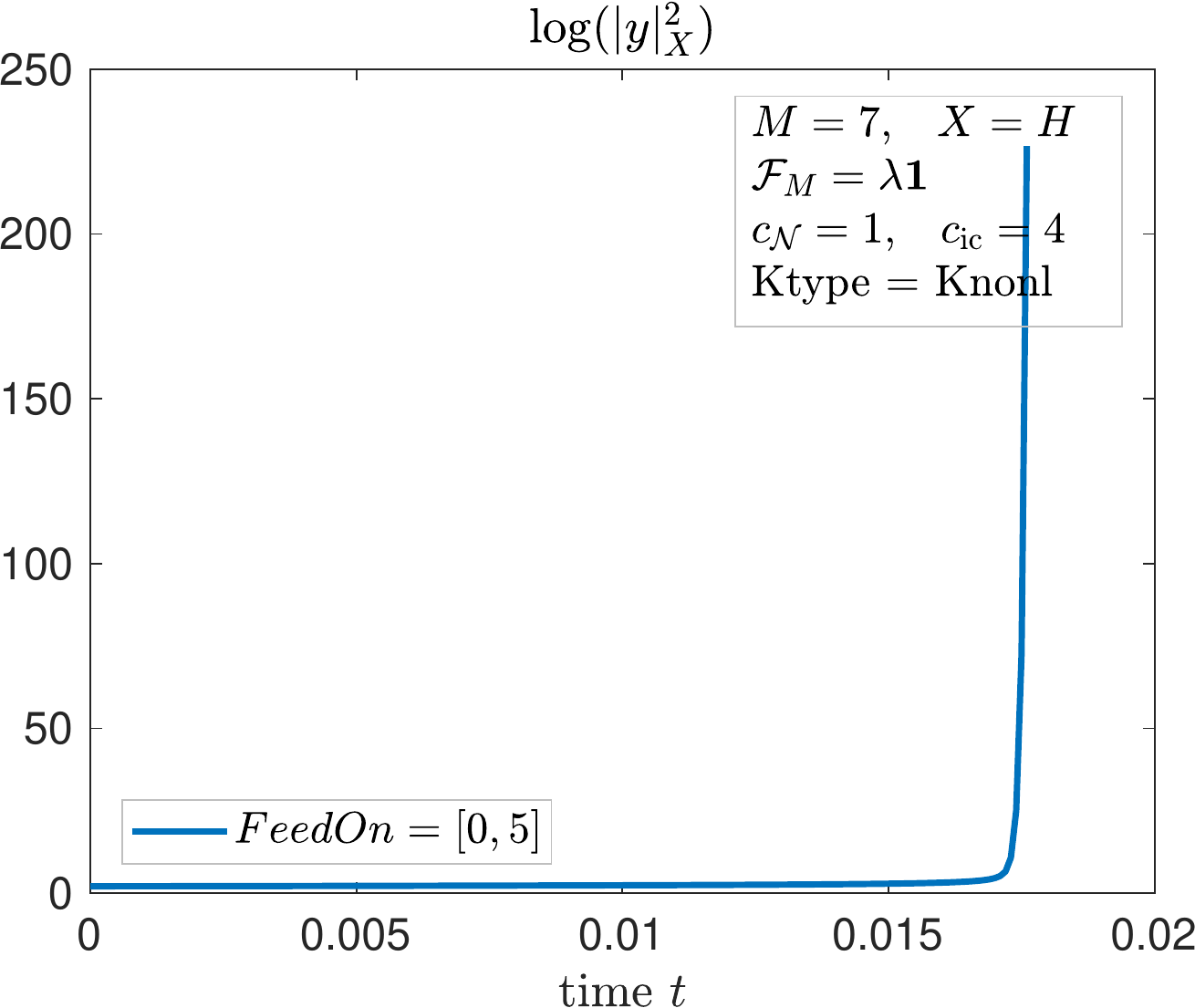}}
\qquad\subfigure
{\includegraphics[width=0.325\textwidth]{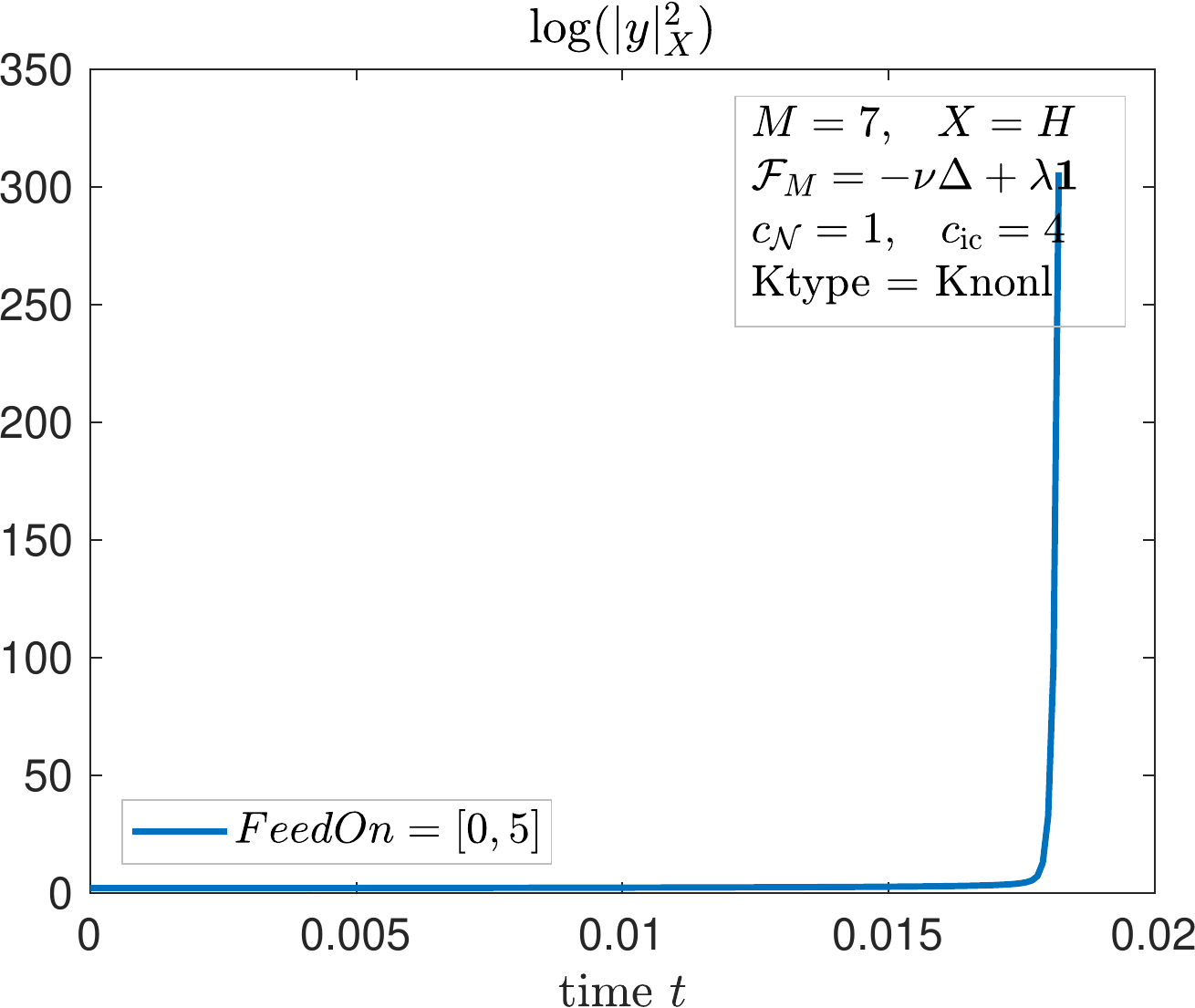}}
\caption{Nonlinear systems and nonlinear feedback. Bigger initial condition.}
\label{Fig:CT.SF=NN.fic4}
\end{figure}
\begin{figure}[ht]
\centering
\subfigure
{\includegraphics[width=0.325\textwidth]{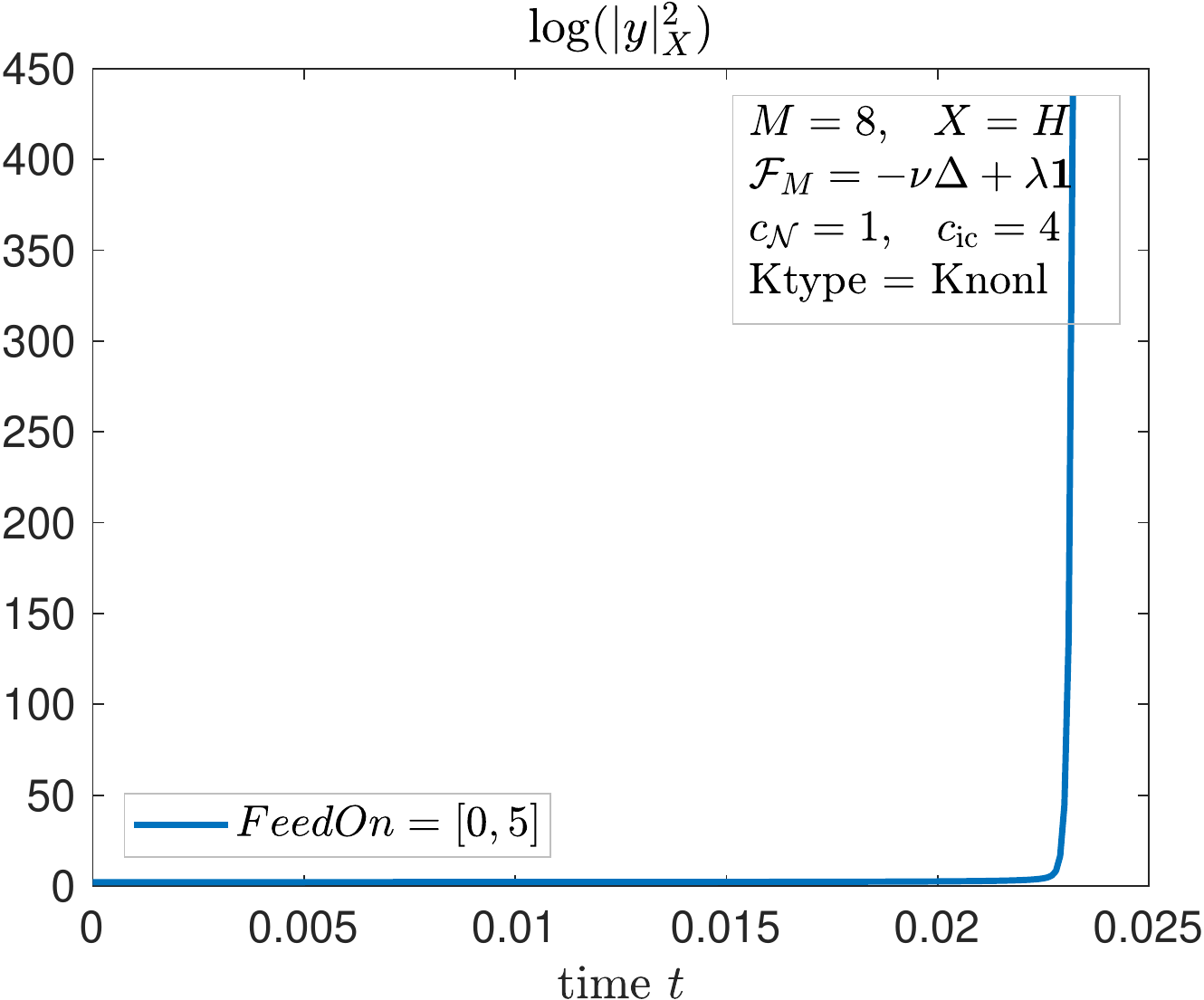}}
\subfigure
{\includegraphics[width=0.315\textwidth]{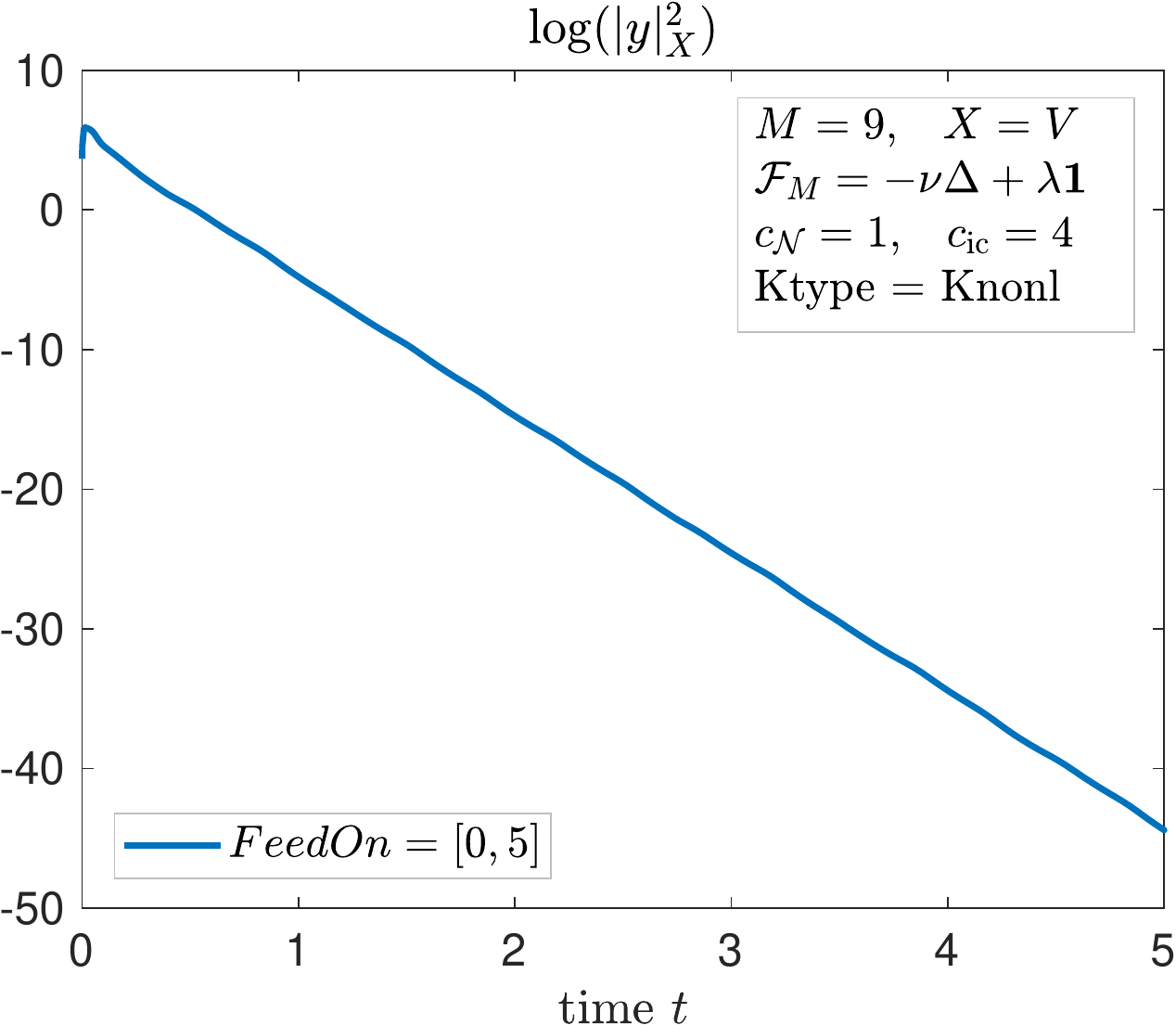}}
\subfigure
{\includegraphics[width=0.315\textwidth]{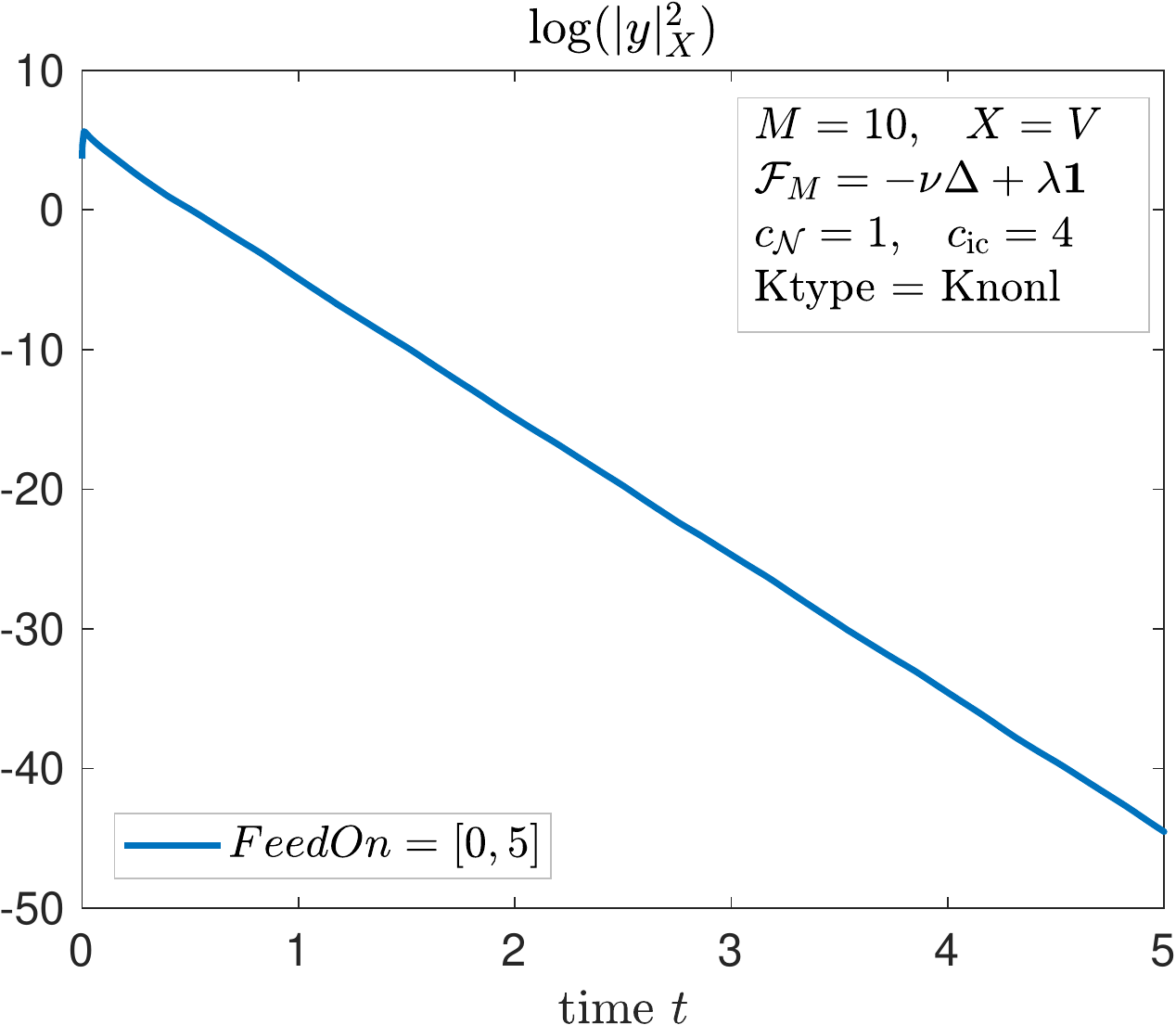}}\\
\subfigure
{\includegraphics[width=0.327\textwidth]{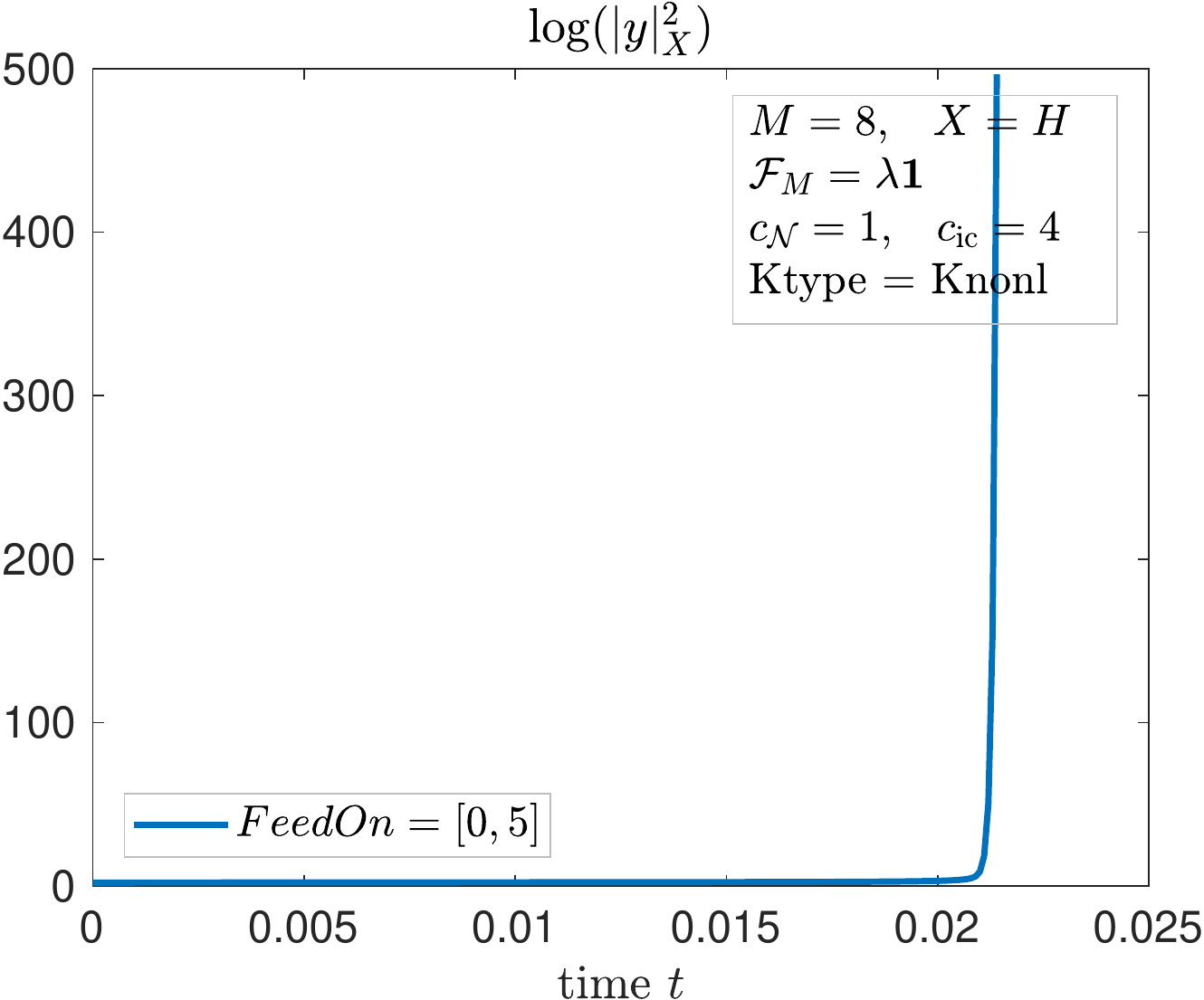}}
\subfigure
{\includegraphics[width=0.327\textwidth]{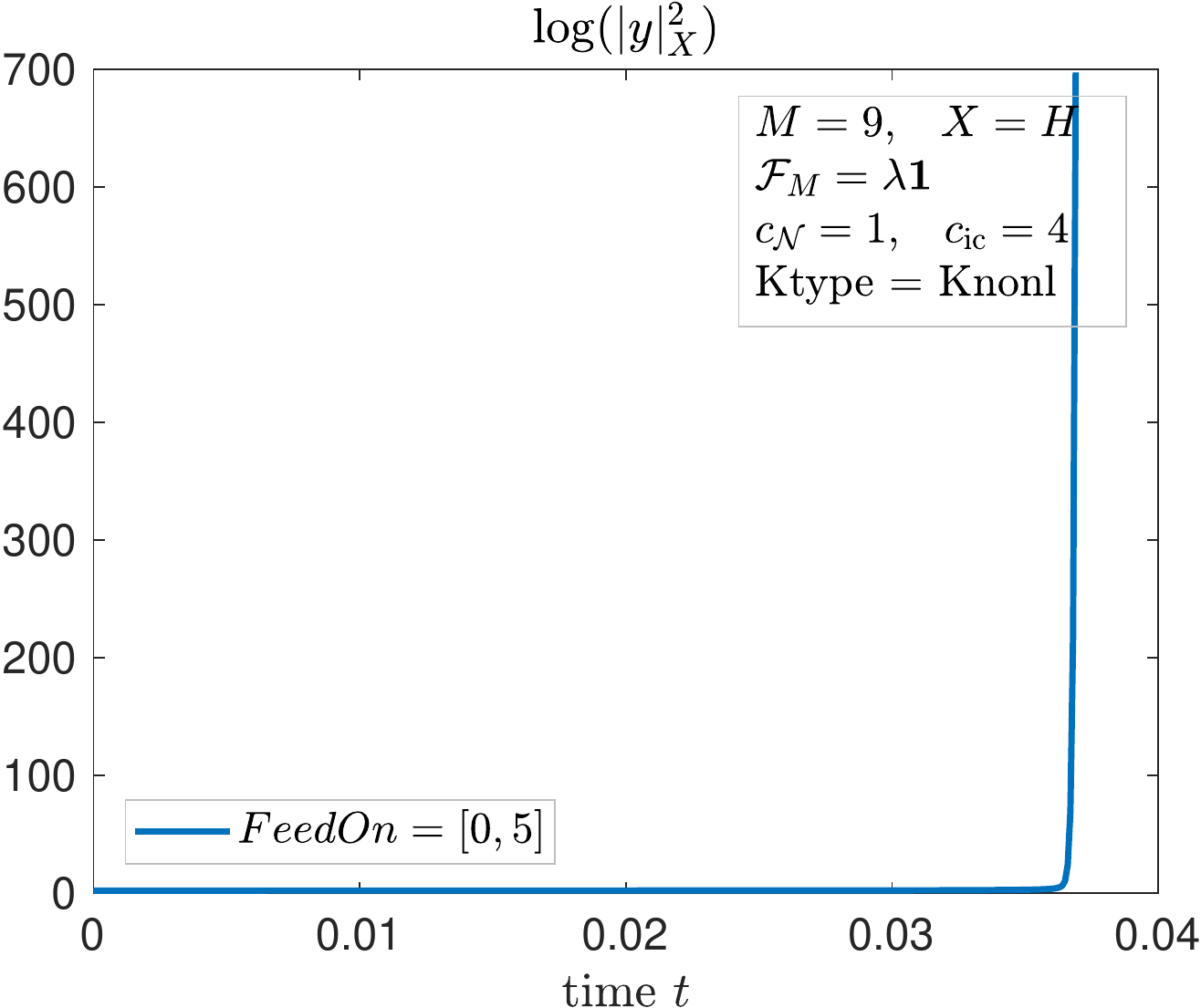}}
\subfigure
{\includegraphics[width=0.31\textwidth]{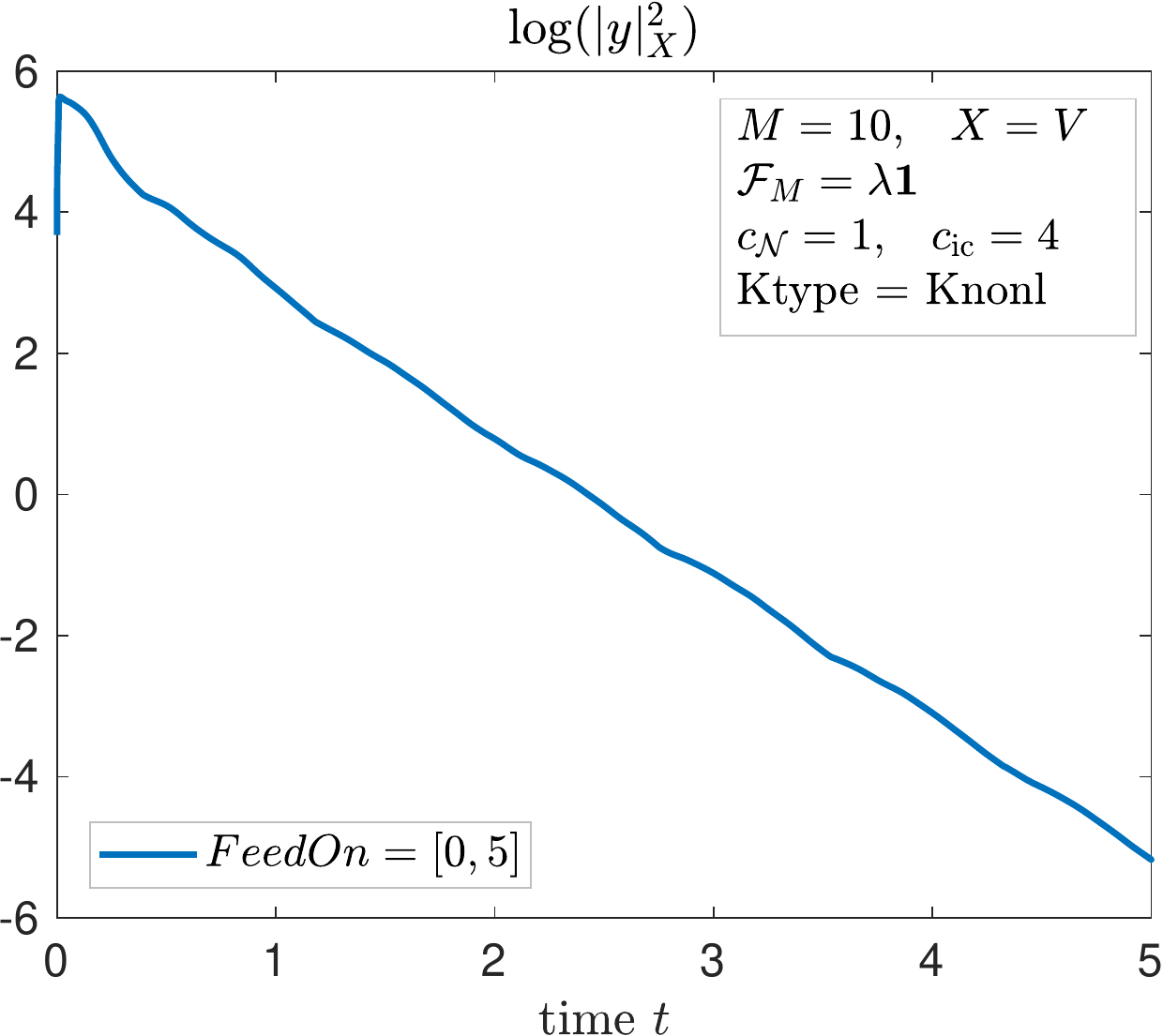}}
\caption{Nonlinear systems and nonlinear feedback. Increasing the number of actuators.}
\label{Fig:CT.SF=NN.fic4.M789}
\end{figure}

\begin{figure}[ht]
\centering
\subfigure
{\includegraphics[width=0.315\textwidth]{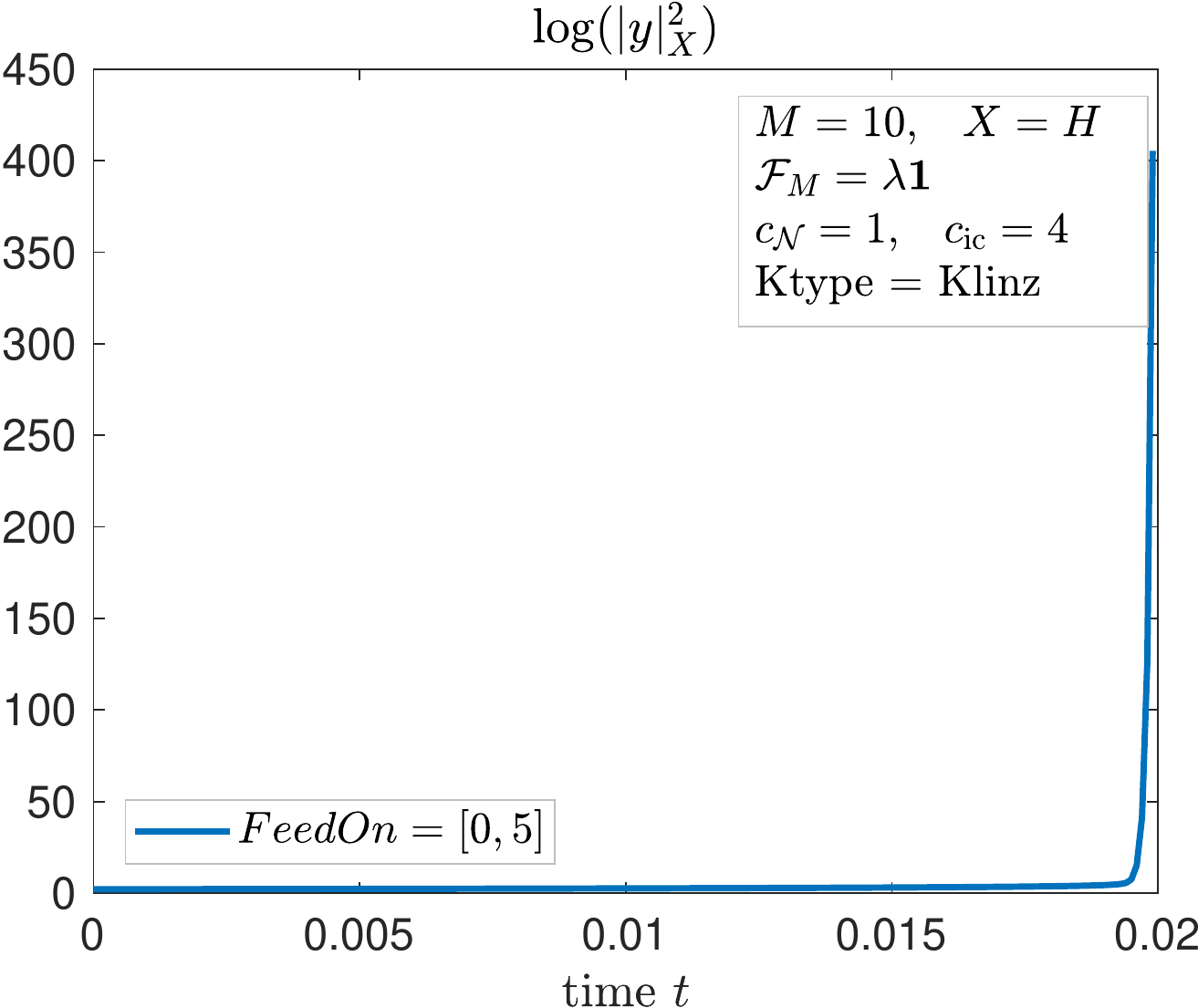}}
\subfigure
{\includegraphics[width=0.315\textwidth]{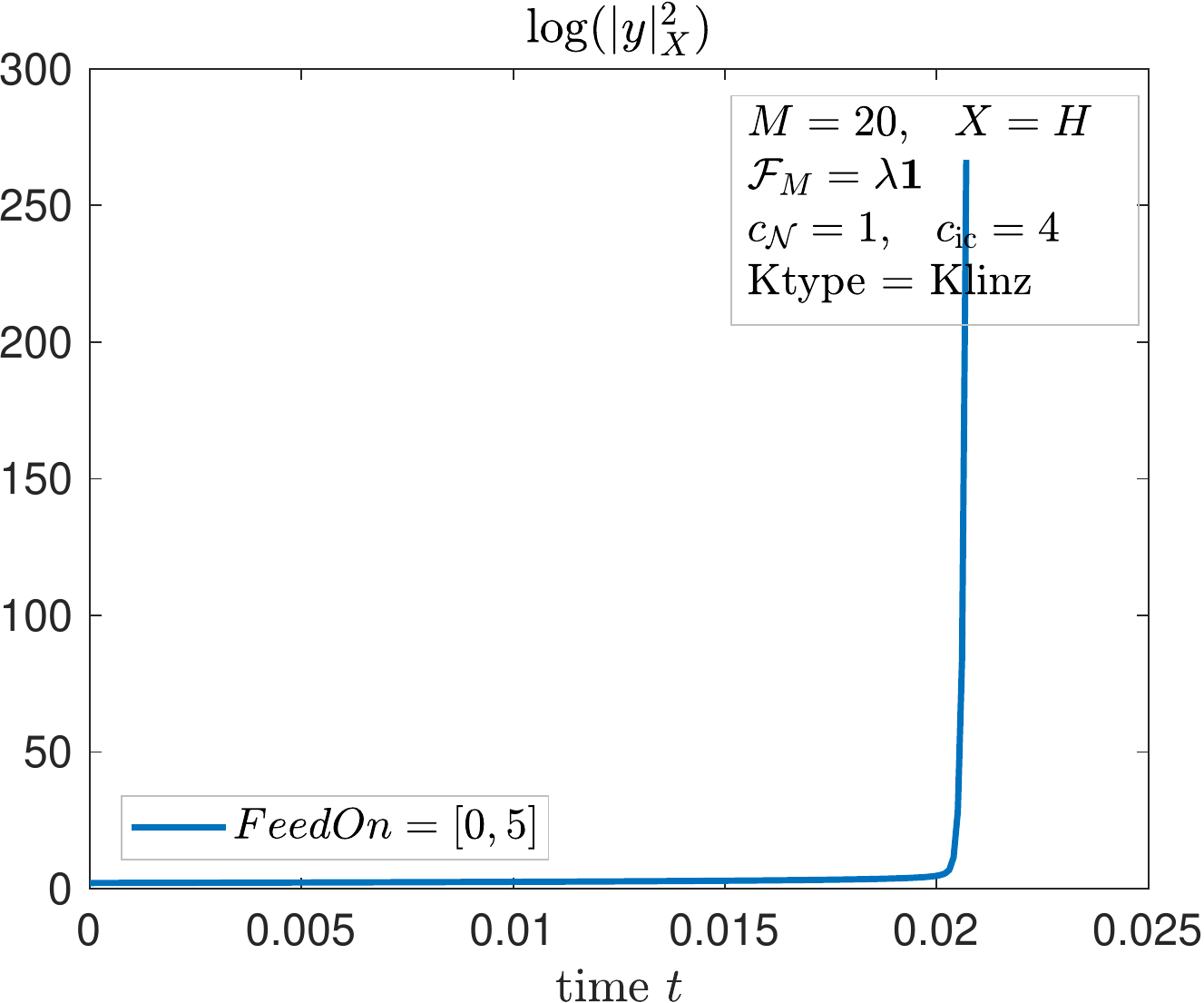}}
\subfigure
{\includegraphics[width=0.315\textwidth]{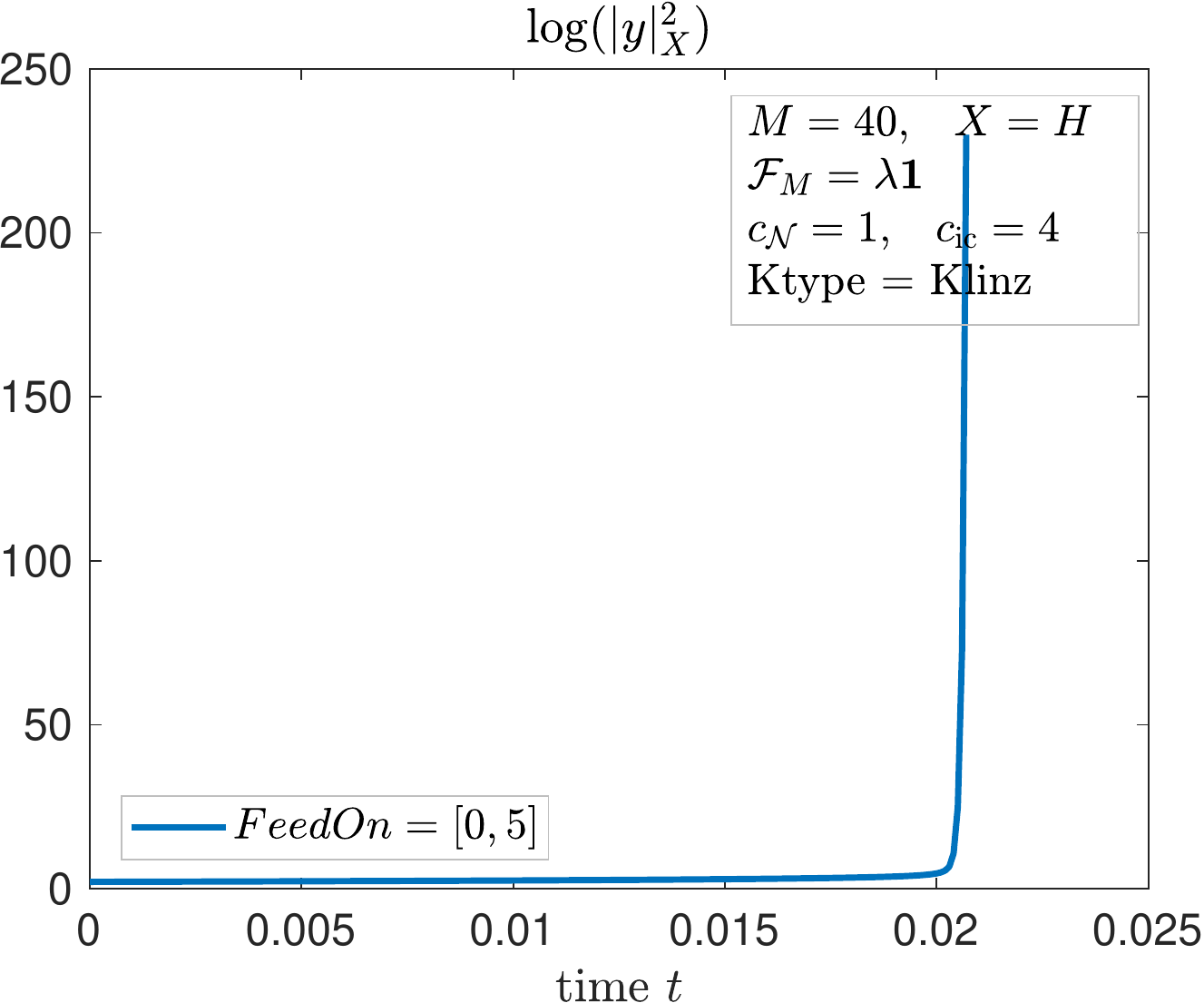}}\\
\subfigure
{\includegraphics[width=0.315\textwidth]{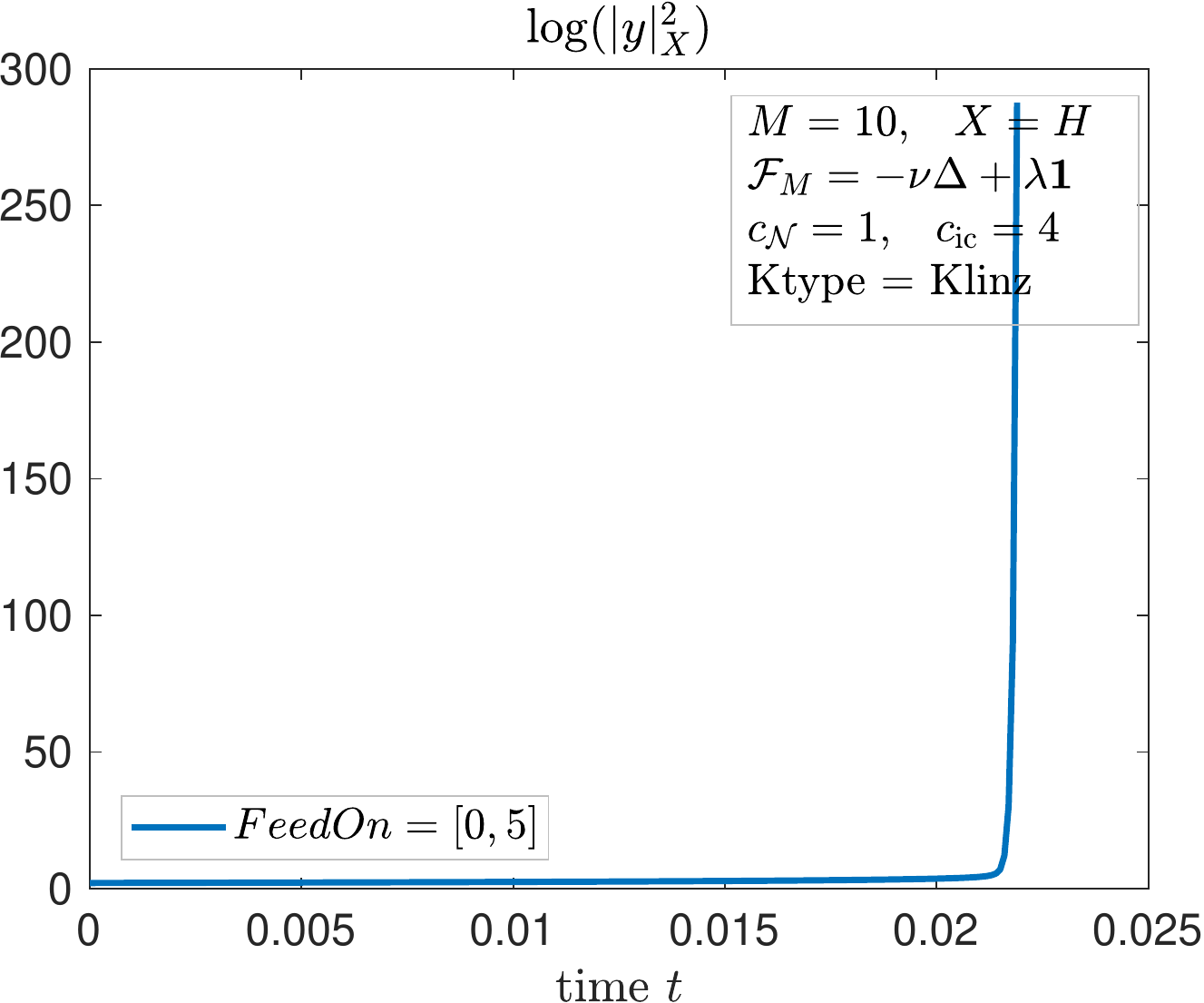}}
\subfigure
{\includegraphics[width=0.315\textwidth]{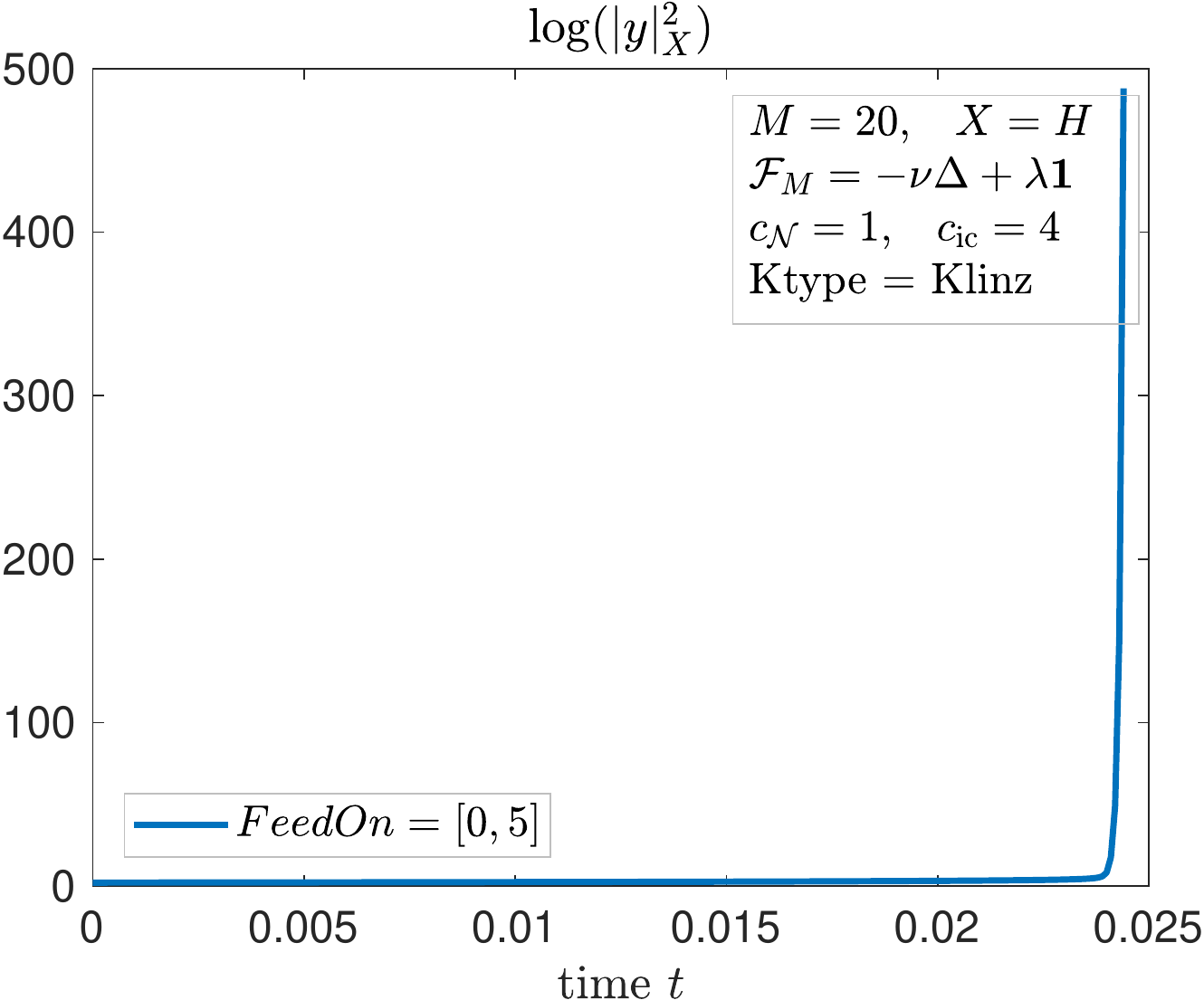}}
\subfigure
{\includegraphics[width=0.315\textwidth]{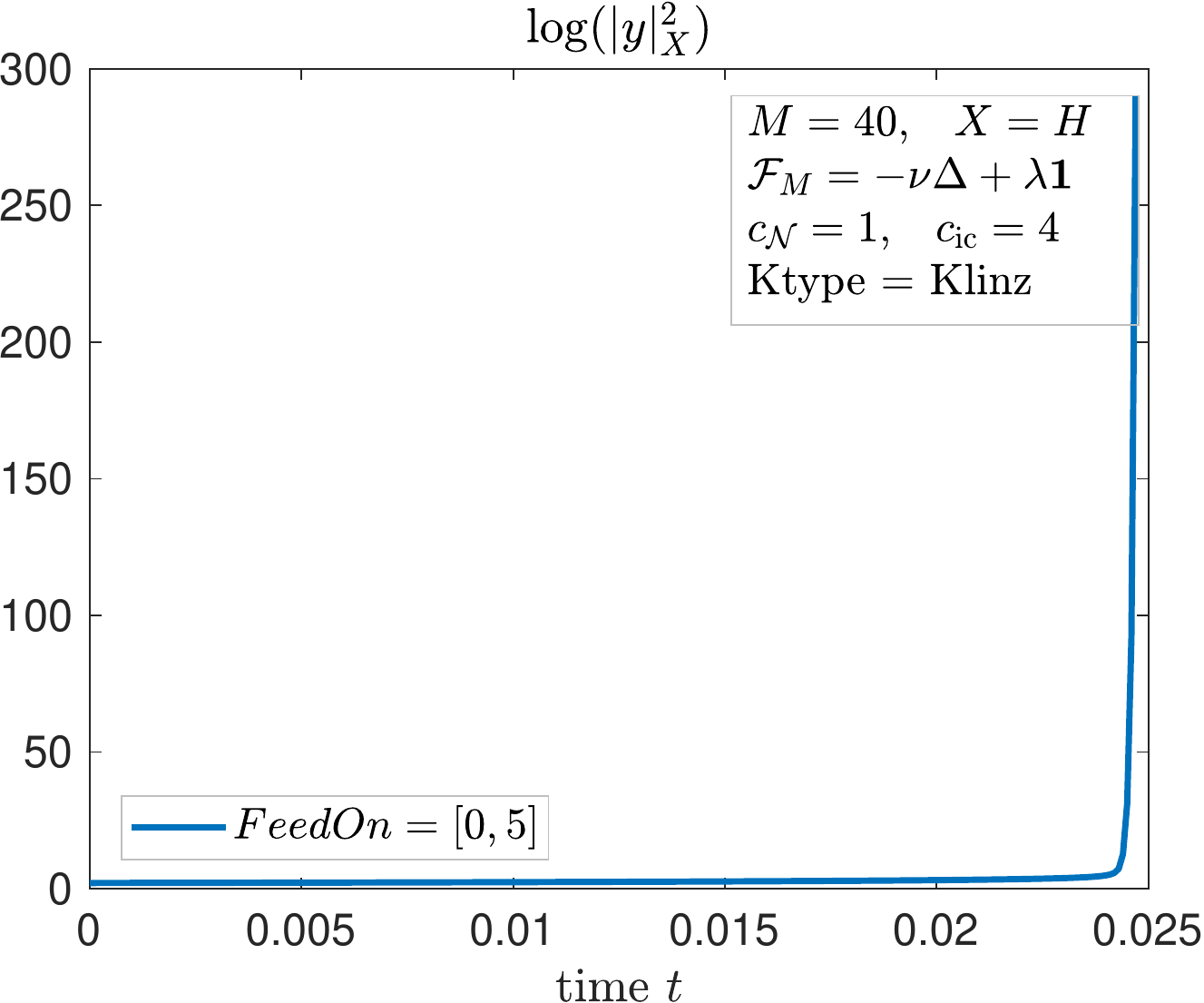}}
\caption{Nonlinear systems and linear feedback. Increasing the number of actuators.}
\label{Fig:CT.SF=NL102040}
\end{figure}

\begin{remark}
Note that when~$q(0)\coloneqq P_{E_{\M_\sigma}}y(0)=0$, the feedbacks correponding to $\FF_{\M}=\lambda\Id$ and
to~$\FF_{\M}=-\nu\Delta+\lambda\Id$ do coincide, because necessarily the solution~$q$ of~$\dot q=\FF_{\M}q$
vanishes in both cases. In particular, the corresponding closed-loop solutions must coincide. This is observed in
Figure~\ref{Fig:CT.SF=NN.q0=0}, where we have taken the initial condition~$y(0)=\sin(8\pi x)$. Note that, with~$M=7$, $E_{\M_\sigma}=E_\M=\linspan\{\sin(i\pi x)\mid i\in\{1,2,\dots,7\}\}$ and thus
$q(0)=P_{E_\M}\sin(8\pi x)=0$.
\end{remark}
\begin{figure}[ht]
\centering
\subfigure
{\includegraphics[width=0.325\textwidth]{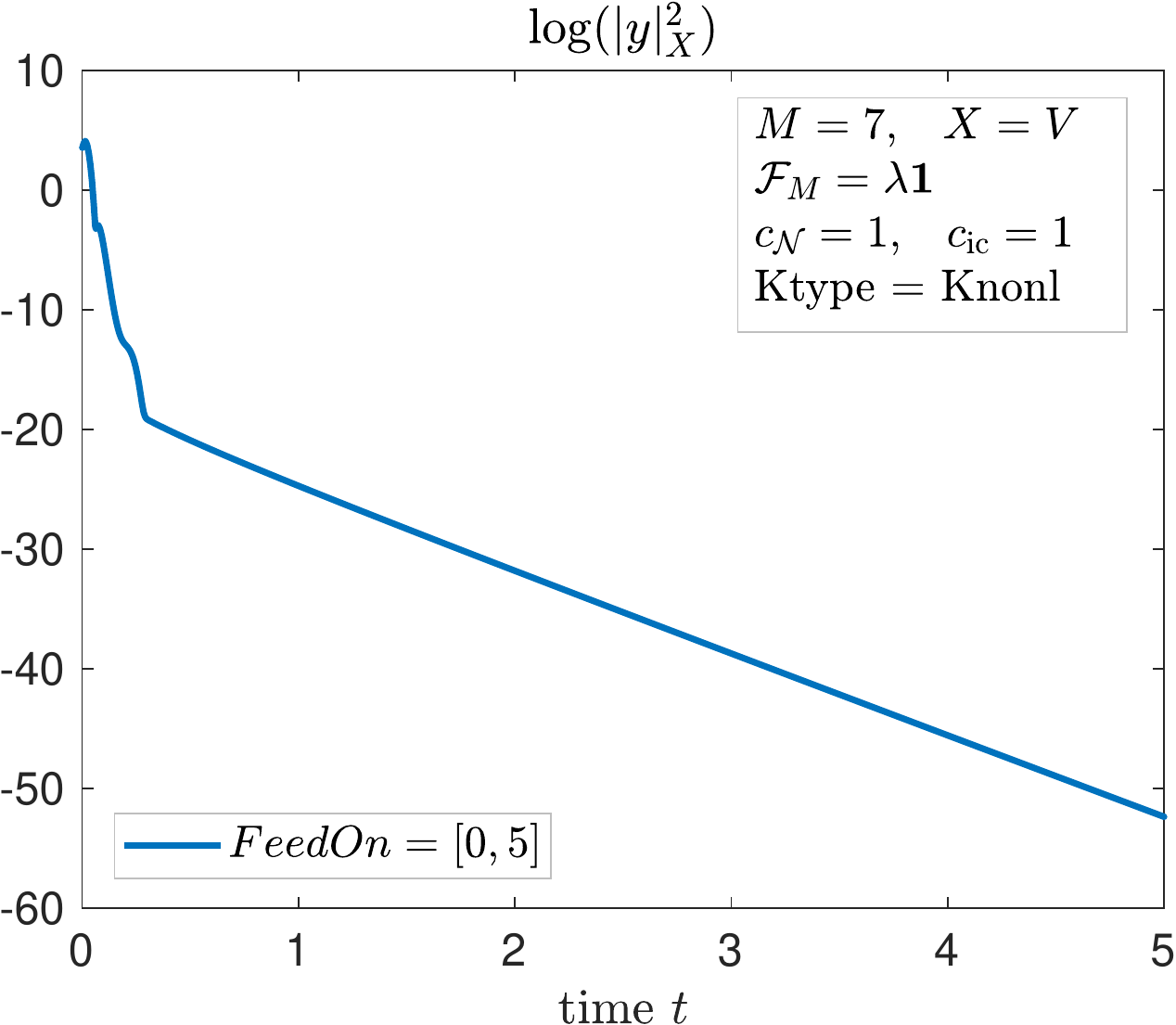}}
\qquad\subfigure
{\includegraphics[width=0.325\textwidth]{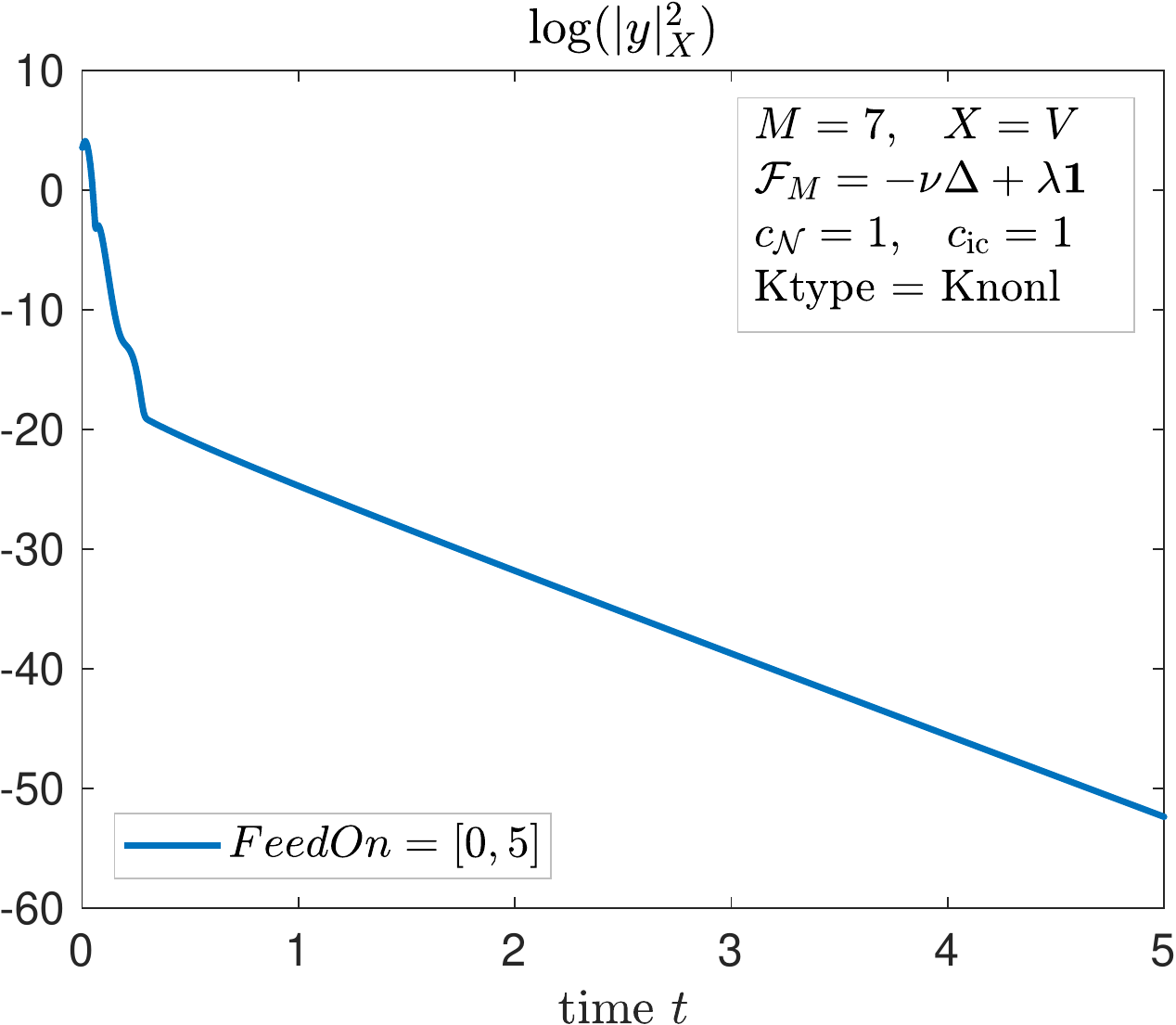}}
\caption{Nonlinear systems and nonlinear feedback.  $y(0)=c_{\rm ic}\sin(8\pi x)\in E_{\M_\sigma}^\perp$.}
\label{Fig:CT.SF=NN.q0=0}
\end{figure}

\begin{remark} We have taken our $M$~actuators with centers location~$c\in[0,1]^M$ as in~\eqref{Act-mxe}, which guarantee that the
norm~$\norm{P_{U_M}^{E_\M}}{\LL(H)}\le C_P$ remains bounded as~$M$ increases, with~$C_P$ independent of~$M$.
Since the number~$M$ of actuators needed to stabilize the system increase with~$C_P$, see~\eqref{depM_Vnorm},
it would be interesting to know the/an optimal location~$c=c^{\rm opt}\in[0,1]^M$ for the $M$~actuators
minimizing~$\norm{P_{U_M}^{E_\M}(c)}{\LL(H)}$.
We would like to refer
to~\cite{Morris11,MorrisYang16,KaliseKunischSturm18,PrivatTrelatZuazua17,MunchPedregalPeriago09}
for works related to finding a/the placement (and/or shape) of actuators, though
the functional to be minimized in those works is not~$\norm{P_{U_M}^{E_\M}(c)}{\LL(H)}$.
\end{remark}

\appendix
\section*{Appendix}\normalsize
\setcounter{section}{1}
\setcounter{theorem}{0} \setcounter{equation}{0}
\numberwithin{equation}{section}

\subsection{Proof of Proposition~\ref{P:NN}}\label{Apx:proofP:NN}

We recall the Young inequality~\cite{Young12} as follows: for all~$(a,b,\gamma_0)\in\R_0^3$ and all~$s>1$, 
we have
 \begin{align*}
 ab=(\gamma_0a)(\gamma_0^{-1}b)
 \le\fractx{1}{s}\gamma_0^{s}a^s+\fractx{s-1}{s}\gamma_0^{-\frac{s}{(s-1)}}b^\frac{s}{s-1},
  \end{align*}
 Note that~$r\coloneqq\frac{s}{s-1}$ satisfies $\frac{1}{s}+\frac{1}{r}=1$. In particular, $ab\le a^s+b^\frac{s}{s-1}$.
Assumption~\ref{A:NN} 
implies that
\begin{align}
 &\quad2\Bigl( P_{E_{\M_\sigma}^\perp}^{U_{\#\M_\sigma}}\left(\NN(t,y_1)-\NN(t,y_2)\right),A(y_1-y_2)\Bigr)_{H}\notag\\
 &\le 2\norm{P_{E_{\M_\sigma}^\perp}^{U_{\#\M_\sigma}}}{\LL(H)}C_\NN\textstyle\sum\limits_{j=1}^{n}
  \left( \norm{y_1}{V}^{\zeta_{1j}}\norm{y_1}{\D(A)}^{\zeta_{2j}}+\norm{y_2}{V}^{\zeta_{1j}}\norm{y_2}{\D(A)}^{\zeta_{2j}}\right)
   \norm{y_1-y_2}{V}^{\delta_{1j}}\norm{y_1- y_2}{\D(A)}^{1+\delta_{2j}}\label{FQAQ1a}
    \end{align}
  and the Young inequality gives us  for all~$\gamma_0>0$, writing for
  simplicity~$\|P\|_\LL\coloneqq\norm{P_{E_{\M_\sigma}^\perp}^{U_{\#\M_\sigma}}}{\LL(H)}$,
 \begin{align*} 
 &2\Bigl( P_{E_{\M_\sigma}^\perp}^{U_{\#\M_\sigma}}\left(\NN(t,y_1)-\NN(t,y_2)\right),A(y_1-y_2)\Bigr)_{H}
 \le \textstyle\sum\limits_{j=1}^{n}\fractx{(1+\delta_{2j})}{2}\gamma_0^{\frac{2}{1+\delta_{2j}}}\norm{y_1- y_2}{\D(A)}^{2}\notag\\
 &\hspace*{3em}+\textstyle\sum\limits_{j=1}^{n}\fractx{(1-\delta_{2j})}{2}
 \left(\fractx{2C_\NN}{\gamma_0}\|P\|_\LL
 \norm{y_1-y_2}{V}^{\delta_{1j}}\right)^{\frac{2}{1-\delta_{2j}}}
 \left(
  \norm{y_1}{V}^{\zeta_{1j}}\norm{y_1}{\D(A)}^{\zeta_{2j}}+\norm{y_2}{V}^{\zeta_{1j}}\norm{y_2}{\D(A)}^{\zeta_{2j}}
  \right)^{\frac{2}{1-\delta_{2j}}}.
   \end{align*}

 From~\cite[Prop.~2.6]{PhanRod17} we have that $(a_1+a_2)^s\le2^{\norm{s-1}{}}(a_1^s+a_2^s)$
 for~$(a_1,a_2,s)\in[0,+\infty)$, which implies
\begin{align}
 &2\Bigl( P_{E_{\M_\sigma}^\perp}^{U_{\#\M_\sigma}}\left(\NN(t,y_1)-\NN(t,y_2)\right),A(y_1-y_2)\Bigr)_{H} 
  \le \textstyle\sum\limits_{j=1}^{n}\fractx{(1+\delta_{2j})}{2}\gamma_0^{\frac{2}{1+\delta_{2j}}}
  \norm{y_1- y_2}{\D(A)}^{2}\label{FQAQ1}\\
 &\hspace*{1em}+\textstyle\sum\limits_{j=1}^{n}(1-\delta_{2j})
 2^{\frac{2(1+\delta_{2j})}{1-\delta_{2j}}}\left(\fractx{C_\NN}{\gamma_0}\|P\|_\LL
 \right)^{\frac{2}{1-\delta_{2j}}}
 \left( \norm{y_1}{V}^\frac{2\zeta_{1j}}{1-\delta_{2j}} \norm{y_1}{\D(A)}^\frac{2\zeta_{2j}}{1-\delta_{2j}}
 +\norm{y_2}{V}^\frac{2\zeta_{1j}}{1-\delta_{2j}} \norm{y_2}{\D(A)}^\frac{2\zeta_{2j}}{1-\delta_{2j}}
 \right)\norm{y_1-y_2}{V}^{\frac{2\delta_{1j}}{1-\delta_{2j}}}.\notag
 \end{align}

 Observe that if we fix an arbitrary~$\widehat\gamma_0>0$ and set, in~\eqref{FQAQ1},
  \[
 \gamma_0= \gamma_{0j}\coloneqq\left(\fractx{2}{1+\delta_{2j}}\fractx{\widehat\gamma_0}{n}\right)^\frac{1+\delta_{2j}}{2} 
 \quad\Longleftrightarrow\quad
 \fractx{\widehat\gamma_0}{n}=\fractx{1+\delta_{2j}}{2}\gamma_{0j}^\frac{2}{1+\delta_{2j}},
\]
then, since~$\delta_{2j}<1$, we obtain
\[
 \left(\fractx1{\gamma_{0j}}\right)^{\frac{2}{1-\delta_{2j}}}
 =\left( \fractx{n(1+\delta_{2j})}{2\widehat\gamma_0}
 \right)^{\frac{1+\delta_{2j}}{1-\delta_{2j}}}
 <\left( \fractx{n}{\widehat\gamma_0} \right)^{\frac{1+\delta_{2j}}{1-\delta_{2j}}}
  <1+\left( \fractx{n}{\widehat\gamma_0} \right)^{\frac{1+\|\delta_2\|}{1-\|\delta_2\|}}.
\]
with~$\|\delta_2\|\coloneqq\textstyle\max\limits_{1\le j\le n}|\delta_{2j}|$.
Hence, we arrive at
\begin{align}
 &2\Bigl( P_{E_{\M_\sigma}^\perp}^{U_{\#\M_\sigma}}\left(\NN(t,y_1)-\NN(t,y_2)\right),A(y_1-y_2)\Bigr)_{H} 
  \le \widehat\gamma_0 \norm{y_1- y_2}{\D(A)}^{2}\notag\\
  &\hspace*{0em}+\!\left(\!1+\widehat\gamma_0^{-\frac{1+\|\delta_2\|}{1-\|\delta_2\|} }\!\right)
  \!\ovlineC{n,\|\delta_{2}\|,\frac{1}{1-\|\delta_{2}\|},C_\NN,\|P\|_\LL}
 \textstyle\sum\limits_{j=1}^n\!\left(\! \norm{y_1}{V}^\frac{2\zeta_{1j}}{1-\delta_{2j}}
 \!\norm{y_1}{\D(A)}^\frac{2\zeta_{2j}}{1-\delta_{2j}}
 +\norm{y_2}{V}^\frac{2\zeta_{1j}}{1-\delta_{2j}}\! \norm{y_2}{\D(A)}^\frac{2\zeta_{2j}}{1-\delta_{2j}}
 \!\right)\!\norm{y_1-y_2}{V}^{\frac{2\delta_{1j}}{1-\delta_{2j}}}\!.\notag
 \end{align}

 In the particular case~$y_1=q+Q$ and~$y_2=q$ with~$(q,Q)\in E_{\M_\sigma}\times E_{\M_\sigma}^\perp$,
 estimate~\eqref{FQAQ1a} also gives us
 \begin{align}
 &\quad2\Bigl( P_{E_{\M_\sigma}^\perp}^{U_{\#\M_\sigma}}\left(\NN(t,q+Q)-\NN(t,q)\right),AQ\Bigr)_{H}\notag\\
 &\le 2\dnorm{P}{\LL}C_\NN\textstyle\sum\limits_{j=1}^{n}
  \left( \norm{q+Q}{V}^{\zeta_{1j}}\norm{q+Q}{\D(A)}^{\zeta_{2j}}+\norm{q}{V}^{\zeta_{1j}}\norm{q}{\D(A)}^{\zeta_{2j}}\right)
   \norm{Q}{V}^{\delta_{1j}}\norm{Q}{\D(A)}^{1+\delta_{2j}}\notag\\
   &\le 2(1+2^{\|\zeta-1\|})\dnorm{P}{\LL}C_\NN\textstyle\sum\limits_{j=1}^{n}
  \left( \norm{q}{V}^{\zeta_{1j}}+\norm{Q}{V}^{\zeta_{1j}}\right)\left(\norm{q}{\D(A)}^{\zeta_{2j}}+\norm{Q}{\D(A)}^{\zeta_{2j}}\right)
   \norm{Q}{V}^{\delta_{1j}}\norm{Q}{\D(A)}^{1+\delta_{2j}}\notag\\
     &\le 2(1+2^{\|\zeta-1\|})\dnorm{P}{\LL}C_\NN\textstyle\sum\limits_{j=1}^{n}
  \left( \norm{q}{V}^{\zeta_{1j}}+\norm{Q}{V}^{\zeta_{1j}}\right)
   \norm{Q}{V}^{\delta_{1j}}\norm{Q}{\D(A)}^{1+\delta_{2j}+\zeta_{2j}}\notag\\
     &\hspace*{3em}+ 2(1+2^{\|\zeta-1\|})\dnorm{P}{\LL}C_\NN\textstyle\sum\limits_{j=1}^{n}
  \left( \norm{q}{V}^{\zeta_{1j}}+\norm{Q}{V}^{\zeta_{1j}}\right)\norm{q}{\D(A)}^{\zeta_{2j}}
   \norm{Q}{V}^{\delta_{1j}}\norm{Q}{\D(A)}^{1+\delta_{2j}}\notag\\
    &\le 2^{2+\|\zeta-1\|}\dnorm{P}{\LL}C_\NN\textstyle\sum\limits_{j=1}^{n}
  \left( \norm{q}{V}^{\zeta_{1j}}\norm{Q}{V}^{\delta_{1j}}+\norm{Q}{V}^{\zeta_{1j}+\delta_{1j}}\right)
   \norm{Q}{\D(A)}^{1+\delta_{2j}+\zeta_{2j}}\notag\\
     &\hspace*{3em}+ 2^{2+\|\zeta-1\|}\dnorm{P}{\LL}C_\NN\textstyle\sum\limits_{j=1}^{n}
  \left( \norm{q}{V}^{\zeta_{1j}}\norm{q}{\D(A)}^{\zeta_{2j}}\norm{Q}{V}^{\delta_{1j}}
  +\norm{q}{\D(A)}^{\zeta_{2j}}\norm{Q}{V}^{\zeta_{1j}+\delta_{1j}}\right)
   \norm{Q}{\D(A)}^{1+\delta_{2j}}
 \notag
    \end{align}
 with~$\|\zeta-1\|\coloneqq\max\{|\zeta_{k,j}-1|\mid 1\le j\le n, 1\le k\le2\}$.
  By the Young inequality, with~$\gamma_0>0$ and~$\widetilde\gamma_0>0$,
 \begin{align}
 &\quad2\Bigl( P_{E_{\M_\sigma}^\perp}^{U_{\#\M_\sigma}}\left(\NN(t,q+Q)-\NN(t,q)\right),AQ\Bigr)_{H}
 \le\textstyle\sum\limits_{j=1}^{n}\left(\fractx{1+\delta_{2j}
 +\zeta_{2j}}{2}\widetilde\gamma_0^{\fractx{2}{1+\delta_{2j}+\zeta_{2j}}}
 +\fractx{1+\delta_{2j}}{2}\gamma_0^{\fractx{2}{1+\delta_{2j}}}\right)\norm{Q}{\D(A)}^2\notag\\
    &\hspace*{3em}+ \textstyle\sum\limits_{j=1}^{n}\left(2^{2+\|\zeta-1\|}
    \dnorm{P}{\LL}\fractx{C_\NN}{\widetilde\gamma_0}\right)^{\frac{2}{1-\delta_{2j}-\zeta_{2j}}}
  \left( \norm{q}{V}^{\zeta_{1j}}
  \norm{Q}{V}^{\delta_{1j}}+\norm{Q}{V}^{\zeta_{1j}+\delta_{1j}}\right)^{\frac{2}{1-\delta_{2j}-\zeta_{2j}}}
   \notag\\
     &\hspace*{3em}+ \textstyle\sum\limits_{j=1}^{n}\left(2^{2+\|\zeta-1\|}
     \dnorm{P}{\LL}\fractx{C_\NN}{\gamma_0}\right)^{\frac{2}{1-\delta_{2j}}}
   \left( \norm{q}{V}^{\zeta_{1j}}\norm{q}{\D(A)}^{\zeta_{2j}}\norm{Q}{V}^{\delta_{1j}}
  +\norm{q}{\D(A)}^{\zeta_{2j}}\norm{Q}{V}^{\zeta_{1j}+\delta_{1j}}\right)^{\frac{2}{1-\delta_{2j}}}
 \label{FQAQ2}
    \end{align}
    
Fixing an arbitrary~$\widehat\gamma_0>0$ and setting, in~\eqref{FQAQ2},
  \[
 \gamma_0= \gamma_{0j}\coloneqq\left(\fractx{2}{1+\delta_{2j}}\fractx{\widehat\gamma_0}{2n}\right)^\frac{1+\delta_{2j}}{2} 
 \quad\Longleftrightarrow\quad
 \fractx{\widehat\gamma_0}{2n}=\fractx{1+\delta_{2j}}{2}\gamma_{0j}^\frac{2}{1+\delta_{2j}},
\]
then, since~$\delta_{2j}<1$, we obtain
\[
 \left(\fractx1{\gamma_{0j}}\right)^{\frac{2}{1-\delta_{2j}}}
 =\left( \fractx{2n(1+\delta_{2j})}{2\widehat\gamma_0}
 \right)^{\frac{1+\delta_{2j}}{1-\delta_{2j}}}
 <\left( \fractx{2n}{\widehat\gamma_0} \right)^{\frac{1+\delta_{2j}}{1-\delta_{2j}}}
 <1+\left( \fractx{2n}{\widehat\gamma_0} \right)^{\frac{1+\|\delta_{2}\|}{1-\|\delta_{2}+\zeta_{2}\|}}
 \]
with~$\|\zeta_2+\delta_2\|\coloneqq\textstyle\max\limits_{1\le j\le n}|\delta_{2j}+\zeta_{2j}|$.

Also, with~$ \widetilde\gamma_0= \widetilde\gamma_{0j}\coloneqq\left(\fractx{2}{1+\delta_{2j}+\zeta_{2j}}
 \fractx{\widehat\gamma_0}{2n}\right)^\frac{1+\delta_{2j}+\zeta_{2j}}{2}
 \quad\Longleftrightarrow\quad
 \fractx{\widehat\gamma_0}{2n}=\fractx{1+\delta_{2j}+\zeta_{2j}}{2}\widetilde\gamma_{0j}^\frac{2}{1+\delta_{2j}+\zeta_{2j}}
$,
we find $\left(\fractx1{\widetilde\gamma_{0j}}\right)^{\frac{2}{1-\delta_{2j}-\zeta_{2j}}}
  <1+\left( \fractx{2n}{\widehat\gamma_0} \right)^{\frac{1+\|\delta_2\|}{1-\|\zeta_2+\delta_2\|}}$.
 Now we show that, from~\eqref{FQAQ2} we can obtain
\begin{align}
 &2\Bigl(  P_{E_{\M_\sigma}^\perp}^{U_{\#\M_\sigma}}\left(\NN(t,q+Q)-\NN(t,q)\right),AQ\Bigr)_{H} 
  \le \widehat\gamma_0 \norm{Q}{\D(A)}^{2}\notag\\
  &\hspace*{2em}+\left(1+\widehat\gamma_0^{-\frac{1+\|\delta_2\|}{1-\|\zeta_2+\delta_2\|} }\right)
  \overline{C}
 \left( 1 +\norm{q}{V}^{2\|\frac{\zeta_1+\delta_1}{1-\zeta_2-\delta_2}\|}\right)
 \left(1+\norm{q}{\D(A)}^{\|\frac{\zeta_2}{1-\delta_2}\|}\right)
 \left( 1+\norm{Q}{V}^{2\|\frac{\zeta_1+\delta_1}{1-\zeta_2-\delta_2}\|-2} \right)
 \norm{Q}{V}^{2}.\notag
\end{align}  
 with the following constants:~$\|\frac{\zeta_1+\delta_1}{1-\zeta_2-\delta_2}\|\coloneqq\textstyle\max\limits_{1\le j\le n}
 |\frac{\delta_{1j}+\zeta_{1j}}{1-\delta_{2j}-\zeta_{2j}}|$, ~$\|\frac{\zeta_2}{1-\delta_2}\|
 \coloneqq\textstyle\max\limits_{1\le j\le n}
 |\frac{\zeta_{2j}}{1-\delta_{2j}}|$,
 and~$\overline{C}=\ovlineC{n,\|\zeta_1\|,\|\zeta_2\|,\frac{1}{1-\|\delta_{2}\|},
 \frac{1}{1-\|\delta_{2}+\zeta_2\|},C_\NN,\|P\|_\LL}$.
 Indeed, we can use the inequalities
 \begin{align*}
  &\left( \norm{q}{V}^{\zeta_{1j}}
  \norm{Q}{V}^{\delta_{1j}}+\norm{Q}{V}^{\zeta_{1j}+\delta_{1j}}\right)^{\frac{2}{1-\delta_{2j}-\zeta_{2j}}}
  \le C_0\left( \norm{q}{V}^\frac{2(\zeta_{1j}+\delta_{1j})}{1-\delta_{2j}-\zeta_{2j}}
  +\norm{Q}{V}^\frac{2(\zeta_{1j}+\delta_{1j})}{1-\delta_{2j}-\zeta_{2j}}\right),
  \\
  &\left( \!\norm{q}{V}^{\zeta_{1j}}\!\norm{q}{\D(A)}^{\zeta_{2j}}\!\norm{Q}{V}^{\delta_{1j}}
  +\norm{q}{\D(A)}^{\zeta_{2j}}\!\norm{Q}{V}^{\zeta_{1j}+\delta_{1j}}\right)^{\frac{2}{1-\delta_{2j}}}
  \le C_1\!\left(\! \norm{q}{V}^\frac{2\zeta_{1j}}{1-\delta_{2j}}\norm{q}{\D(A)}^\frac{2\zeta_{2j}}{1-\delta_{2j}}
  \norm{Q}{V}^\frac{2\delta_{1j}}{1-\delta_{2j}}
  +\norm{q}{\D(A)}^\frac{2\zeta_{2j}}{1-\delta_{2j}}
  \norm{Q}{V}^\frac{2(\zeta_{1j}+\delta_{1j})}{1-\delta_{2j}}\!\right)\\
  &\hspace*{5em}\le C_2\left( 1+\norm{q}{V}^{\frac{2\zeta_{1j}}{1-\delta_{2j}}}\right)
  \norm{q}{\D(A)}^\frac{2\zeta_{2j}}{1-\delta_{2j}}
  \left(\norm{Q}{V}^{\frac{2\delta_{1j}}{1-\delta_{2j}}-2}
  + \norm{Q}{V}^{\frac{2(\zeta_{1j}+\delta_{1j})}{1-\delta_{2j}}-2}\right)\norm{Q}{V}^{2},
 \end{align*}
where the constants~$C_k$, $k\in\{1,2,3\}$ are of the
form~$C_k=\ovlineC{\|\zeta_1+\delta_1\|,\frac{1}{1-\|\delta_{2}+\zeta_2\|},\frac{1}{1-\|\delta_{2}\|}}$.
 \qed

\subsection{Proof of Proposition~\ref{P:odeh0}}\label{Apx:proofP:odeh0}
Observe that, since~$p\ge0$, the function
$w\mapsto \norm{w}{\R}^{p}w$ is locally Lipschitz. Therefore, the solutions
of~\eqref{dyn-odew}, do exist and are unique, in a small time interval, say for
time~$t\in[0,\tau)$ with~$\tau$ small. When~$w_0=0$ the solution is the trivial one~$w=0$.
Note that the equilibria of~\eqref{dyn-odewh0}, that is, the solutions of~$\dot w=0$, are given by
$\overline w_1=0$ and~$\overline w_2^\pm
 =\pm\left(\fractx{\breve C_1}{\breve C_2}\right)^\frac1{p}$. 
Furthermore, we observe that~$\dot w<0$ if~$w\in(0,\overline w_2^+)$, which implies that
the solution issued from $w(s)\in(0,\overline w_2^+)$ at time $t=s$, is globally
defined, for all time $t\ge s$, is decreasing, and thus remains in~$(0,\overline w_2^+)$. Note that
\[
-\breve C_1 w \le \dot w \le-\bigl(\breve C_1 -\breve C_2\norm{w_0}{\R}^{p}\bigr) w, \quad\mbox{for}\quad w\in(0,\overline w_2^+).
\]
Therefore we can conclude that~\eqref{est.wh0} holds for~$w_0\in(0,\overline w_2^+)$.
Next we consider the case~$w_0\in(-\overline w_2^+,0)$. Denoting the solution issued from~$w(s)=w_s\in\R$, at
time~$s$, by~$w(t)=S(t,s)(w_{s})$, $t\ge s$, we find~$S(t,s)(w_{s})=-S(t,s)(-w_{s})$,
because with~$w^+(t)\coloneqq S(t,s)(-w_{s})$, we have
\[
 \fractx{\ed}{\ed t}(-w^+)=-\dot w^+=-\left(-\bigl(\breve C_1 -\breve C_2\norm{w^+}{\R}^{p}\bigr)w^+\right)
 =-\bigl(\breve C_1 -\breve C_2\norm{-w^+}{\R}^{p}\bigr)(-w^+),\qquad-w^+(s)=w_s.
\]
The uniqueness of the solution, implies that~$S(t,s)(w_s)=-w^+$. Since~$-w_{0}\in(0,\overline w_2^+)$,
it follows, from above, that~$\norm{w^+}{\R}$
satisfies~\eqref{est.wh0} and, from~$\norm{S(t,0)(w_{0})}{\R}=\norm{w^+(t)}{\R}$, we obtain
that~\eqref{est.wh0} holds for~$w_0\in(-\overline w_2^+,0)$.
\qed

%

\end{document}